\documentclass[final]{siamltex}
\usepackage{latexsym,url,amscd,amsmath,amssymb}
\usepackage{graphics}
\usepackage{graphicx}
\usepackage{epstopdf}
\usepackage[usenames]{color}
\usepackage{comment}
\usepackage[misc]{ifsym}
\usepackage[algo2e,linesnumbered,vlined,ruled]{algorithm2e}

\newtheorem{assumption}{Assumption}

\newtheorem{algorithm}{Algorithm}

\usepackage{longtable}
\usepackage{color}
\def\mc{\multicolumn}

\def\Gnorm#1{ \|#1\|_G }

\def\calW{ {\cal W} }

\def\calB{ {\cal B} }
\def\bu{ {\bf u} }

\def\calL{ {\cal L} }
\def\calQ{ {\cal Q} }

\begin{document}
\title{Inertial primal-dual algorithms for structured convex optimization}

\author{Raymond H. Chan\footnotemark[1] \and Shiqian Ma\footnotemark[2]  \and  Junfeng Yang\footnotemark[3]}
\renewcommand{\thefootnote}{\fnsymbol{footnote}}

\footnotetext[1]{Department of Mathematics, The Chinese University of Hong Kong, Hong Kong (Email: {\tt rchan@math.cuhk.edu.hk}).}

\footnotetext[2]{Department of Systems Engineering and Engineering Management, The Chinese University of Hong Kong, Hong Kong (Email: {\tt sqma@se.cuhk.edu.hk}). This author was supported by Hong Kong Research Grants Council General Research Fund Early Career Scheme (Project ID: CUHK 439513).}

\footnotetext[3]{Corresponding author (Email: {\tt jfyang@nju.edu.cn}). Department of Mathematics, Nanjing University, China.
This author was supported by the Natural Science Foundation of China NSFC-11371192.
The work was done while this author was visiting the Chinese University of Hong Kong.}

\renewcommand{\thefootnote}{\arabic{footnote}}
\date{\today}
\maketitle

\begin{abstract}
The primal-dual algorithm recently proposed by Chambolle \& Pock (abbreviated as CPA) for structured convex optimization is very efficient and popular.
It was shown by Chambolle \& Pock in \cite{CP11} and also by Shefi \& Teboulle in \cite{ST14} that CPA and variants  are closely related
to preconditioned versions of the popular alternating direction method of multipliers (abbreviated as ADM).
In this paper, we further clarify this connection and show that CPAs generate exactly the same sequence of points with the so-called linearized ADM (abbreviated as LADM) applied
to either the primal problem or its Lagrangian dual, depending on different updating orders of the primal and the dual variables in CPAs, as long as the initial points for the LADM
are properly chosen. The dependence on initial points for LADM can be relaxed by focusing on cyclically equivalent forms of the algorithms.
Furthermore, by utilizing the fact that CPAs are applications of a general weighted proximal point method to the mixed variational inequality formulation of the KKT system, where the weighting matrix is  positive definite under a parameter condition, we are able to propose and analyze inertial variants of CPAs. Under certain conditions,  global point-convergence, nonasymptotic $O(1/k)$  and asymptotic $o(1/k)$ convergence rate of the proposed inertial CPAs can be guaranteed, where $k$ denotes the iteration index.
Finally, we demonstrate the profits gained by introducing the inertial extrapolation step via experimental results on compressive image reconstruction based on total variation minimization.
\end{abstract}


\thispagestyle{plain}

\section{Introduction}
In this paper, we consider structured convex optimization problem of the form
\begin{equation}\tag{P$_1$}
\min_x \left\{ f(x) + g(Ax): \; x\in\Re^n\right\},
\end{equation}
where  $f: \Re^n\rightarrow(-\infty,\infty]$ and  $g: \Re^{m}\rightarrow(-\infty,\infty]$ are given extended real-valued closed proper convex functions (not necessarily differentiable), and $A$ is a given real matrix of size $m$-by-$n$.
Here the functions $f$ and $g$ are assumed to be extended real-valued and thus hard constraints on the variables can be hidden in the objective function.
Problems of this kind arise from numerous areas and applications, for example, signal and image reconstruction, compressive sensing and machine learning, to name a few, see, e.g., \cite{CP11,Boyd+11,ST14,YZ11,YY13,EZC10,Eck89,BT89} and references therein.
In such applications, the objective function usually consists of a data fitting term and a regularization term, which commonly have different structures and analytic properties, such as separability and Lipschitz continuity, and so on.
Very often, the functions $f$ and $g$ are relatively simple in the sense that their respective characteristics can be explored to design practical and efficient algorithms for solving  (P$_1$) with fairly large scale data.
In this paper, the structures that we assume for $f$ and $g$ are related to the proximity operator of closed proper convex functions, which is defined below.
\begin{definition}[Proximity operator]
  Let $h: \Re^n \rightarrow (-\infty,+\infty]$ be an extended real-valued closed proper convex function, and $\gamma >0$.  The proximity operator of $h$ is defined as
  \begin{equation}\label{def:prox-eq}
    \text{prox}_{\gamma}^h(x) := \arg\min\nolimits_{z} \left\{ h(z) + \frac{1}{2\gamma}\|z - x\|^2: \;   z\in\Re^n\right\}, \; x \in \Re^n.
  \end{equation}
\end{definition}
Throughout this paper, we  make the following assumption on $f$ and $g$.
\begin{assumption}
   \label{as:f-g-simple}
 Assume that the  proximity operators of $f$ and $g$ can be evaluated efficiently.
\end{assumption}

We now present other formulations of (P$_1$).
To sufficiently explore the problem  structures, it is not uncommon to introduce an auxiliary variable and decouple the composition of $g$ and $A$.
As such,   (P$_1$) is usually reformulated as a linear equality constrained  problem with  separable  objective function, that is
\begin{equation}\tag{P$_2$}
\min_{x,u}  \left\{  f(x) + g(u): \ \hbox{s.t. } u - Ax  = 0, \, u\in\Re^m, \, x\in\Re^n\right\}.
\end{equation}
This formulation allows one to take advantage of the structures of $f$ and $g$ (and possibly $A$) individually and adequately.
To present the dual problem of (P$_1$) or (P$_2$), we need the notion of conjugate of a closed proper convex function $h: \Re^n\rightarrow (-\infty,\infty]$, which is defined as
\begin{eqnarray*}
  \label{def:conj-eq}
  h^*(x) := \sup\nolimits_{z} \left\{ \langle x, z\rangle - h(z): \;  z\in\Re^n\right\}, \; x\in\Re^n.
\end{eqnarray*}
The well-known Moreau's decomposition (see, e.g., \cite{Rock70}) links the proximity operator of a  closed proper convex function $h$ and that of its conjugate $h^*$ as follows
\begin{equation}
\label{lem:Moreau-eq}
  z = \text{prox}_{t}^h(z) + t \, \text{prox}_{t^{-1}}^{h^*}\left(t^{-1}z\right), \; \forall z\in\Re^n, \; \forall t > 0.
\end{equation}
It follows from \eqref{lem:Moreau-eq} and Assumption \ref{as:f-g-simple} that  the proximity operators of $f^*$ and $g^*$ are also easy to evaluate.
The Lagrangian dual problem of (P$_1$) or (P$_2$) can be represented as
\begin{equation}
\tag{D$_1$}
  \max_y \{  - g^*(y) -f^*(-A^\top y): \, y\in\Re^m\},
\end{equation}
where $y$ is the dual variable.
Similarly, we can rewrite (D$_1$), by introducing a new variable $v$,  as
\begin{equation}
\tag{D$_2$}
  \max_{y,v} \{ - g^*(y) -f^*(v): \, \hbox{s.t. } v + A^\top y=0, \; v\in \Re^n, \, y\in\Re^m\}.
\end{equation}
The primal-dual formulation of (P$_1$) or (D$_1$) is given by
\begin{equation}\tag{PD}
\min_x \max_y \left\{  f(x) + \langle Ax, y\rangle - g^*(y): \; x\in\Re^n, y\in\Re^m\right\}.
\end{equation}

It is apparent that any algorithm that solves (P$_1$) (resp., (D$_1$)) also solves (P$_2$) (resp., (D$_2$)), and vice versa.
If an algorithm solves (PD), then (P$_1$) and (D$_1$) are solved simultaneously. All the algorithms discussed in this paper solve (PD), and thus (P$_1$) and (D$_1$).
We emphasize that, throughout this paper, $x$ and $y$ denote, respectively, the essential primal and dual variables, while $u$ and $v$ are auxiliary variables for the primal and the dual problems, respectively.
In the following, we first review briefly  some  popular methods that are closely related to this work  and then summarize our motivation and contributions.

\subsection{Augmented Lagrangian based methods}
The augmented Lagrangian method (abbreviated as ALM, also known as the method of multipliers), pioneered by Hestense \cite{Hes69} and Powell \cite{Pow69}, is among the most influential approaches for solving constrained optimization problems, especially when the constraints contain linear equalities.
We take the dual problem (D$_2$) as an example and explain the main idea of the ALM.
The augmented Lagrangian function associated with (D$_2$) is given by
\begin{eqnarray}\label{def:AL-D}
\nonumber
  \calL_D^\tau(y,v,x) &:=& g^*(y) + f^*(v) - \langle x,  v + A^\top y \rangle + \frac{\tau}{2}\| v + A^\top y\|^2 \\
                       &=& g^*(y) + f^*(v)  + \calQ_D^\tau(y,v,x),
\end{eqnarray}
where $x\in\Re^n$ is the Lagrangian multiplier (and also the primal variable in (P$_1$) or (P$_2$)), $\tau>0$ is a penalty parameter, and $\calQ_D^\tau(y,v,x)$ is defined as
\begin{equation}
  \label{def:Q-D}
\calQ_D^\tau(y,v,x) := \frac{\tau}{2} \| v + A^\top y - \tau^{-1} x \|^2 - \frac{1}{2\tau}\|x\|^2.
\end{equation}
Given $x^k\in\Re^n$, the ALM iterates as
\begin{subequations}
  \label{ALM}
  \begin{eqnarray}
    \label{ALM-yv}
    (y^{k+1},v^{k+1}) &=& \arg\min_{y,v} \calL_D^\tau(y,v,x^k), \\
    \label{ALM-x}
     x^{k+1} &=& x^k - \tau (v^{k+1} + A^\top y^{k+1}).
  \end{eqnarray}
\end{subequations}
The most important feature of the ALM is that it solves a constrained optimization problem via solving a sequence of unconstrained ones.
Note that in our setting the functions $f$ and $g$ have structures. It is thus rather unwise to ignore the separability of the objective function and apply a joint minimization with respect to $(y,v)$ as in \eqref{ALM-yv}, because in this case it can be very difficult to fully explore the structures of $f$ and $g$ due to the mixing of variables.

In contrast, the alternating direction method of multipliers (abbreviated as ADM), pioneered by Glowinski and Marrocco \cite{GM75} and Gabay and Mercier \cite{GM76}, is a practical variant of the ALM. It applies alternating minimization with respect to $y$ and $v$ in \eqref{ALM-yv} in a Gauss-Seidel fashion, with the other variable fixed. After each sweep of alternating minimization, the multiplier $x$ is updated just as in the ALM. Specifically, given $v^k,x^k\in\Re^n$, the ADM iterates as
\begin{subequations}\label{ADM-yvx}
\begin{eqnarray}
\label{ADM-yvx-y}
y^{k+1} &=& \arg\min_{y} \calL_D^\tau(y,v^k,x^k), \\
\label{ADM-yvx-v}
v^{k+1} &=& \arg\min_{v} \calL_D^\tau(y^{k+1},v,x^k), \\
\label{ADM-yvx-x}
     x^{k+1} &=& x^k - \tau (v^{k+1} + A^\top y^{k+1}).
\end{eqnarray}
\end{subequations}
Compared to the ALM, an obvious advantage of the ADM is that it solves simpler subproblems in each round. Since the minimizations for $y$ and $v$ are now separated, the structures of $f$ and $g$ can be utilized individually.
Interested readers are referred to the recent tutorial paper \cite{Eck11} for more details on the ALM and the ADM, including convergence analysis with nonexpansive mappings.

To make ADM efficient, it is very important to have low per-iteration cost and fast iterations.
Note that, by the definition of $\calL_D^\tau$ in \eqref{def:AL-D}, the $v$-subproblem \eqref{ADM-yvx-v} is already a proximal minimization step and is thus simple enough under Assumption \ref{as:f-g-simple}.
In comparison, it is very likely that the $y$-subproblem \eqref{ADM-yvx-y} is not easy in the sense that it calls for an iterative solver, even though the proximity operator of $g$ is easily obtainable. %
It is easy to observe that the components of $y$ in the quadratic term $\calQ_D^\tau$ are mixed-all-together due to the presence of the linear operator $A$.
To avoid solving it iteratively, the $y$-subproblem \eqref{ADM-yvx-y} needs to be treated wisely.
The most widely used technique to modify \eqref{ADM-yvx-y} so that the resulting subproblem can be solved easily by utilizing the structure of $g$ is to linearize  with respect to $y$ the quadratic term  $\calQ_D^\tau(y,v^k,x^k)$ at $y=y^k$ and meanwhile adding a proximal term $\frac{1}{2\sigma}\|y-y^k\|^2$ for some $\sigma >0$.
This proximal-linearization technique has been used extensively in the context of structured convex optimization, see, e.g., \cite{Nes07,HYZ08,BT09,MGC11,YZ11,ST14}.
In this paper, we refer the algorithm resulting from ADM and this proximal-linearization technique applied to one of the ADM subproblems as linearized ADM or LADM for short.

\subsection{Proximal point method and its inertial variant}
Another approach closely related to this work is the classical proximal point method (abbreviated as PPM, \cite{Mor65,Mar70,Roc76a}).
 Let $T: \Re^n \rightrightarrows \Re^n$ be a set-valued maximal monotone operator from $\Re^n$ to its power set.  The PPM  is an approach to solving the maximal monotone operator inclusion problem, i.e.,  find $w^*\in\Re^n$ such that $0 \in T(w^*)$.
The PPM has proven to be an extremely powerful algorithmic tool and contains many well known algorithms as special cases, including the aforementioned ALM  and   ADM, see  \cite{Roc76b, EB92}. See also \cite{Roc76a, EB92, Gul92} for  inexact, relaxed and accelerated variants of the PPM.
The primary PPM for minimizing a differentiable function $\psi:\; \Re^n\rightarrow \Re$ can be interpreted as an implicit one-step discretization method for the ordinary differential equations (abbreviated as ODEs) $w' + \nabla \psi(w) = 0$, where $w: \Re\rightarrow \Re^n$ is differentiable,  $w'$ denotes its derivative, and $\nabla \psi$ is the gradient of $\psi$. Suppose that $\psi$ is closed proper convex and its minimum value  is attained, then every solution trajectory $\{w(t): \; t\geq 0\}$ of this differential system   converges to a minimizer of $\psi$ as $t\rightarrow\infty$. Similar conclusion can be drawn for the maximal monotone operator inclusion problem by considering the evolution differential inclusion problem $0 \in w'(t) + T(w(t))$ almost everywhere on $\Re_+$, provided that the operator $T$ satisfies certain conditions \cite{Bru75}.

The PPM is a single-step method, meaning that each new iterate depends only on the current point.
To accelerate speed of convergence, multi-step methods have been proposed in the literature, which can usually be viewed as certain discretizations of second-order ODEs of the form
\begin{equation}
  \label{HBF}
w'' + \gamma w' + \nabla \psi(w) = 0,
\end{equation}
where $\gamma>0$ is a friction parameter. Relevant studies in the context of optimization  can be traced back to \cite{Pol64}.
It was shown in  \cite{Alv00} that if $\psi$ is convex and its minimum value is attained then each solution trajectory $\{w(t): t\geq 0\}$ of \eqref{HBF} converges  to a minimizer of $\psi$.
In theory the convergence of the solution trajectories of \eqref{HBF} to a stationary point of $\psi$  can be faster than those of the first-order ODEs, while
in practice the second order   term $w''$ can be exploited to design faster algorithms \cite{APZ84,Ant94}.
Motivated by the properties of \eqref{HBF}, an implicit  discretization method was proposed in \cite{Alv00}. Specifically, given $w^{k-1}$ and $w^k$, the next point $w^{k+1}$ is determined via
\[
\frac{w^{k+1} - 2w^k + w^{k-1}}{h^2}  + \gamma\frac{w^{k+1} -  w^k}{h} +  \nabla \psi(w^{k+1}) = 0,
\]
which results to an iterative algorithm of the form
\begin{equation}
  \label{iPPA-min-f}
w^{k+1} = (I + \lambda \nabla \psi)^{-1} (w^k + \alpha (w^k-w^{k-1})),
\end{equation}
where $\lambda = h^2/(1+\gamma h)$, $\alpha = 1/(1+\gamma h)$ and $I$ is the identity operator. Note that  \eqref{iPPA-min-f} is no more than a proximal point step taken at the extrapolated point $w^k + \alpha (w^k-w^{k-1})$, rather than $w^k$ itself as in the classical PPM. The iterative scheme \eqref{iPPA-min-f} is a two-step method and is usually referred as inertial PPM, because \eqref{HBF} describes in the two dimensional case the motion of a heavy ball on the graph of $\psi$ under its own inertial, together with friction  and gravity forces. Convergence properties of \eqref{iPPA-min-f} were studied in \cite{Alv00} under some assumptions on the parameters $\alpha$ and $\lambda$.
Subsequently, this inertial technique was extended to solve the maximal monotone operator inclusion problem  in \cite{AA01,ME03,Alv04, MM07}.
Recently, there are increasing interests in studying inertial type algorithms, e.g., inertial forward-backward splitting methods \cite{OCBP14,OBP14,APR14,BC14d}, inertial Douglas-Rachford operator splitting method \cite{BCH14} and inertial ADM \cite{BC14a}.
Lately, we proposed in \cite{CMY14a} a general inertial PPM for mixed variational inequality, where the weighting matrix  is allowed to be positive semidefinite. Global point-convergence and certain convergence rate results are also given there.

\subsection{Motivation and contributions}
In this paper, we study  inertial versions of Chambolle-Pock's primal-dual algorithm \cite{CP11} and variants.
In the recent work \cite{CMY14a}, we proposed a general inertial PPM under the setting of mixed variational inequality (abbreviated as MVI). There, an inertial LADM was proposed, where the two ADM subproblems must be linearized simultaneously in order to guarantee the positive definiteness of a weighting matrix when the resulting algorithm is viewed as a general PPM.
Note that the main aim of applying the proximal-linearization technique is to solve all subproblems efficiently via utilizing the proximity operators.
It is apparent that the $v$-subproblem \eqref{ADM-yvx-v} is already a proximal minimization step, which can make full use of the proximity operator of $f$ under Assumption \ref{as:f-g-simple} and the Moreau's decomposition \eqref{lem:Moreau-eq}. Thus, an approximation of \eqref{ADM-yvx-v} by proximal-linearization is unnecessary in any sense.
It is thus desirable to consider inertial LADM with only one of the ADM subproblems linearized.

In this paper, we first further clarify, based on previous observations in \cite{CP11,ST14}, the connection between CPAs and LADM, where only one of the ADM subproblems is modified by proximal-linearization. In particular, we show that CPAs generate exactly the same sequence of points with LADM applied to either the primal problem or its Lagrangian dual, as long as the initial points for the LADM are properly chosen. By focusing on cyclically equivalent forms of the algorithms, we can relax the dependence on initial points for LADM. Then, by utilizing the fact that CPAs are applications of a general PPM, we are able to propose inertial CPAs, whose global point-convergence, nonasymptotic $O(1/k)$ and asymptotic $o(1/k)$ rates can be guaranteed.
Since CPAs are equivalent to LADMs, the proposed algorithms are also inertial LADMs.

\subsection{Organization}
The paper is organized as follows.
In Section \ref{sc:CPA=LADM}, we study the equivalence of  CPAs and different applications of LADM.
In Section \ref{sc:VI-LADM}, we explain CPAs  as  applications of a general PPM to the MVI formulation of the KKT system. This explanation allows us to study CPAs within the setting of PPM.
Inertial CPAs are also proposed in this section, and convergence results including global point-convergence, nonasymptotic $O(1/k)$ and asymptotic $o(1/k)$ convergence rates are given.
In Section \ref{sc:numerical},   we demonstrate the performance of inertial CPAs  via experimental results on compressive image reconstruction based on total variation minimization. Finally, we give some concluding remarks in Section \ref{sc:concluding}.

\subsection{Notation}
Our notation is rather standard, as used above in this section. The standard inner product and $\ell_2$ norm are denoted by $\langle\cdot,\cdot\rangle$ and $\|\cdot\|$, respectively.
The superscript ``$^\top$" denotes the matrix/vector transpose operator.
For any positive semidefinite matrix $M$ of size $n$-by-$n$ and vectors $u, v\in \Re^n$, we let $\langle u, v\rangle_M := u^\top Mv$ and $\|u\|_M := \sqrt{\langle u, u\rangle_M}$.
The spectral radius of a square matrix $M$ is denoted by  $\rho(M)$. The identity matrix of order $m$ is denoted by $I_m$.
With a little abuse of notation, the columnwise adhesion of two columns vectors $u$ and $v$, i.e., $(u^\top,v^\top)^\top$, is often denoted by $(u,v)$ whenever it does not incure any confusion. Other notation will be introduced as the paper progresses.

\section{Equivalence of CPAs and LADMs}\label{sc:CPA=LADM}
In this section, we first present two versions of CPAs, which can be viewed as applying the original CPA proposed in \cite{CP11} to the primal problem (P$_1$) and the dual problem (D$_1$), respectively. Each version of the CPAs can appear in two different forms, depending on which variable is updated first.
We then study the equivalence of CPAs  and different applications of LADM. Our results are based on the previous observations given in \cite{CP11,ST14}.

\subsection{Two versions of CPAs}\label{ssc:CPA}
Recall that the proximity operator of a closed proper convex function is defined in \eqref{def:prox-eq}.
The original CPA proposed by Chambolle \& Pock  in \cite{CP11} to solve (P$_1$) and its other formulations is summarized below in Algorithm \ref{alg:yxx}.
\begin{algorithm}[CP-$yx\bar x$]\label{alg:yxx}
  Given $\sigma, \tau > 0$, $x^0\in\Re^n$ and $y^0\in\Re^m$. Set $\bar x^0 = x^0$. For $k\geq 0$, iterate as
  \begin{subequations}\label{yxx}
    \begin{eqnarray}
    \label{yxx-y}
      y^{k+1} &=& \text{prox}_{\sigma}^{g^*}(y^k + \sigma A \bar x^k), \\
    \label{yxx-x}
      x^{k+1} &=& \text{prox}_{\tau}^f(x^k - \tau A^\top  y^{k+1}),  \\
    \label{yxx-xbar}
      \bar x^{k+1} &=& 2 x^{k+1} - x^k.
    \end{eqnarray}
  \end{subequations}
\end{algorithm}

Algorithm \ref{alg:yxx} will be referred as CP-$yx\bar x$ subsequently, because it updates the dual variable $y$ first, followed by the primal variable $x$, and finally an extrapolation step in $x$ to obtain $\bar x$.
We note that in the original work \cite{CP11} the extrapolation step \eqref{yxx-xbar} takes the form $\bar x^{k+1} =  x^{k+1} + \gamma(x^{k+1}- x^k)$, where $\gamma\in[0,1]$ is a parameter.
In this paper, we only focus on the case $\gamma=1$, which is exclusively used in practice.
It was shown in \cite{CP11} that, under certain assumptions, the sequence $\{(x^k,y^k)\}_{k=0}^\infty$ generated by \eqref{yxx} converges to a solution of (PD) for general closed proper convex functions $f$ and $g$. In particular, an ergodic sublinear convergence result was obtained. Accelerations of CP-$yx\bar x$ were also considered there for problems with strong convexity.

By moving \eqref{yxx-y} to after \eqref{yxx-xbar} and reindexing the  points, we obtain a  cyclically equivalent form of CP-$yx\bar x$, which is summarized below in Algorithm \ref{alg:xxy} and will be referred as CP-$x\bar xy$ subsequently. Note that, compared to CP-$yx\bar x$, $\bar x^0$ is no longer needed to launch the algorithm.
\begin{algorithm}[CP-$x\bar xy$]\label{alg:xxy}
Given $\sigma, \tau > 0$,  $x^0\in\Re^n$ and $y^0\in\Re^m$. For $k\geq 0$, iterate as
  \begin{subequations}\label{xxy}
\begin{eqnarray}
\label{xxy-x}
  x^{k+1} &=& \text{prox}_{\tau}^f(x^k - \tau A^\top  y^k),  \\
\label{xxy-xbar}
  \bar x^{k+1} &=& 2x^{k+1} - x^k, \\
\label{xxy-y}
  y^{k+1} &=& \text{prox}_{\sigma}^{g^*}(y^k + \sigma A \bar x^{k+1}).
\end{eqnarray}
\end{subequations}
\end{algorithm}

By comparing (D$_1$) with (P$_1$), it is easy to observe  that $f^*$, $g^*$, $-A^\top$ and $y$ in (D$_1$) play, respectively, the roles of
$g$, $f$, $A$ and $x$ in (P$_1$). Thus, by  exchanging of variables, functions and parameters in CP-$yx\bar x$ as follows
\begin{eqnarray*}
  g \longleftarrow f^*, \;  f \longleftarrow g^*, \; A \longleftarrow -A^\top,  \; x \longleftrightarrow y,  \; \bar x \longleftarrow \bar y,  \; \sigma \longleftrightarrow \tau,
\end{eqnarray*}
and using the fact that $h^{**} = h$ for any closed proper convex function $h$, see, e.g., \cite{Rock70}, we obtain another  CPA, which is stated below and will be referred as
CP-$xy\bar y$ later.

\begin{algorithm}[CP-$xy\bar y$]\label{alg:xyy}
  Given $\sigma, \tau > 0$,  $x^0\in\Re^n$ and $y^0\in\Re^m$. Set $\bar y^0 = y^0$. For $k\geq 0$, iterate as
  \begin{subequations}\label{xyy}
    \begin{eqnarray}
    \label{xyy-x}
    x^{k+1} &=& \text{prox}_{\tau}^f(x^k - \tau A^\top  \bar y^k),  \\
    \label{xyy-y}
    y^{k+1} &=& \text{prox}_{\sigma}^{g^*}(y^k + \sigma A  x^{k+1}), \\
    \label{xyy-ybar}
    \bar y^{k+1} &=& 2y^{k+1} - y^k.
    \end{eqnarray}
    \end{subequations}
\end{algorithm}
If CP-$yx\bar x$ is viewed as applying the original CPA in \cite{CP11} to the primal problem (P$_1$), then CP-$xy\bar y$ can be considered as applying the original CPA to the dual problem (D$_1$).
Similarly,  by moving  \eqref{xyy-x}  to after \eqref{xyy-ybar}   and reindexing the  points, we obtain a cyclically equivalent algorithm, which
will be referred as CP-$y\bar yx$ and is given below in Algorithm \ref{alg:yyx}.
It is alike that, compared to CP-$xy\bar y$, $\bar y^0$ is no longer needed to start the algorithm.

\begin{algorithm}[CP-$y\bar yx$]\label{alg:yyx}
Given $\sigma, \tau > 0$,  $y^0\in\Re^m$ and $x^0\in\Re^n$. For $k\geq 0$, iterate as
\begin{subequations}\label{yyx}
\begin{eqnarray}
\label{yyx-y}
y^{k+1} &=& \text{prox}_{\sigma}^{g^*}(y^k + \sigma A  x^k), \\
\label{yyx-ybar}
\bar y^{k+1} &=& 2y^{k+1} - y^k, \\
\label{yyx-x}
x^{k+1} &=& \text{prox}_{\tau}^f(x^k - \tau A^\top  \bar y^{k+1}).
\end{eqnarray}
\end{subequations}
\end{algorithm}

Compared with CP-$yx\bar x$ given in \eqref{yxx}, it is easy to see that CP-$xy\bar y$ given in \eqref{xyy}  just exchanged the updating order of the primal and the dual variables.
After updating both the primal and the dual variables, CP-$yx\bar x$ and CP-$xy\bar y$ apply an extrapolation step to the latest updated variable.
In comparison, CP-$x\bar xy$ and CP-$y\bar yx$ given  respectively  in \eqref{xxy} and \eqref{yyx} apply an extrapolation step immediately after one of the variables
is updated, followed by updating the other variable.  In any case, the extrapolation step must be applied to the latest updated variable.

We emphasize that, although CP-$yx\bar x$  and CP-$x\bar xy$ are  cyclically equivalent, it is more convenient to analyze CP-$x\bar xy$ in the setting of PPM. This is because
the iteration of CP-$x\bar xy$ is from $(x^k,y^k)$ to $(x^{k+1},y^{k+1})$. By using the notion of proximity operator, we can express $x^{k+1}$ and $y^{k+1}$ explicitly in terms of $x^k$ and $y^k$, which is convenient for analysis.  In comparison, CP-$yx\bar x$ does not have this feature. Similar remarks are applicable to CP-$xy\bar y$ and CP-$y\bar yx$. Therefore, we only concentrate on CP-$x\bar xy$ and CP-$y\bar yx$ for convergence analysis.
In the following, we prove that CP-$yx\bar x$/CP-$x\bar xy$ and CP-$xy\bar y$/CP-$y\bar yx$ are equivalent to applying LADM to the dual problem (D$_2$) and the primal problem (P$_2$), respectively.

\subsection{Equivalence of CP-$yx\bar x$/CP-$x\bar xy$ and LADM for the dual problem (D$_2$)}
Let $\calL_D^\tau(y,v, x)$ and $\calQ_D^\tau(y,v,x)$ be defined in \eqref{def:AL-D} and \eqref{def:Q-D}, respectively.
To solve (D$_2$) by the ADM, the following subproblem needs to be solved repeatedly:
\begin{eqnarray*}\label{ADM-D-y}
\min_y \left\{ \calL_D^\tau(y,v^k, x^k) =  f^*(v^k) + g^*(y)  + \calQ_D^\tau(y,v^k,x^k): \; y\in\Re^m \right\}.
\end{eqnarray*}
To avoid solving it iteratively and construct an algorithm with cheap per-iteration cost, LADM linearizes the quadratic term $ \calQ_D^\tau(y,v^k,x^k)$
at $y=y^k$ and meanwhile adds a proximal term $\frac{1}{2\sigma}\|y-y^k\|^2$ for some $\sigma>0$. As such, $y^{k+1}$ is obtained as the solution of the following approximation problem (constant terms are omitted)
\begin{eqnarray*}\label{ADM-D-y-approx}
\nonumber
\min_y  g^*(y) + \big\langle \nabla_y \calQ_D^\tau(y^k,v^k, x^k), y-y^k\big\rangle + \frac{1}{2\sigma}\|y-y^k\|^2,
\end{eqnarray*}
where $\sigma >0$ is a proximal parameter. By using the proximity operator defined in \eqref{def:prox-eq}, we can summarize  the resulting LADM below in Algorithm \ref{alg:LADM-D}, which will be referred as LADMD-$yvx$ for apparent reason.
\begin{algorithm}[LADMD-$yvx$]\label{alg:LADM-D}
   Given $\tau,\sigma>0$, $x^0\in\Re^n$, $y^0\in\Re^m$ and $v^0\in\Re^n$. The LADM applied to the dual problem (D$_2$) iterates, for $k\geq 0$, as
  \begin{subequations}
    \label{LADM-D}
    \begin{eqnarray}
    \label{LADM-D-y}
      y^{k+1} &=& \text{prox}_{\sigma}^{g^*} \big(y^k - \sigma \nabla_y \calQ_D^\tau(y^k,v^k, x^k)\big),\\
    \label{LADM-D-v}
      v^{k+1} &=& \text{prox}_{\tau^{-1}}^{f^*} (\tau^{-1} x^k  - A^\top y^{k+1}),\\
    \label{LADM-D-x}
      x^{k+1} &=& x^k - \tau  (v^{k+1} + A^\top y^{k+1}).
    \end{eqnarray}
  \end{subequations}
\end{algorithm}

The next theorem establishes the equivalence of CP-$yx\bar x$ and LADMD-$yvx$.  In \cite{CP11}, CP-$yx\bar x$ was explained as a preconditioned  ADM.

\begin{theorem}[Equivalence of CP-$yx\bar x$  and LADMD-$yvx$]\label{thm:yxx=LADM-D}
Let $\tau,\sigma>0$, $x^0\in\Re^n$, $y^0\in\Re^m$ and $v^0\in\Re^n$ be given. Suppose that $ v^0 + A^\top y^0 = 0$.
Then,  CP-$yx\bar x$ and LADMD-$yvx$ given in \eqref{yxx} and \eqref{LADM-D}, respectively, are  equivalent  in the sense that both algorithms
generate exactly the same sequence $\{(x^k,y^k)\}_{k=1}^\infty$.
\end{theorem}
\begin{proof}
We will show that the sequence $\{(x^k,y^k)\}_{k=1}^\infty$ generated by \eqref{LADM-D} satisfies \eqref{yxx}.
Let $k\geq 0$. From the Moreau's decomposition \eqref{lem:Moreau-eq}, $v^{k+1}$ given in \eqref{LADM-D-v} can be rewritten as
  \begin{equation}\label{LADM-D-v-2}
  v^{k+1} 
  = \tau^{-1} x^k   - A^\top y^{k+1} - \tau^{-1} \text{prox}_{\tau}^{f} \big(x^k   - \tau A^\top y^{k+1}\big).
\end{equation}
It is easy to see from \eqref{LADM-D-v-2}  that $x^{k+1}$ given in \eqref{LADM-D-x} satisfies
  \begin{equation*}
  \label{LADM-D-x-2}
   x^{k+1} = x^k - \tau (v^{k+1} + A^\top y^{k+1}) = \text{prox}_{\tau }^{f} \big(x^k   - \tau A^\top y^{k+1}\big),
\end{equation*}
which is exactly \eqref{yxx-x}. From \eqref{LADM-D-x} and the assumption that $ v^0 + A^\top y^0 = 0$, we obtain
\begin{equation}
  \label{eq:xn-tau(v+Aty)}
x^k - \tau ( v^k + A^\top y^k) =
\left\{
  \begin{array}{ll}
    x^0, & \hbox{if $k=0$,} \\
    2x^k-x^{k-1}, & \hbox{if $k\geq 1$.}
  \end{array}
\right.
\end{equation}
Note that $\bar x^0 = x^0$ in CP-$yx\bar x$. It is thus clear from \eqref{eq:xn-tau(v+Aty)} and \eqref{yxx-xbar}
that $x^k - \tau ( v^k + A^\top y^k) = \bar x^k$ for all $k\geq 0$.
By direct calculation, we have
\[
y^k - \sigma \nabla_y \calQ_D^\tau(y^k,v^k, x^k) = y^k + \sigma  A( x^k - \tau  (v^k + A^\top y^k) ) = y^k + \sigma  A\bar x^k.
\]
Therefore, $y^{k+1}$ given in \eqref{LADM-D-y} reduces to $y^{k+1} = \text{prox}_{\sigma}^{g^*} \big( y^k + \sigma  A\bar x^k \big)$,
which is \eqref{yxx-y}.
\end{proof}

By moving \eqref{LADM-D-y} to after \eqref{LADM-D-x} and reindexing the  points, we obtain a  cyclically equivalent form of LADMD-$yvx$,
which is  given below and will be referred as LADMD-$vxy$.

\begin{algorithm}[LADMD-$vxy$]\label{alg:LADM-D-vxy}
 Given $\tau,\sigma>0$, $x^0\in\Re^n$ and $y^0\in\Re^m$. The LADM applied to the dual problem (D$_2$) iterates, for $k\geq 0$, as
  \begin{subequations}\label{LADM-D-vxy}
\begin{eqnarray}
\label{LADM-D-vxy-v}
  v^{k+1} &=& \text{prox}_{\tau^{-1}}^{f^*} (\tau^{-1} x^k  - A^\top y^k),\\
\label{LADM-D-vxy-x}
  x^{k+1} &=& x^k - \tau (v^{k+1} + A^\top y^k), \\
\label{LADM-D-vxy-y}
  y^{k+1} &=& \text{prox}_{\sigma}^{g^*} \big(y^k - \sigma \nabla_y \calQ_D^\tau(y^k,v^{k+1}, x^{k+1})\big).
\end{eqnarray}
\end{subequations}
\end{algorithm}

Note that, since LADMD-$vxy$ updates $v$ first, it can be launched with $(x^0,y^0)$ but without initialization of $v$.
Similarly, $\bar x^0$ is not needed to start   CP-$x\bar xy$.
The equivalence of CP-$x\bar xy$ and LADMD-$vxy$ is stated in Theorem \ref{thm:xxy=LADM-D-vxy}, whose proof is analogous to that of Theorem \ref{thm:yxx=LADM-D} and is thus omitted.
In contrast to the equivalence of CP-$yx\bar x$ and LADMD-$yvx$, that of CP-$x\bar xy$ and LADMD-$vxy$ does not require the condition $v^0 + A^\top y^0 = 0$ anymore.

\begin{theorem}[Equivalence of CP-$x\bar xy$ and LADMD-$vxy$]\label{thm:xxy=LADM-D-vxy}
Let $\tau,\sigma>0$, $x^0\in\Re^n$ and $y^0\in\Re^m$ be given.
Then, CP-$x\bar xy$ and LADMD-$vxy$ given, respectively, in \eqref{xxy} and \eqref{LADM-D-vxy} are  equivalent  in the sense that both algorithms
generate exactly the same sequence $\{(x^k,y^k)\}_{k=1}^\infty$.
\end{theorem}

\subsection{Equivalence of CP-$xy\bar y$/CP-$y\bar yx$ and LADM for the primal problem (P$_2$)}
Now we apply LADM to the primal problem (P$_2$).
The augmented Lagrangian function associated with (P$_2$) is given by
\begin{eqnarray*}\label{def:AL-P}
\nonumber
\calL_P^\sigma(x,u,y) &:=& f(x) + g(u) - \langle y, u-Ax\rangle + \frac{\sigma}{2} \|u-Ax\|^2 \\
                      &=& f(x) + g(u) + \calQ_P^\sigma(x,u,y),
\end{eqnarray*}
where $y\in\Re^m$ is the Lagrangian multiplier (and also the dual variable in (D$_1$) or (D$_2$)), $\sigma >0$ is a penalty parameter,  and $\calQ_P^\sigma(x,u,y)$ is defined as
\begin{equation*}
  \label{def:Q-P}
\calQ_P^\sigma(x,u,y) := \frac{\sigma}{2} \|u-Ax -\sigma^{-1}y\|^2 - \frac{1}{2\sigma}\|y\|^2.
\end{equation*}
Given $u^k, y^k\in\Re^m$, the ADM for solving (P$_2$) iterates as
\begin{subequations}\label{ADM-xuy}
\begin{eqnarray}
\label{ADM-xuy-x}
  x^{k+1} &=& \arg\min_{x\in \Re^n} \calL_P^\sigma(x,u^k,y^k), \\
\label{ADM-xuy-u}
  u^{k+1} &=& \arg\min_{u\in \Re^m} \calL_P^\sigma(x^{k+1},u,y^k), \\
\label{ADM-xuy-y}
  y^{k+1} &=& y^k - \sigma (u^{k+1} - Ax^{k+1}).
\end{eqnarray}
\end{subequations}
Similarly, due to the presence of linear operator $A$ in $\calQ_P^\sigma$, the solution of \eqref{ADM-xuy-x} calls for an inner loop in general.
To avoid solving it iteratively, we linearize at each iteration the quadratic term $ \calQ_P^\sigma(x,u^k,y^k)$
at $x=x^k$, add a proximal term and approximate it  by (again, constant terms are omitted)
\begin{eqnarray*}\label{ADM-xuy-x-approx}
\nonumber
\min_x  f(x) + \big\langle \nabla_x \calQ_P^\sigma(x^k,u^k, y^k), x-x^k\big\rangle + \frac{1}{2\tau}\|x-x^k\|^2,
\end{eqnarray*}
where $\tau >0$ is a proximal parameter. The resulting LADM is given in Algorithm \ref{alg:LADM-P} and will be referred as LADMP-$xuy$.
\begin{algorithm}[LADMP-$xuy$]\label{alg:LADM-P}
  Given $\tau,\sigma>0$, $x^0\in\Re^n$, $u^0\in\Re^m$ and $y^0\in\Re^m$. The LADM applied to the primal problem (P$_2$) iterates, for $k\geq 0$, as
  \begin{subequations}\label{LADM-P}
    \begin{eqnarray}
    \label{LADM-P-x}
        x^{k+1} &=& \text{prox}_{\tau}^{f} \left(x^k - \tau \nabla_x \calQ_P^\sigma(x^k,u^k, y^k)\right),\\
    \label{LADM-P-u}
        u^{k+1} &=& \text{prox}_{\sigma^{-1}}^{g} (\sigma^{-1}y^k  + Ax^{k+1}),\\
    \label{LADM-P-y}
        y^{k+1} &=& y^k - \sigma (u^{k+1} - A  x^{k+1}).
    \end{eqnarray}
  \end{subequations}
\end{algorithm}

The equivalence of CP-$xy\bar y$ and LADMP-$xuy$ can be established completely in analogous as in Theorem \ref{thm:yxx=LADM-D}. See also
\cite{ST14}. Similarly, to guarantee that both schemes generate exactly the same sequence of points, a condition $u^0 = A x^0$ must be imposed on the initial points,
which was not stated in the literature.

\begin{theorem}[Equivalence of CP-$xy\bar y$ and LADMP-$xuy$]\label{thm:xyy=LADM-P}
Let $\tau,\sigma>0$, $x^0\in\Re^n$, $u^0\in\Re^m$ and $y^0\in\Re^m$ be given. Suppose that $u^0 = A x^0$.
Then,  CP-$xy\bar y$ and LADMP-$xuy$ given in \eqref{xyy} and  \eqref{LADM-P}, respectively, are  equivalent  in the sense that both algorithms
generate exactly the same sequence $\{(x^k,y^k)\}_{k=1}^\infty$.
\end{theorem}

Similarly,  by moving  \eqref{LADM-P-x} to after \eqref{LADM-P-y}  and reindexing the  points, we obtain a cyclically equivalent algorithm that does not need $u^0$ in initialization. The algorithm, which  will be referred as LADMP-$uyx$,  and its equivalence to CP-$y\bar yx$ are summarized in Algorithm \ref{alg:LADM-P-uyx} and Theorem \ref{thm:yyx=LADM-P-uyx}, respectively.
\begin{algorithm}[LADMP-$uyx$]\label{alg:LADM-P-uyx}
  Given $\tau,\sigma>0$, $x^0\in\Re^n$ and $y^0\in\Re^m$. The LADM applied to the primal problem (P$_2$) iterates, for $k\geq 0$, as
\begin{subequations}\label{LADM-P-uyx}
\begin{eqnarray}
\label{LADM-P-uyx-u}
  u^{k+1} &=& \text{prox}_{\sigma^{-1}}^{g} (\sigma^{-1} y^k   + Ax^k),\\
\label{LADM-P-uyx-y}
  y^{k+1} &=& y^k - \sigma (u^{k+1} - A  x^k),\\
\label{LADM-P-uyx-x}
  x^{k+1} &=& \text{prox}_{\tau}^{f} \left(x^k - \tau \nabla_x \calQ_P^\sigma(x^k,u^{k+1}, y^{k+1})\right).
\end{eqnarray}
\end{subequations}
\end{algorithm}

\begin{theorem}[Equivalence CP-$y\bar yx$ and LADMP-$uyx$]\label{thm:yyx=LADM-P-uyx}
Let $\tau,\sigma>0$, $x^0\in\Re^n$ and $y^0\in\Re^m$ be given.
Then, CP-$y\bar yx$ and LADMP-$uyx$ given in \eqref{yyx} and \eqref{LADM-P-uyx}, respectively, are  equivalent  in the sense that both algorithms
generate exactly the same sequence $\{(x^k,y^k)\}_{k=1}^\infty$.
\end{theorem}

\section{Inertial CPAs}\label{sc:VI-LADM}
In this section, we first show that  CPAs are equivalent to applying a general PPM to the MVI formulation
of the KKT system of (PD). We then propose inertial CPAs.
Again, since CPAs are equivalent to LADMs, the proposed inertial CPAs can also be called inertial LADMs.
In the following, we mainly focus on CP-$y\bar yx$ given in \eqref{yyx} and discussions for other CPAs are alike.

\subsection{CP-$y\bar yx$ and CP-$x\bar xy$ are applications of a general PPM}\label{ssc:CPA=PPM}
Under certain regularity assumptions, see, e.g., \cite{Eck89}, solving the primal-dual pair (P$_1$) and (D$_1$) is equivalent to finding
$(x^*,y^*)\in \Re^n\times\Re^m$ such that the following KKT conditions are satisfied:
\begin{subequations}\label{kkt-P2}
\begin{eqnarray}
   f(x) - f(x^*) + \langle x - x^*, A^\top y^*\rangle &\geq& 0, \; \forall x\in\Re^n,\\
   g^*(y) - g^*(y^*) + \langle y - y^*, -Ax^*\rangle &\geq& 0, \; \forall y\in\Re^m.
\end{eqnarray}
\end{subequations}
In the rest of this paper, we use the notation  $\calW := \Re^n\times\Re^m$,
\begin{equation}\label{def-w-F-theta}
w := \left(
      \begin{array}{c}
        x \\
        y \\
      \end{array}
    \right), \quad
\theta(w) :=  f(x) + g^*(y), \quad
F(w) :=
\left(
  \begin{array}{cc}
    0 & A^\top   \\
    -A &  0 \\
  \end{array}
\right)
\left(
      \begin{array}{c}
        x \\
        y \\
      \end{array}
    \right).
\end{equation}
Since the coefficient matrix defining $F$ is skew-symmetric, $F$ is thus monotone.
Using these notation, the KKT system \eqref{kkt-P2} can be equivalently represented as a MVI problem:  find $w^*\in\calW$ such that
\begin{equation}\label{mVI}
  \theta(w) - \theta(w^*) + \langle w-w^*,  F(w^*)\rangle \geq 0,\; \forall w\in \calW.
\end{equation}
We make the following assumption on the problem (PD).
\begin{assumption}
  \label{as:KKT-nonempty}
Assume that the set of  solutions of \eqref{mVI}, denoted by $\calW^*$, is nonempty.
\end{assumption}

Using analysis similar to that in \cite[Lemma 2.2]{HY12}, we can show that CP-$y\bar yx$ is a general PPM applied to the MVI formulation \eqref{mVI}. Though the proof is simple, we give it for completeness.

\begin{lemma}\label{lem:LADM-mVI-k+1}
  For given $w^k = (x^k,y^k) \in \calW$, the new iterate $w^{k+1} = (x^{k+1},y^{k+1})$ generated by CP-$y\bar yx$ given in \eqref{yyx} satisfies
  \begin{equation}\label{LADM-mVI-k+1}
  w^{k+1} \in \calW, \;     \theta(w) - \theta(w^{k+1}) + \langle w - w^{k+1}, F(w^{k+1}) + G (w^{k+1} -  w^k)\rangle \geq 0, \; \forall w\in \calW,
  \end{equation}
 where $G$ is given by
 \begin{equation}\label{def:G-yyx}
 G  = \left(
  \begin{array}{cc}
 \frac{1}{\tau} I_n & A^\top  \\
 A & \frac{1}{\sigma}I_m \\
  \end{array}
\right).
 \end{equation}
\end{lemma}

\begin{proof}
Recall that the proximity operator is defined as the solution of an optimization problem in \eqref{def:prox-eq}.
The optimality conditions of \eqref{yyx-x} and \eqref{yyx-y} read
\begin{eqnarray*}
f(x) - f(x^{k+1}) + \left\langle  x-x^{k+1},  \, \frac{1}{\tau}(x^{k+1} - x^k) + A^\top  (2y^{k+1}-y^k) \right\rangle \geq 0, \; \forall x\in\Re^n, \\
g^*(y) - g^*(y^{k+1}) + \left\langle y-y^{k+1}, \,  \frac{1}{\sigma} ( y^{k+1} - y^k ) - A x^k \right\rangle \geq 0, \; \forall y\in\Re^m,
\end{eqnarray*}
which can be equivalently represented as
\begin{eqnarray*}
f(x) - f(x^{k+1}) + \left\langle  x-x^{k+1},  \, A^\top  y^{k+1} + \frac{1}{\tau}(x^{k+1} - x^k) + A^\top (y^{k+1}-y^k) \right\rangle \geq 0, \; \forall x\in\Re^n, \\
g^*(y) - g^*(y^{k+1}) + \left\langle y-y^{k+1}, \, - A x^{k+1} + A (x^{k+1}-   x^k) + \frac{1}{\sigma} ( y^{k+1} - y^k )  \right\rangle \geq 0, \; \forall y\in\Re^m.
\end{eqnarray*}
By the notation defined in \eqref{def-w-F-theta}, it is clear that the addition of the above two inequalities yields \eqref{LADM-mVI-k+1}, with $G$ defined in \eqref{def:G-yyx}.
\end{proof}

For CP-$x\bar xy$ given in  \eqref{xxy}, similar result holds. Specifically, the new iterate $w^{k+1} = (x^{k+1},y^{k+1})$ generated by CP-$x\bar xy$ from a given
$w^k = (x^k,y^k) \in \calW$ satisfies \eqref{LADM-mVI-k+1} with the weighting matrix $G$ given by
 \begin{equation}\label{def:G-xxy}
 G  = \left(
  \begin{array}{cc}
 \frac{1}{\tau} I_n & -A^\top  \\
 -A & \frac{1}{\sigma}I_m \\
  \end{array}
\right).
 \end{equation}

Throughout this paper, we make the following assumption on the parameters $\tau$ and $\sigma$.
\begin{assumption}\label{as:tau-sigma}
  The parameters $\tau$ and $\sigma$ satisfy the conditions $\tau, \sigma>0$ and $\tau\sigma < 1/\rho(A^\top A)$.
\end{assumption}

It is apparent that $G$ defined in \eqref{def:G-yyx} or \eqref{def:G-xxy} is symmetric and positive definite under Assumption \ref{as:tau-sigma}.
Thus, CP-$y\bar yx$ and CP-$x\bar xy$ can be viewed as a general PPM with a symmetric and positive definite weighting matrix $G$.
With this explanation, the convergence results of CPA can be established very conveniently under the PPM framework.
Here we  present the convergence results and omit the proof. Interested readers can refer to, e.g., \cite{CP11,HY12a,HY12b,ST14,CMY14a}, for similar convergence results and different analytic techniques.

\begin{theorem}[Convergence results of CP-$y\bar yx$ and CP-$x\bar xy$]\label{thm:convergence-LADM}
Assume that $\tau$ and $\sigma$ satisfy Assumption \ref{as:tau-sigma}.
Let $\{w^k = (x^k,y^k)\}_{k=0}^\infty$ be generated by  CP-$y\bar yx$ given in \eqref{yyx} or CP-$x\bar xy$ given in \eqref{xxy} from any starting point    $w^0 = (x^0,y^0) \in\calW$.
The following results hold.
\begin{enumerate}
  \item The sequence $\{w^k = (x^k,y^k)\}_{k=0}^\infty$ converges to a solution of \eqref{mVI}, i.e., there exists $w^\star  = (x^\star,y^\star) \in\calW^*$ such that $\lim_{k\rightarrow\infty}w^k = w^\star$, where $x^\star$ and $y^\star$ are, respectively, solutions of (P$_1$) and (D$_1$).
  \item   For any fixed integer $k>0$, define $\bar w^k := \frac{1}{k+1} \sum_{i=0}^k w^{i+1}$. Then, it holds that
\begin{equation*}\label{ergodic-rate}
\bar w^k\in\calW, \;
  \theta(w) - \theta(\bar w^k)  + ( w - \bar w^k)^\top F(w)    \geq - \frac{\|w - w^0\|_G^2}{2(k+1)}, \; \forall  w\in \calW.
\end{equation*}
  \item  After $k>0$ iterations, we have
  \begin{equation*}\label{non-ergodic-rate}
    \| w^k - w^{k-1}\|_G^2 \leq   \frac{\|w^0 - w^*\|_G^2} {k}.
  \end{equation*}
Moreover, it holds  as $k\rightarrow\infty$ that
 $ \|w^k - w^{k-1}\|_G^2 = o\left(1/k\right)$.
\end{enumerate}

\end{theorem}

\subsection{Inertial versions of CP-$y\bar yx$ and CP-$x\bar xy$}\label{sc:iLADM}
Since CPAs are applications of a general PPM, we can study the corresponding inertial algorithms by following the analysis in \cite{CMY14a}.
In this section, we propose inertial versions of CP-$y\bar yx$ and CP-$x\bar xy$ and present their convergence results.
The inertial versions of CP-$y\bar yx$ and CP-$x\bar xy$ are summarized below in Algorithms \ref{alg:iyyx} and \ref{alg:ixxy}, respectively.
\begin{algorithm}[Inertial   CP-$y\bar yx$, or iCP-$y\bar yx$]\label{alg:iyyx}
Let $\sigma, \tau >0$ and a sequence of nonnegative parameters $\{\alpha_k\}_{k=0}^\infty$ be given. Starting at any initial point $(x^0,y^0) = (x^{-1},y^{-1})$, the   algorithm iterates, for $k\geq 0$, as
\begin{subequations}\label{iyyx}
\begin{eqnarray}
\label{iyyx-xhat}
\hat x^k &=&  x^k + \alpha_k  (x^k-x^{k-1}),  \\
\label{iyyx-yhat}
\hat y^k &=&  y^k + \alpha_k  (y^k-y^{k-1}),  \\
\label{iyyx-y}
y^{k+1} &=& \text{prox}_{\sigma}^{g^*}(\hat y^k + \sigma A \hat x^k), \\
\label{iyyx-ybar}
\bar y^{k+1} &=& 2y^{k+1} - \hat y^k, \\
\label{iyyx-x}
x^{k+1} &=& \text{prox}_{\tau}^f(\hat x^k - \tau A^\top  \bar y^{k+1}).
\end{eqnarray}
\end{subequations}
\end{algorithm}

\begin{algorithm}[Inertial   CP-$x\bar xy$, or iCP-$x\bar xy$]\label{alg:ixxy}
Let $\sigma, \tau >0$ and a sequence of nonnegative parameters $\{\alpha_k\}_{k=0}^\infty$ be given. Starting at any initial point $(x^0,y^0) = (x^{-1},y^{-1})$, the   algorithm iterates, for $k\geq 0$, as
\begin{subequations}\label{ixxy}
\begin{eqnarray}
\label{ixxy-xhat}
\hat x^k &=&  x^k + \alpha_k  (x^k-x^{k-1}),  \\
\label{ixxy-yhat}
\hat y^k &=&  y^k + \alpha_k  (y^k-y^{k-1}),  \\
\label{ixxy-x}
x^{k+1} &=& \text{prox}_{\tau}^f(\hat x^k - \tau A^\top  \hat y^k),  \\
\label{ixxy-xbar}
\bar x^{k+1} &=& 2x^{k+1} - \hat x^k, \\
\label{ixxy-y}
y^{k+1} &=& \text{prox}_{\sigma}^{g^*}(\hat y^k + \sigma A \bar x^{k+1}).
\end{eqnarray}
\end{subequations}
\end{algorithm}
We will refer to Algorithms \ref{alg:iyyx} and \ref{alg:ixxy} as iCP-$y\bar yx$ and iCP-$x\bar xy$, respectively.
Recall that we use the notation $w$, $\theta$ and $F$ defined in \eqref{def-w-F-theta}.  We further define
\begin{eqnarray}\label{def:what} 
   \hat w^k &:=& w^k + \alpha_k (w^k - w^{k-1}).
\end{eqnarray}
According to Theorem \ref{thm:convergence-LADM}, the new point $w^{k+1}$ generated by iCP-$y\bar yx$  or  iCP-$x\bar xy$  conforms to
  \begin{equation}\label{iLADM-mVI}
  w^{k+1} \in \calW, \;     \theta(w) - \theta(w^{k+1}) + \langle w - w^{k+1}, F(w^{k+1}) + G (w^{k+1} -  \hat w^k)\rangle \geq 0, \; \forall w\in \calW,
  \end{equation}
where $G$ is given by \eqref{def:G-yyx} for  iCP-$y\bar yx$  and \eqref{def:G-xxy} for  iCP-$x\bar xy$.

The global point-convergence,  nonasymptotic $O(1/k)$  and asymptotic $o(1/k)$ convergence rate results of iCP-$y\bar yx$ and iCP-$x\bar xy$, or equivalently \eqref{def:what}-\eqref{iLADM-mVI} with
$G$ given by \eqref{def:G-yyx} for  iCP-$y\bar yx$  and \eqref{def:G-xxy} for  iCP-$x\bar xy$, follow directly from \cite[Theorem 2.2]{CMY14a}.

\begin{theorem}[Convergence results of iCP-$y\bar yx$ and iCP-$x\bar xy$]\label{Theorem1}
Suppose that $\tau$ and $\sigma$ satisfy Assumption  \ref{as:tau-sigma}, and, for all $k\geq 0$, it holds that $0\leq \alpha_k \leq \alpha_{k+1}\leq \alpha$ for some $0\leq\alpha < 1/3$.
Let $\{w^k\}_{k=0}^\infty$ conform to \eqref{iyyx} or \eqref{ixxy},
or equivalently, \eqref{def:what}-\eqref{iLADM-mVI} with $G$ given by \eqref{def:G-yyx} for  iCP-$y\bar yx$  and \eqref{def:G-xxy} for  iCP-$x\bar xy$.
Then, the following results hold.
\begin{enumerate}
    \item[(i)] The sequence $\{w^k\}_{k=0}^\infty$ converges to a member in $\calW^*$ as $k\rightarrow\infty$;
    \item[(ii)]  For any $w^*\in\calW^*$ and positive integer $k$, it holds that
\begin{equation}
  \label{rate:O(1/k)}
  \min_{0\leq j\leq k-1}  \|w^{j+1}-  \hat w^j\|_G^2 \leq \frac{\left(1+ \frac{2}{ 1-3\alpha}\right)\|w^0-w^*\|_G^2}{ k }.
\end{equation}
Furthermore, it holds as $k\rightarrow\infty$ that
\begin{eqnarray}
  \label{rate:o(1/k)}
  \min_{0\leq j\leq k-1}\Gnorm{w^{j+1}-\hat w^j }^2 = o\left(1\over k\right).
\end{eqnarray}
\end{enumerate}
\end{theorem}

It is easy to see from \eqref{iLADM-mVI} and the definition of $G$ in \eqref{def:G-yyx} or \eqref{def:G-xxy} that if $w^{k+1} = \hat w^k$ then $w^{k+1}$ is a solution of \eqref{mVI}.   Thus, the $O(1/k)$ and $o(1/k)$ results given in \eqref{rate:O(1/k)} and \eqref{rate:o(1/k)}, respectively, can be viewed as convergence rate results of  iCP-$y\bar yx$ and iCP-$x\bar xy$.

It is also worth to point out that the global convergence result given in Theorem \ref{Theorem1} is point-convergence, which is stronger than convergence in function values for the accelerated methods in \cite{Nes83,BT09}, which in fact can also be viewed as inertial type methods. Our stronger convergence result is obtained at the cost of more restrictive conditions on the inertial extrapolation parameters $\{\alpha_k\}_{k=0}^\infty$.

\section{Numerical results}\label{sc:numerical}
In this section, we present numerical results to compare the performance of CPAs and the proposed inertial CPAs.
In particular, we mainly concentrate on CP-$y\bar yx$  and its inertial variant iCP-$y\bar yx$ (the reasons will be explained below in Section \ref{sc:numerical-reason}).
All algorithms were implemented in MATLAB, and the experiments were performed with Microsoft Windows 8 and MATLAB v7.13 (R2011b),
running on a 64-bit Lenovo laptop with an Intel Core i7-3667U CPU at 2.00 GHz and 8 GB of memory.
Here, we only concentrate on a total variation (TV) based image reconstruction problem and compare iCP-$y\bar yx$ with CP-$y\bar yx$.
The performance of CPAs or LADMMs relative to other state-of-the-art algorithms is well illustrated in the literature, see, e.g., \cite{CP11} for various imaging problems, \cite{YZ11} for $\ell_1$-norm minimization in compressive sensing, \cite{YY13} for nuclear norm minimization in low-rank matrix completion, and \cite{WY12} for Danzig selector.

\subsection{Compressive image reconstruction based on total variation minimization}
In our experiments, we tested the problem of reconstructing an image from a number of its linear measurements, as in the setting of compressive sensing.
The reconstruction is realized via total variation  minimization. In variational image processing, TV minimizations have been widely used ever since the pioneering work \cite{ROF92} and have empirically shown to give favorable reconstruction results. It is well known that the edges of images can be well preserved if one minimizes the TV.
Another very important reason of the popularity of TV minimizations for image restoration is the availability of very fast numerical algorithms, e.g., \cite{Cha04,WYYZ08,GO09}.
In the compressive sensing setting, exact reconstruction guarantee of piecewise constant images from their incomplete frequencies and via  TV minimization was first obtained in \cite{CRT06}. Lately, it was shown in \cite{NW13a} that an image can be accurately recovered to within its best $s$-term approximation of its gradient from approximately $O(s \log(N))$ nonadaptive linear measurements, where $N$ denotes the number of pixels of the underlying image. In particular, the reconstruction is exact if the gradient of the image is precisely sparse, i.e., the image is piecewise constant. This is even true for high dimensional signals, see  \cite{NW13b}.

In the following, we let $A^{(1)}, A^{(2)} \in\Re^{n^2 \times n^2}$ be the first-order  global forward finite difference matrices (with certain boundary conditions assumed) in the horizontal and the vertical directions, respectively.
Let $A_i\in\Re^{2\times n^2}$, $i=1,2,\ldots,n^2$, be the corresponding first-order local forward finite difference operator at the $i$th pixel, i.e., each $A_i$ is a two-row matrix formed by stacking the $i$th rows of $A^{(1)}$ and $A^{(2)}$.
Let $x^*\in \Re^{n^2}$ be an original $n$-by-$n$ image, whose columns are stacked in an left-upper and right-lower order to form a vector of length $n^2$. Our discussions can be applied to rectangle images, and here we concentrate on square images only for simplicity.
Given a set of linear measurements $b = \calB x^* \in\Re^q$, where $\calB: \Re^{n^2} \rightarrow \Re^q$  is a linear operator.
The theory developed in \cite{NW13a} guarantees that one can reconstruct $x^*$ from $\calB$ and $b$ to within certain high accuracy, as long as $\calB$ satisfies certain technical conditions.
Specifically, to reconstruct $x^*$ from $\calB$ and $b$, one seeks an image that fits the observation data and meanwhile has the minimum TV, i.e.,
a solution of the following TV minimization problem
\begin{equation}
  \label{prob:tvcs}
  \min_{x\in\Re^{n^2}} \iota_{\{x: \; \calB x = b\}}(x) + \sum\nolimits_{i=1}^{n^2} \|A_ix\|.
\end{equation}
Here $\iota_S(x)$ denotes the indicator function of a set $S$, i.e., $\iota_S(x)$ is equal to $0$ if $x\in S$ and $\infty$ otherwise.
For $u_j\in \Re^{n^2}$, $j=1,2$, we define
\begin{equation*}
  \label{def-u-bui}
u := \left(
      \begin{array}{c}
        u_1 \\
        u_2 \\
      \end{array}
    \right) \in \Re^{2n^2},
\quad
\bu_i := \left(
          \begin{array}{c}
            (u_1)_i \\
            (u_2)_i \\
          \end{array}
        \right)\in\Re^2, \; i=1,2,\ldots,n^2,
\quad
A := \left(
       \begin{array}{c}
         A^{(1)} \\
         A^{(2)} \\
       \end{array}
     \right)\in\Re^{2n^2 \times n^2}.
\end{equation*}
Note that $u = (u_1,u_2)$ and $\{\bu_i: i=1,2,\ldots,n^2\}$ denote the same set of variables.
Let $f: \Re^{n^2} \rightarrow (-\infty,\infty]$ and $g: \Re^{2n^2} \rightarrow (-\infty,\infty)$  be, respectively, defined as
\begin{subequations}
  \label{def-gf-tvcs}
\begin{eqnarray}
\label{def-gf-tvcs-f}
  f(x) &:=&  \iota_{\{x: \; \calB x = b\}}(x), \; x\in\Re^{n^2},\\
\label{def-gf-tvcs-g}
  g(u) &:=&  g(u_1,u_2) = \sum_{i=1}^{n^2} \|\bu_i\|, \; u = (u_1, u_2) \in\Re^{2n^2}.
\end{eqnarray}
\end{subequations}
Then,  \eqref{prob:tvcs} can be rewritten as $  \min_{x\in\Re^{n^2}}  f(x) + g(Ax)$, which is clearly in the form of (P$_1$).
Let $\calB^*$ be the adjoint operator of $\calB$ and ${\cal I}$ be the identity operator.
In our experiments, the linear operator $\calB$ satisfies $\calB\calB^* = {\cal I}$.
Therefore, the proximity operator of $f$ is given by
\begin{equation}\label{prox-f-tvcs}
   \text{prox}^f(x) =  x + \calB^*(b - \calB x), \; x\in\Re^{n^2}.
\end{equation}
Note that the proximity operator of an indicator function reduces to the orthogonal projection onto the underlying set. The proximity parameter is omitted because it is irrelevant in this case.
On the other hand, with the convention $0/0=0$,  the proximity operator of ``$\|\cdot\|$" is given by
\begin{equation}\label{prox-||}
\text{prox}_{\eta}^{\|\cdot\|}(\bu_i) = \max\left\{\|\bu_i\| - \eta, 0\right\} \times \frac{\bu_i}{\|\bu_i\|}, \; \bu_i\in\Re^2, \; \eta >0.
\end{equation}
Furthermore, it is easy to observe from \eqref{def-gf-tvcs-g} that $g$ is separable with respect to $\bu_i$ and thus the proximity operator of $g$ can also be expressed explicitly.
Therefore, the functions $f$ and $g$ defined in \eqref{def-gf-tvcs} satisfy Assumption \ref{as:f-g-simple}.
As a result, CPAs and the proposed inertial CPAs are easy to implement.

\subsection{Experimental data}
In our experiments, the linear operator $\calB$ is set to be randomized partial Walsh-Hadamard transform matrix, whose rows are randomly chosen and columns randomly permuted.
Therefore, it holds that $\calB\calB^* = {\cal I}$.
Specifically, the Walsh-Hadamard transform matrix of order $2^j$ is defined recursively as
\[
H_{2^0} = [1],
H_{2^1} = \left[
            \begin{array}{cc}
              1 & 1 \\
              1 & -1 \\
            \end{array}
          \right],
\ldots,
H_{2^j} = \left[
            \begin{array}{cc}
              H_{2^{j-1}} & H_{2^{j-1}} \\
              H_{2^{j-1}} & -H_{2^{j-1}} \\
            \end{array}
          \right].
\]
It can be shown that $H_{2^j}H_{2^j}^\top = I$. In our experiments, the linear operator $\calB$ contains random selected rows from $2^{j/2}H_{2^j}$, where $2^{j/2}$ is a normalization factor.
It is worth to point out that for some special linear operators, e.g., $\calB$ is a partial Fourier matrix or a partial discrete cosine transform, \eqref{prob:tvcs} (and its denoising  variants  when the observation data contains noise) can be solved
by the classical ADM framework \eqref{ADM-xuy} without proximal-linearizing any of the subproblems, as long as the constraints are wisely treated and the finite difference
operations are assumed to satisfy appropriate boundary conditions. In these cases, the $x$-subproblem can usually be solved by fast transforms, see, e.g., \cite{NCT99,WYYZ08,YZY10}.
In our setting, the matrices $A^\top A$ and $\calB^*\calB$ cannot be diagonalized simultaneously, no matter what boundary conditions are assumed for $A$. Therefore, when solving the problem by the classical ADM, the $x$-subproblem is not easily solvable, no matter how the constraints $\calB x = b$ are adapted (e.g., penalization or relaxation).
In contrast, when solving the problem by CPAs, no linear system needs to be solved and the algorithms are easily implementable as long as the proximity operators of the underlying functions can be efficiently evaluated.

We tested 12 images, most of which are obtained from the USC-SIPI image database\footnote{\url{http://sipi.usc.edu/database/}}. The image sizes are $256$-by-$256$, $512$-by-$512$ and $1024$-by-$1024$, each of which contains 4 images. The tested images, together with their names in the database, are given in Figure \ref{Fig1-images}.

\begin{figure}[htbp]
\centering{
\includegraphics[trim = 0 240 50  160, scale = 0.3]{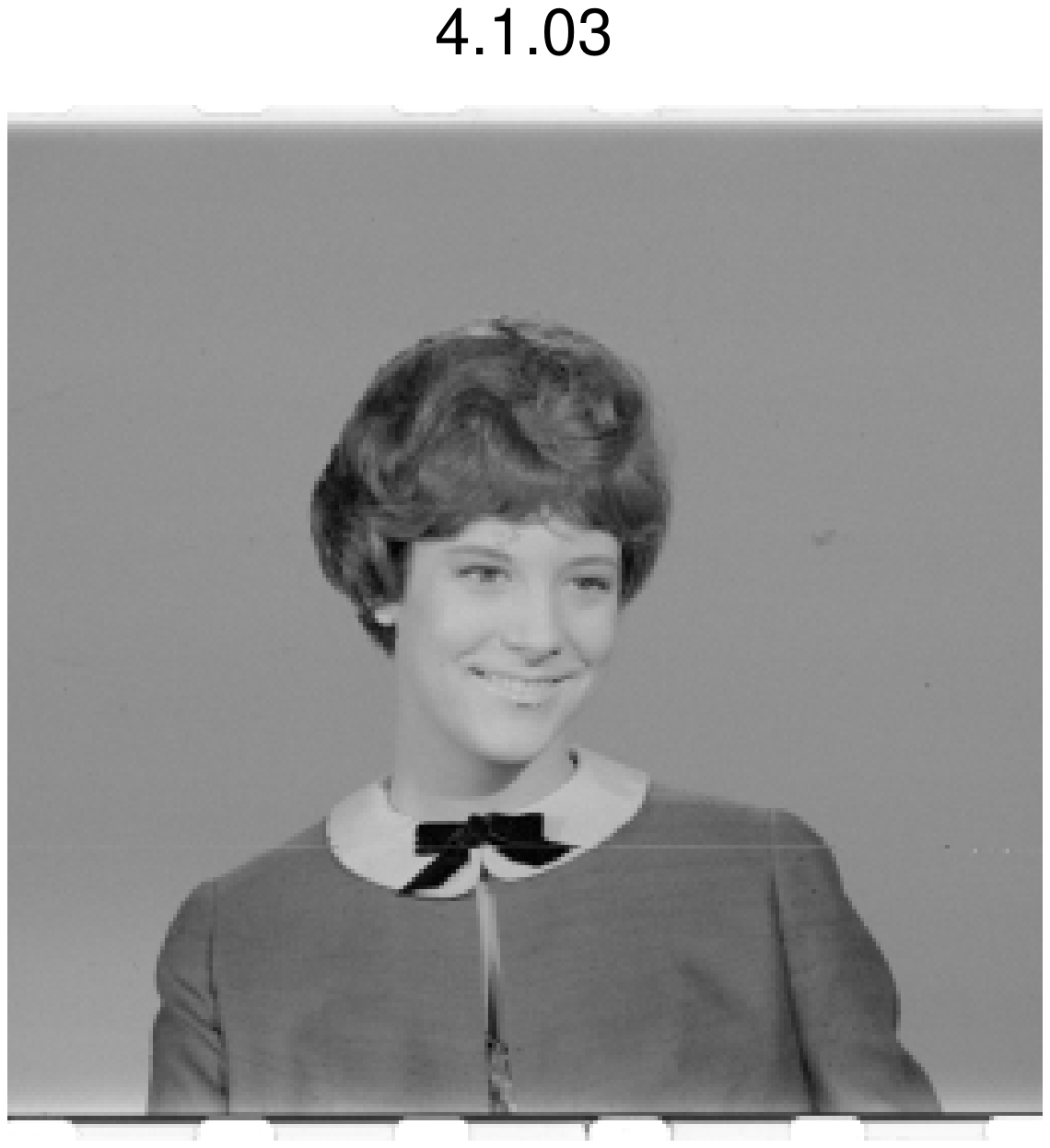}\hspace{-2.62cm}
\includegraphics[trim = 0 240 50  160, scale = 0.3]{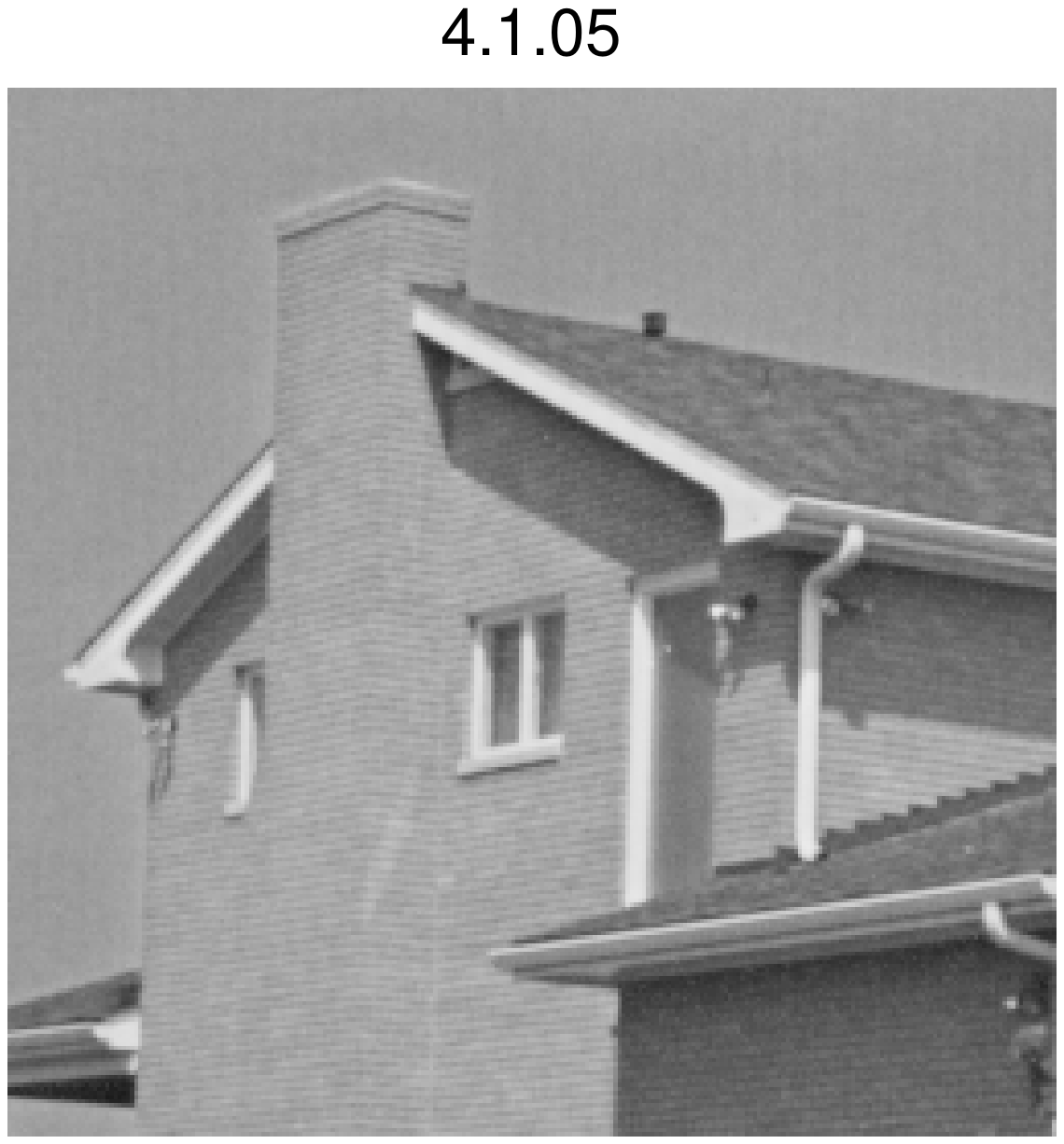}\hspace{-2.62cm}
\includegraphics[trim = 0 240 50  160, scale = 0.3]{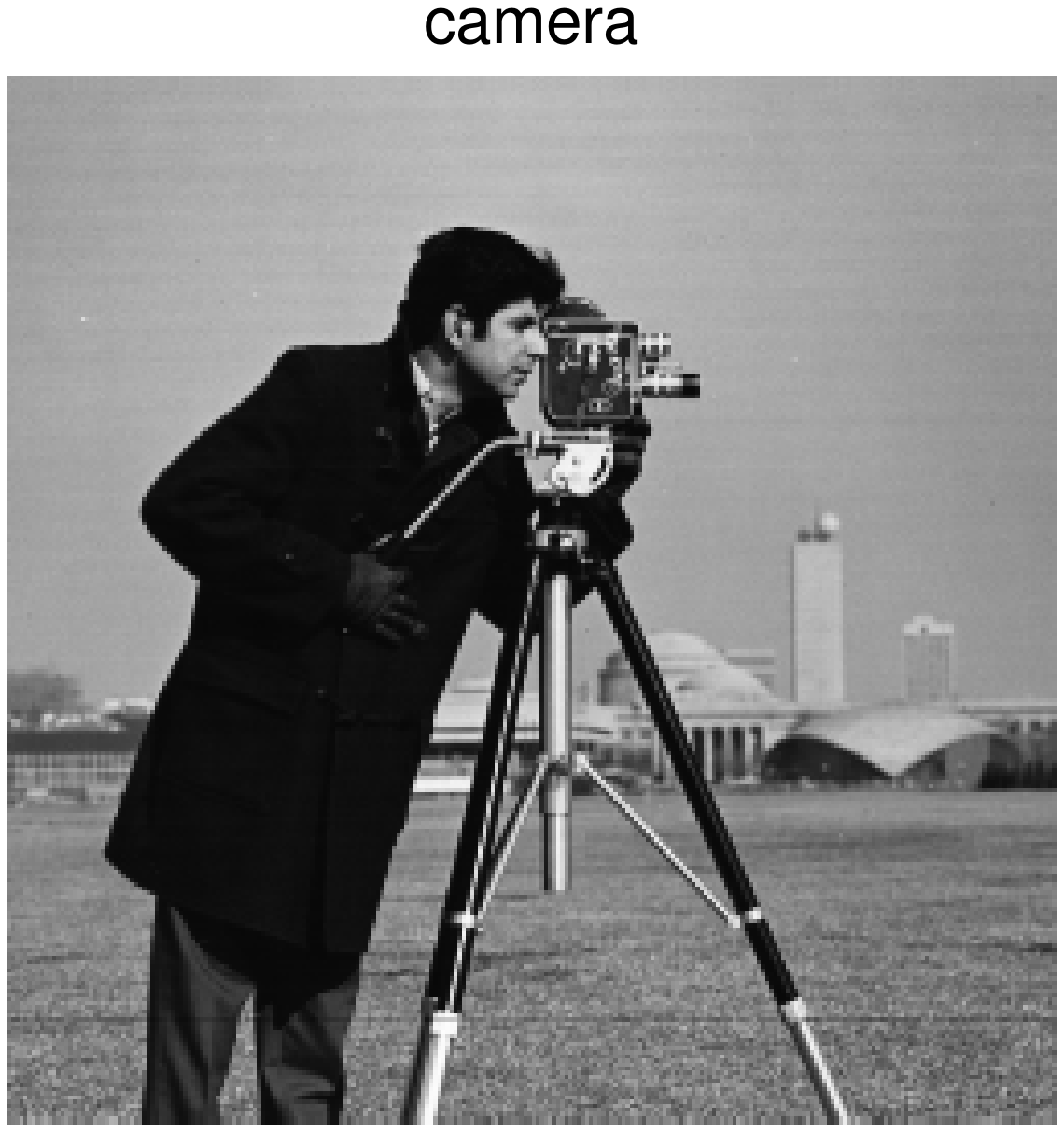}\hspace{-2.62cm}
\includegraphics[trim = 0 240 50  160, scale = 0.3]{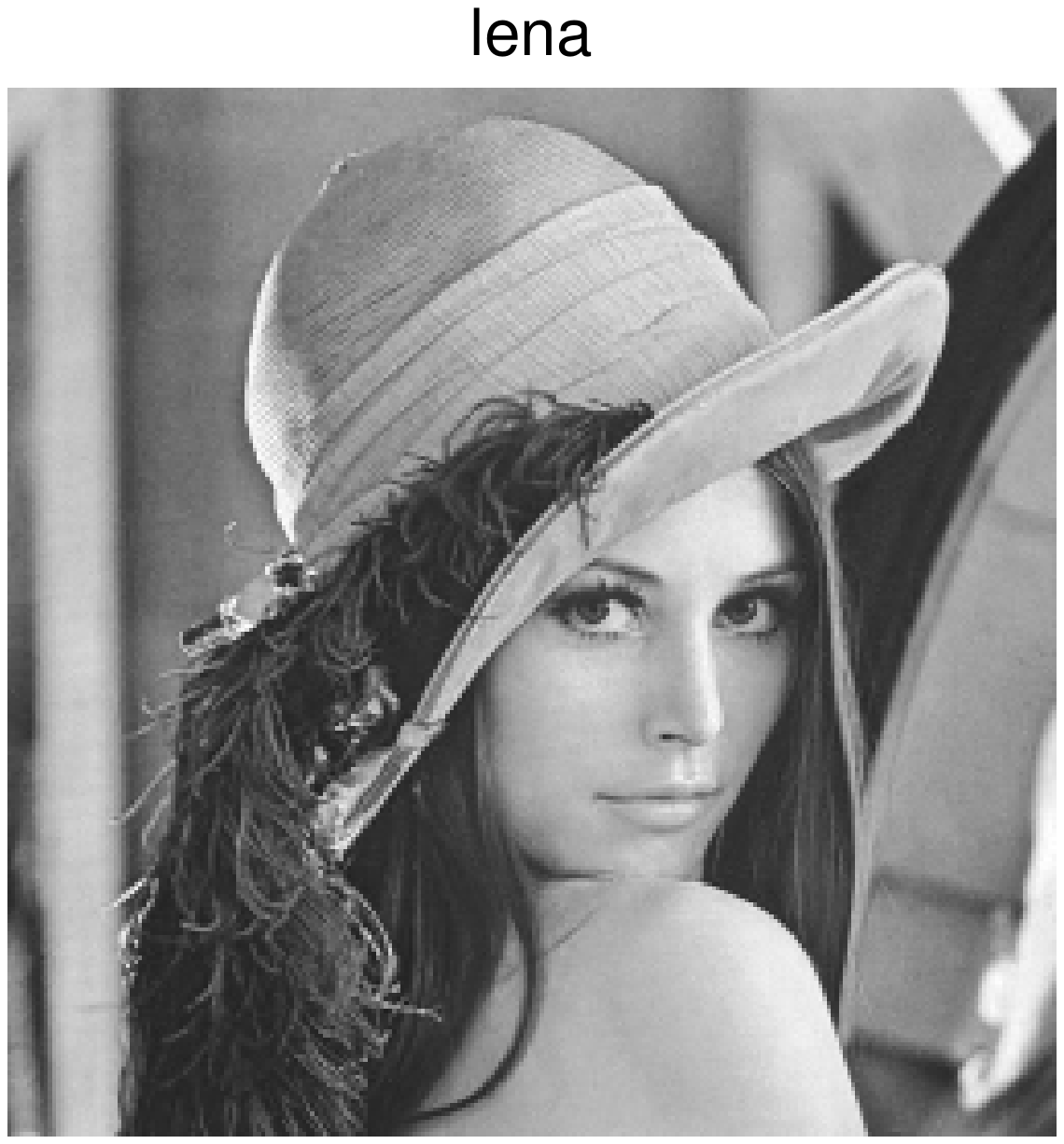}\\
\includegraphics[trim = 0 200 50  160, scale = 0.27]{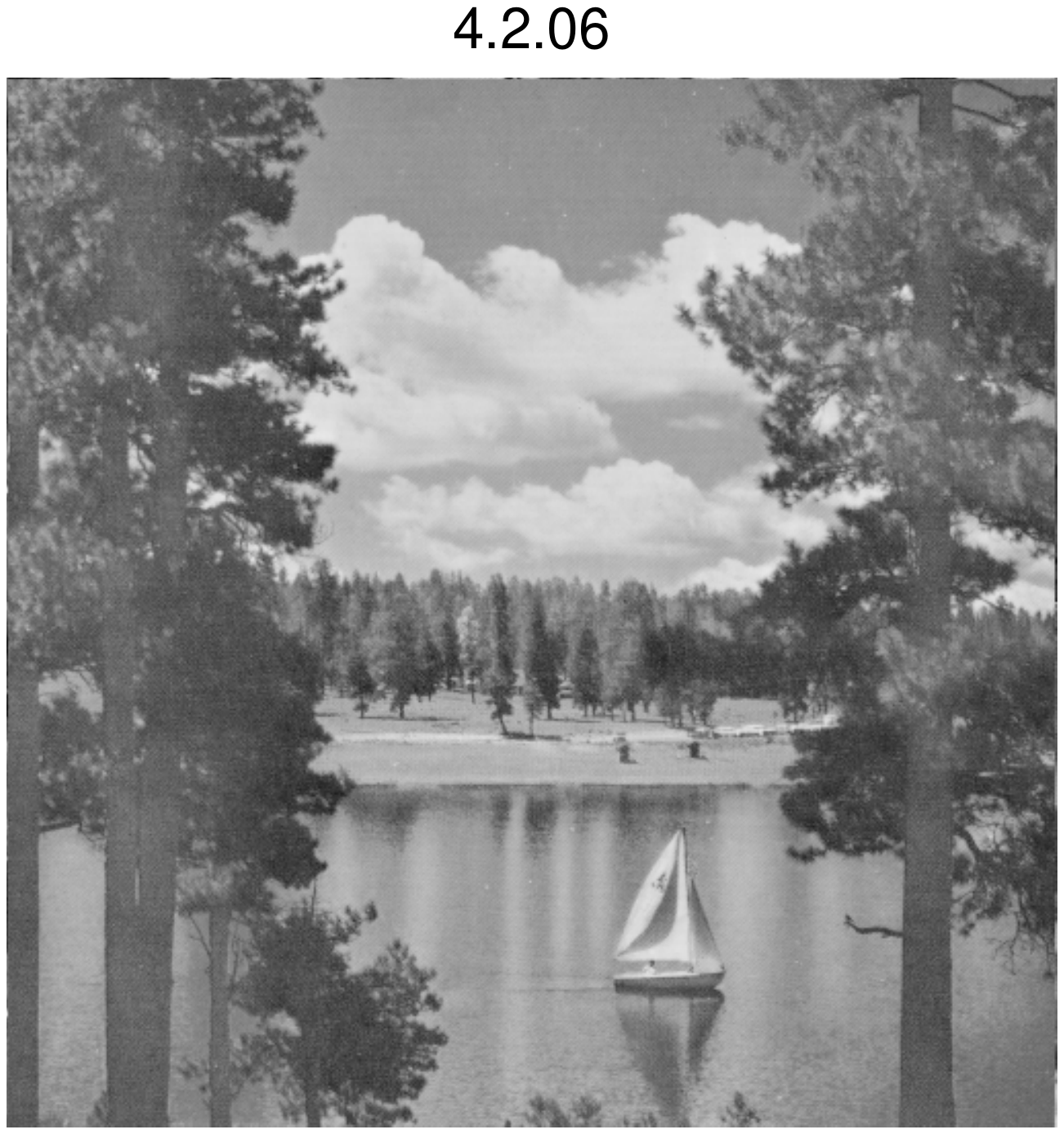}\hspace{-2cm}
\includegraphics[trim = 0 200 50  160, scale = 0.27]{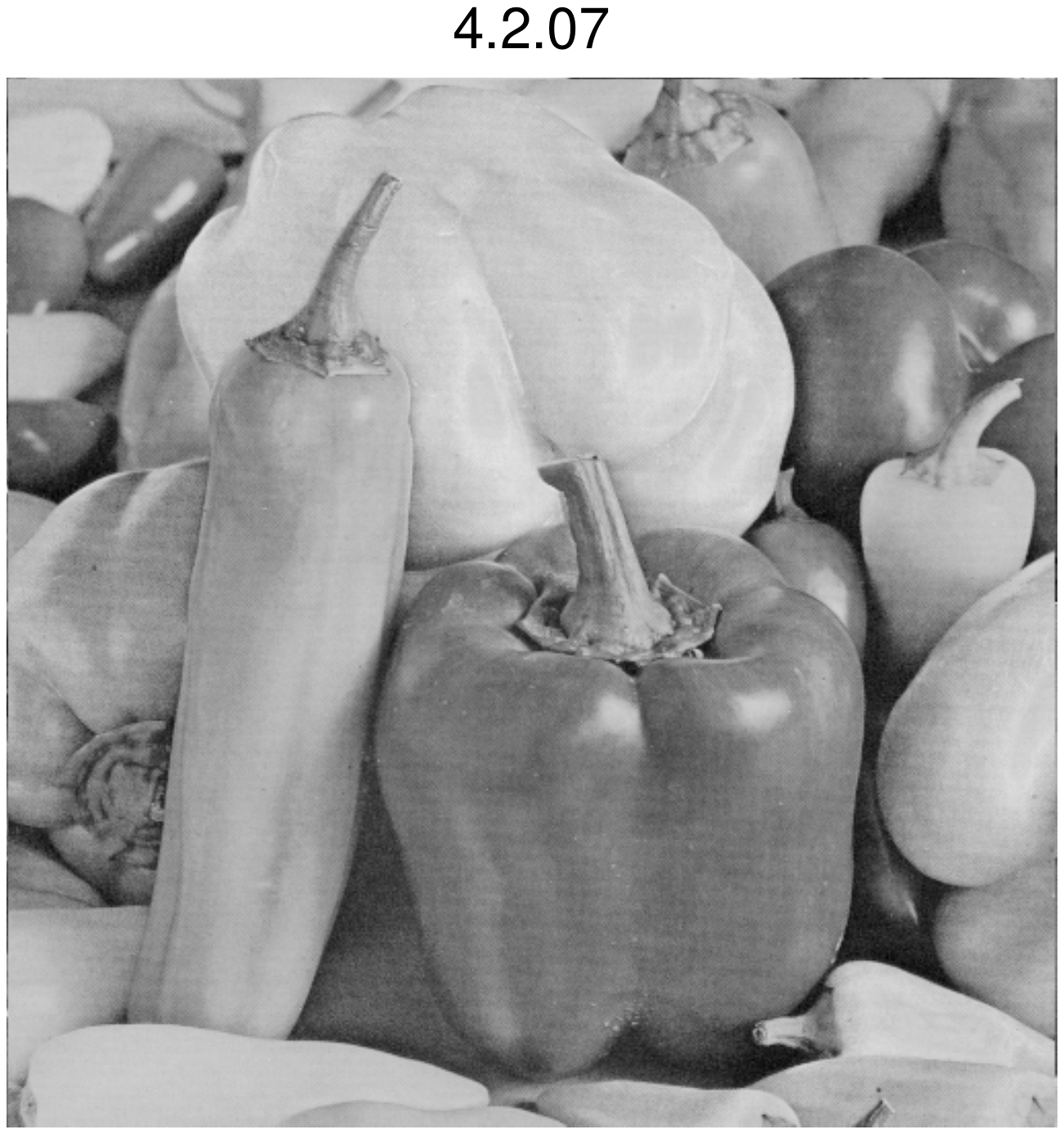}\hspace{-2cm}
\includegraphics[trim = 0 200 50  160, scale = 0.27]{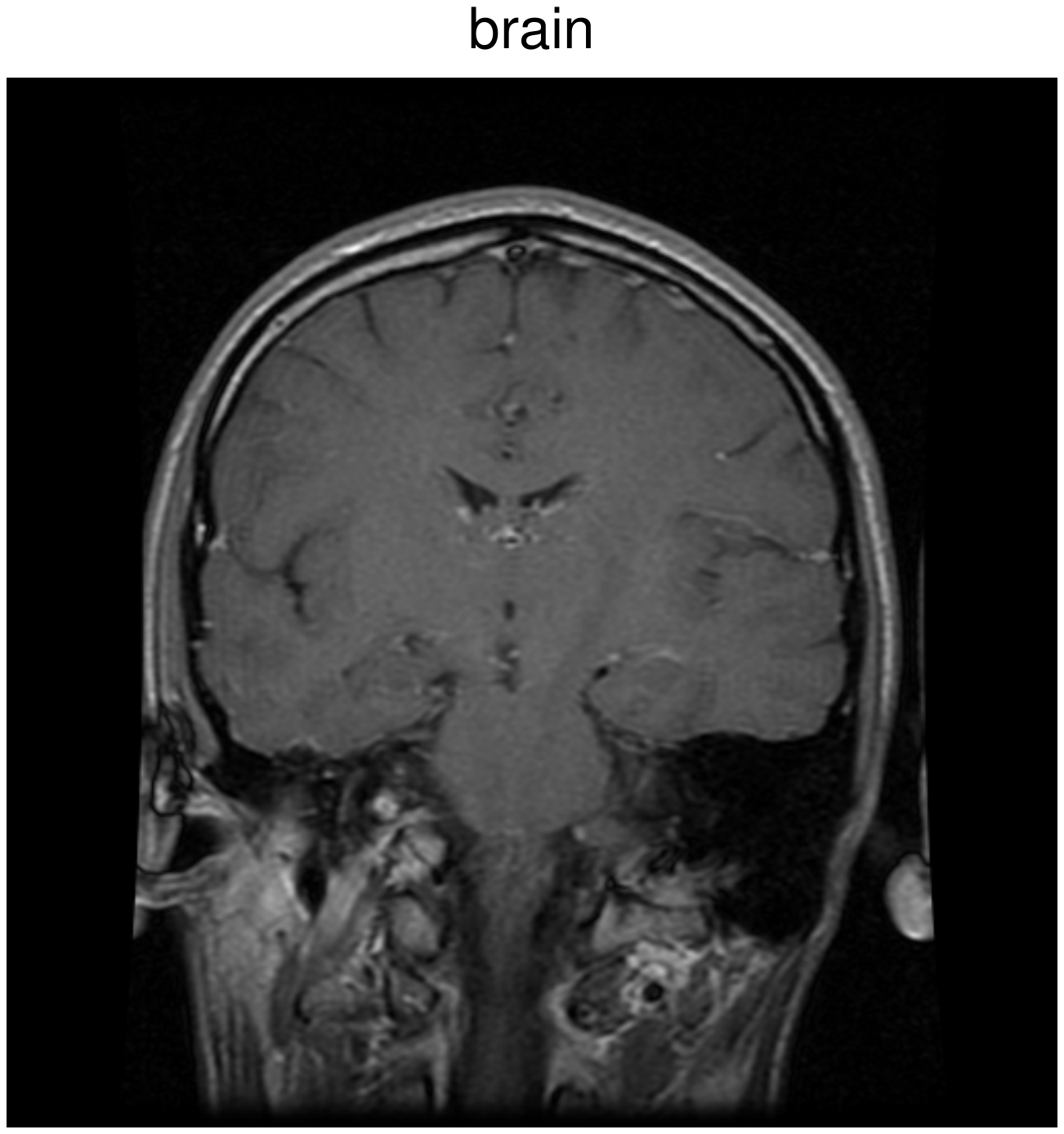}\hspace{-2cm}
\includegraphics[trim = 0 200 50  160, scale = 0.27]{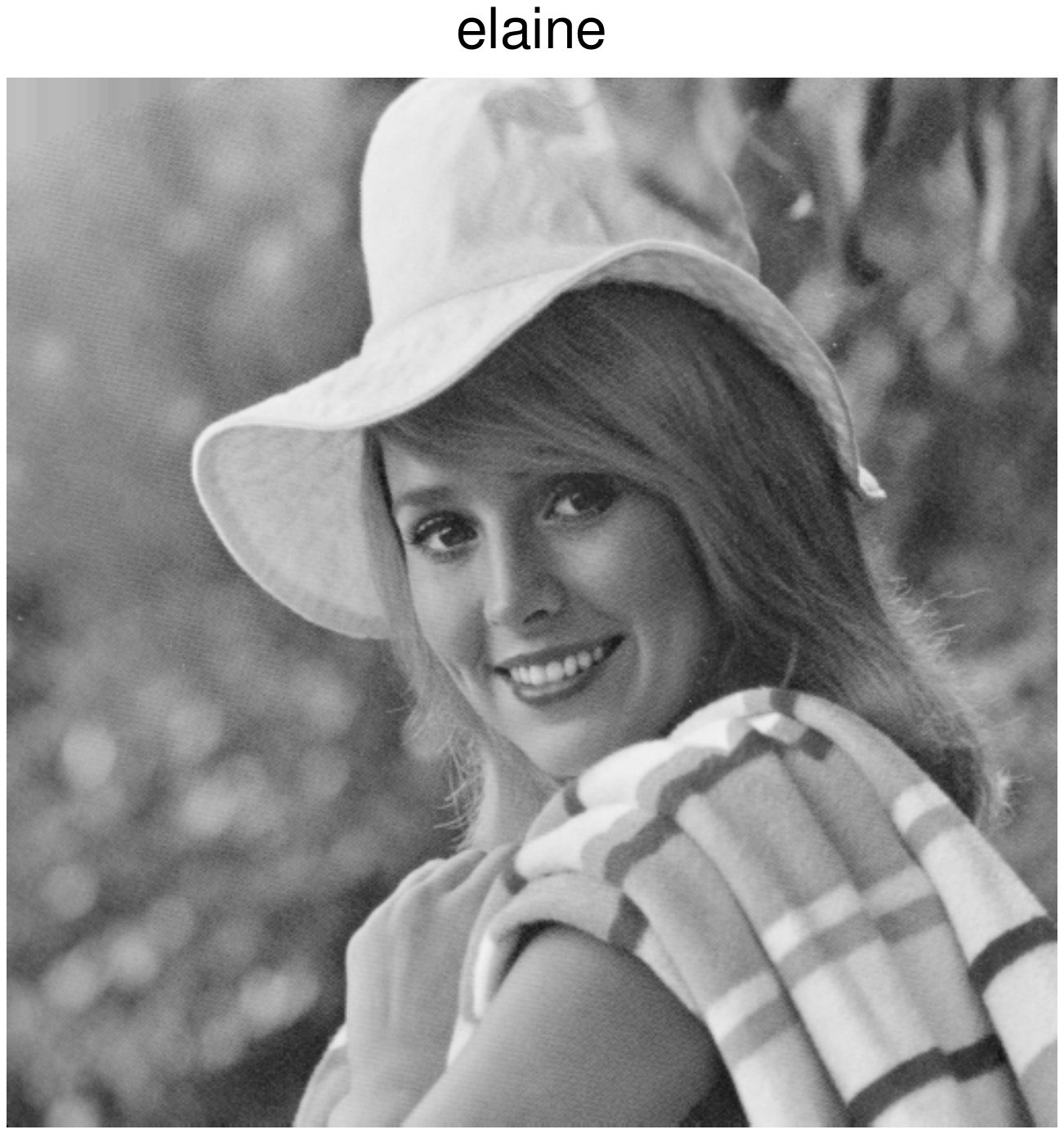}\\
\includegraphics[trim = 0 200 50  160, scale = 0.27]{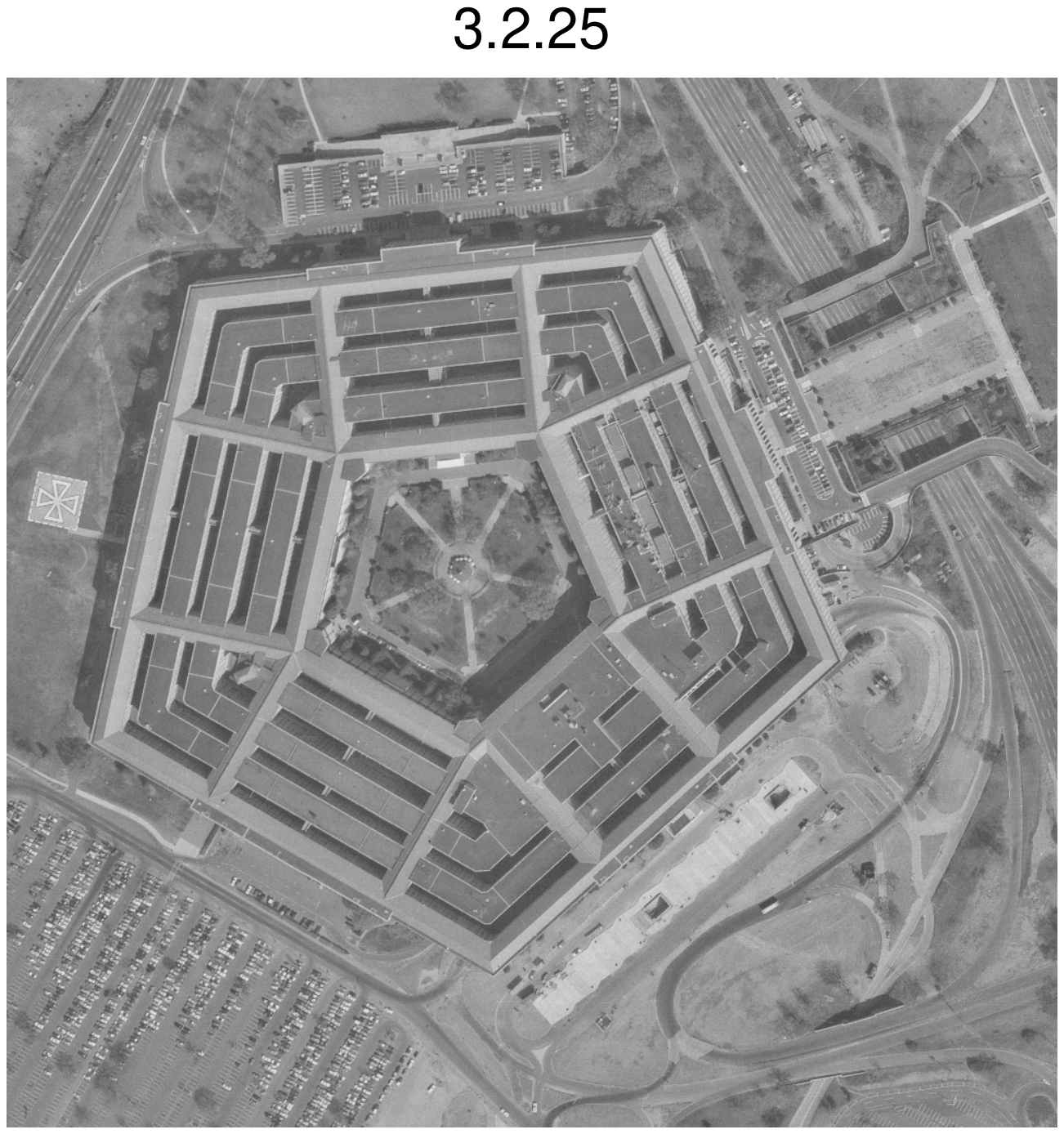}\hspace{-2cm}
\includegraphics[trim = 0 200 50  160, scale = 0.27]{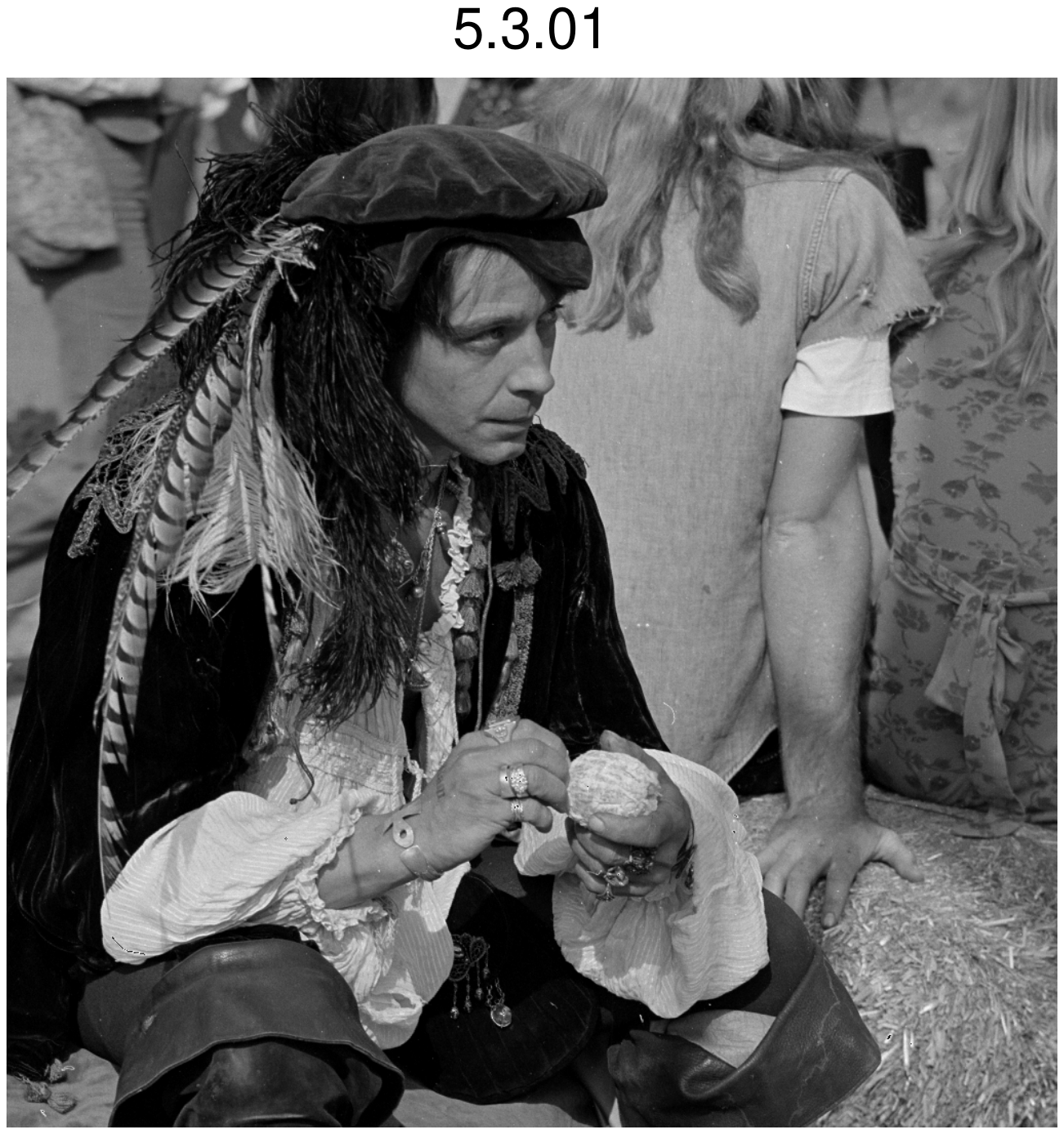}\hspace{-2cm}
\includegraphics[trim = 0 200 50  160, scale = 0.27]{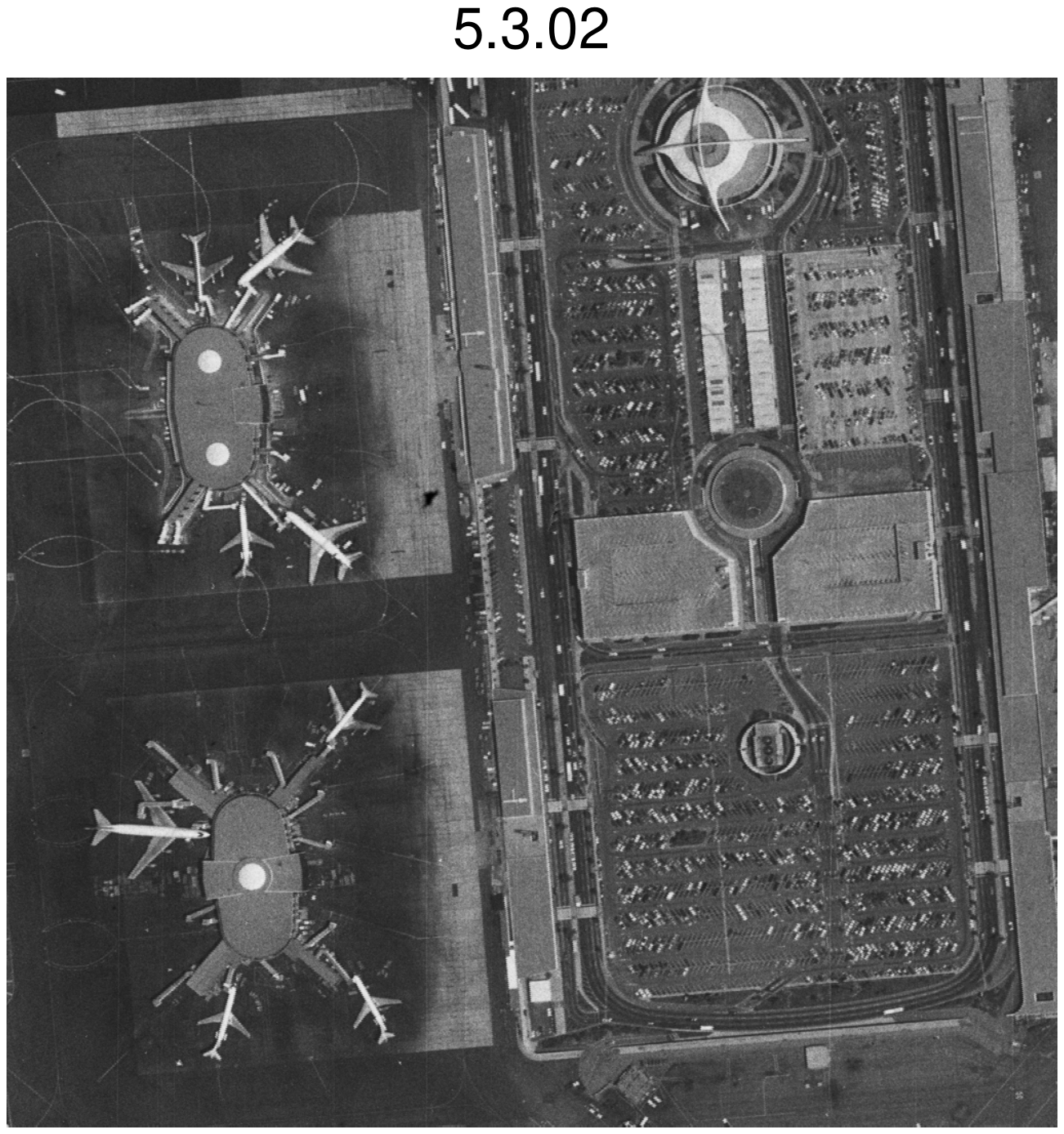}\hspace{-2cm}
\includegraphics[trim = 0 200 50  160, scale = 0.27]{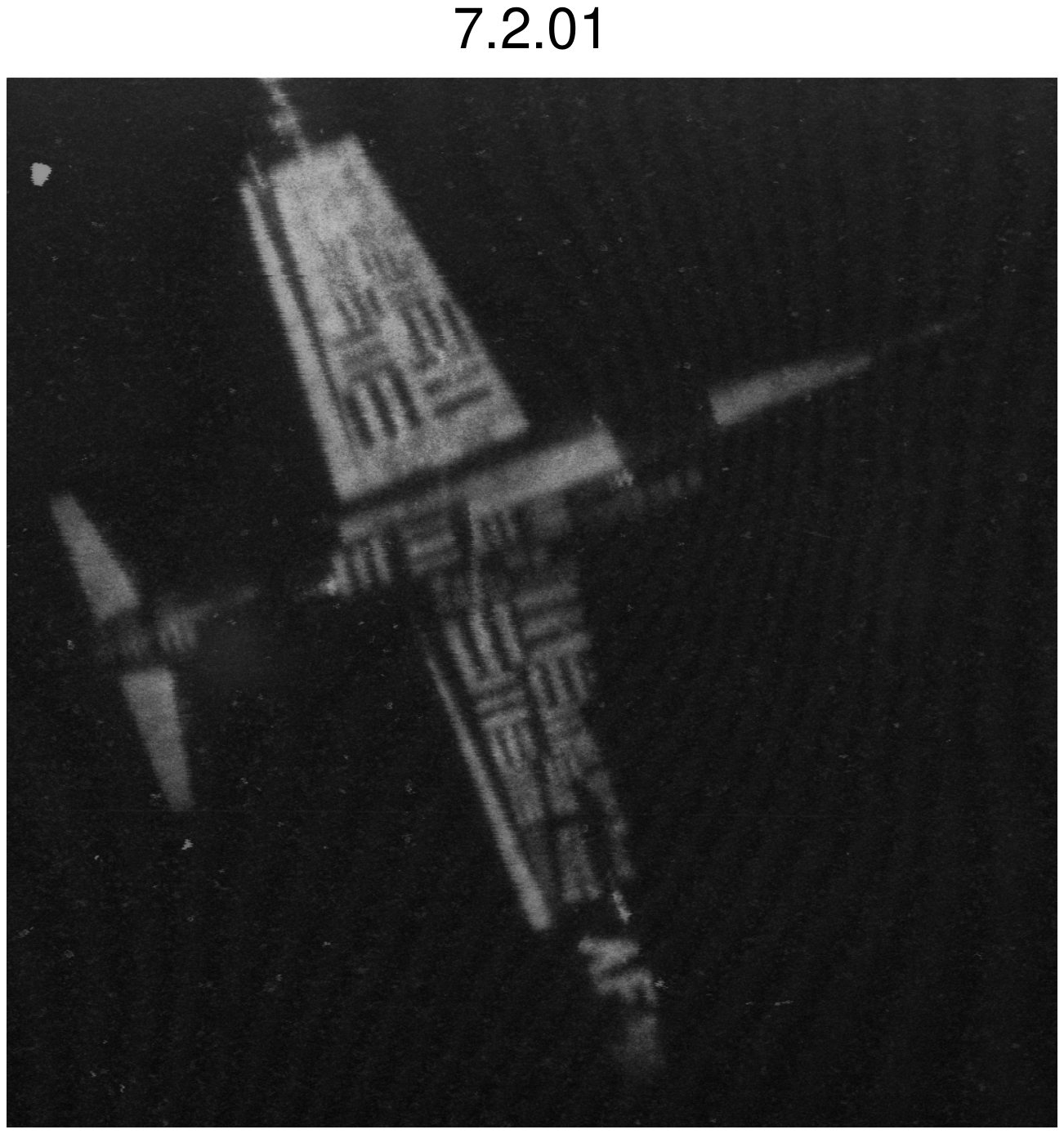}
}\caption{Tested images from the USC-SIPI image database. The image sizes from the first to the third row are $256 \times 256$, $512 \times 512$ and $1024 \times 1024$, respectively. }
\label{Fig1-images}
\end{figure}

\subsection{Parameters, initialization, stopping rules, etc}\label{sc:num-parameters}
The parameters common to CPAs  and their corresponding inertial CPAs are $\sigma$ and $\tau$, for which we used the same set of values.
In our experiments, periodic boundary conditions are assumed for the finite difference operations.  It is easy to show that $\rho(A^\top A)=8$.
The parameters $\sigma$ and $\tau$ were set to be $5$ and $0.124/\sigma$ uniformly for all tests, which may be suboptimal but perform favorably for imaging problems with appropriately scaled data. In particular, this setting satisfies the convergence requirement of all algorithms.
The extrapolation parameter $\alpha_k$ for inertial CPAs was set to be $0.28$ and held constant. This value of $\alpha_k$ is determined based on experiments. How to select $\alpha_k$ adaptively to achieve faster convergence remains a research issue. Here our main goal is to illustrate the effect of the extrapolation steps.
We will take iCP-$y\bar yx$ as an example and present some experimental results to compare its performance with different constant values of $\alpha_k$.
In our experiments, we initialized $x^0 = \calB^* b$ and $y^0 = 0$ for all algorithms.
It is clear from \eqref{LADM-mVI-k+1} that if $x^{k+1}=x^k$ and $y^{k+1}=y^k$ then a solution is already obtained.
Thus,  we terminated CPAs by 
\begin{equation}\label{stop-rule}
  \frac{\|(x^{k+1},y^{k+1}) - (x^k,y^k)\|}{1 + \|(x^k,y^k)\|} < \varepsilon,
\end{equation}
where $\varepsilon > 0$ is a tolerance parameter, and $\|(x,y)\| := \sqrt{\|x\|^2 + \|y\|^2}$.
For inertial CPAs,
the same can be said, except that $(x^k,y^k)$ needs to be replaced by $(\hat x^k,\hat y^k)$.
Thus, we terminated inertial CPAs by
\begin{equation}\label{stop-rule-iCPA}
  \frac{\|(x^{k+1},y^{k+1}) - (\hat x^k,\hat y^k)\|}{1 + \|(\hat x^k,\hat y^k)\|} < \varepsilon.
\end{equation}
The quantities in \eqref{stop-rule} and \eqref{stop-rule-iCPA} can be viewed as optimality residues (in a relative sense).
The tolerance parameter $\varepsilon$ will be specified below.

The quality of recovered images is evaluated by signal-to-noise ratio (SNR), which is defined as
\begin{equation}\label{def:SNR}
\text{SNR} := 20\times \log_{10} \frac{\|\tilde x -x^*\|}{\|x - x^*\|}.
\end{equation}
Here $x^*$ and $x$ represent the original and the recovered images, and $\tilde x$ denotes the mean intensity of $x^*$.
Note that the constraint $\calB x = b$ is always preserved at each iteration and for all algorithms.
Therefore, we only report the  objective function value $\sum_i\|A_ix\|$, denoted by $\text{TV}(x)$, but not the data fidelity $\|\calB x - b\|$.

\subsection{On the performance of different algorithms}\label{sc:numerical-reason}
Recall that CP-$xy\bar y$ (resp. CP-$yx\bar x$) and CP-$y\bar yx$ (resp. CP-$x\bar xy$) are cyclically equivalent. Therefore, they generate exactly the same sequence of points as long as the initial points are properly chosen.  As a result, we only need to concentrate on CP-$y\bar yx$ and CP-$x\bar xy$, and compare with their corresponding inertial variants.
It is interesting that we have observed from our extensive experimental results on total variation based image reconstruction problems that CP-$y\bar yx$ and CP-$x\bar xy$ perform almost identically as long as  the same set of parameters ($\tau$ and $\sigma$) and initial points $(x^0,y^0)$ are used.
In particular, CP-$y\bar yx$ and CP-$x\bar xy$ generate two sequences of points with very close  optimality residues, objective function values (i.e.,  $\text{TV}(x)$) and  SNRs (defined in \eqref{def:SNR}). After very few iterations, these quantities usually differ little.
Note that, instructed by \eqref{LADM-mVI-k+1}, we meassured the optimality residue by ${\bf res} := \|G(w^{k+1} - w^k)\|$ in our experiments, where $G$ is given by \eqref{def:G-yyx} and \eqref{def:G-xxy}, respectively, for  CP-$y\bar yx$ and CP-$x\bar xy$.

The similar performance of CP-$y\bar yx$ and CP-$x\bar xy$ seems to be reasonable by directly comparing the iteration formulas in \eqref{yyx} and \eqref{xxy}. Another plausible explanation of this phenomenon is as follows. Recall that both CP-$y\bar yx$ and CP-$x\bar xy$ are applications of a general PPM, i.e., they satisfy \eqref{LADM-mVI-k+1} with the weighting matrix given by \eqref{def:G-yyx} and \eqref{def:G-xxy}, respectively.
By comparing the two matrices, we see that they are different only in the signs of $A$ and $A^\top$. As a result, they have exactly the same spectrum, which  performs as the essential magnitudes of proximity.

For the  respective inertial algorithms, i.e., iCP-$y\bar yx$ and iCP-$x\bar xy$, we   observed the similar performance and the same remarks given above apply.

\begin{figure}[htbp]
\centering{
\includegraphics[trim = 0 240 50  160, scale = 0.28]{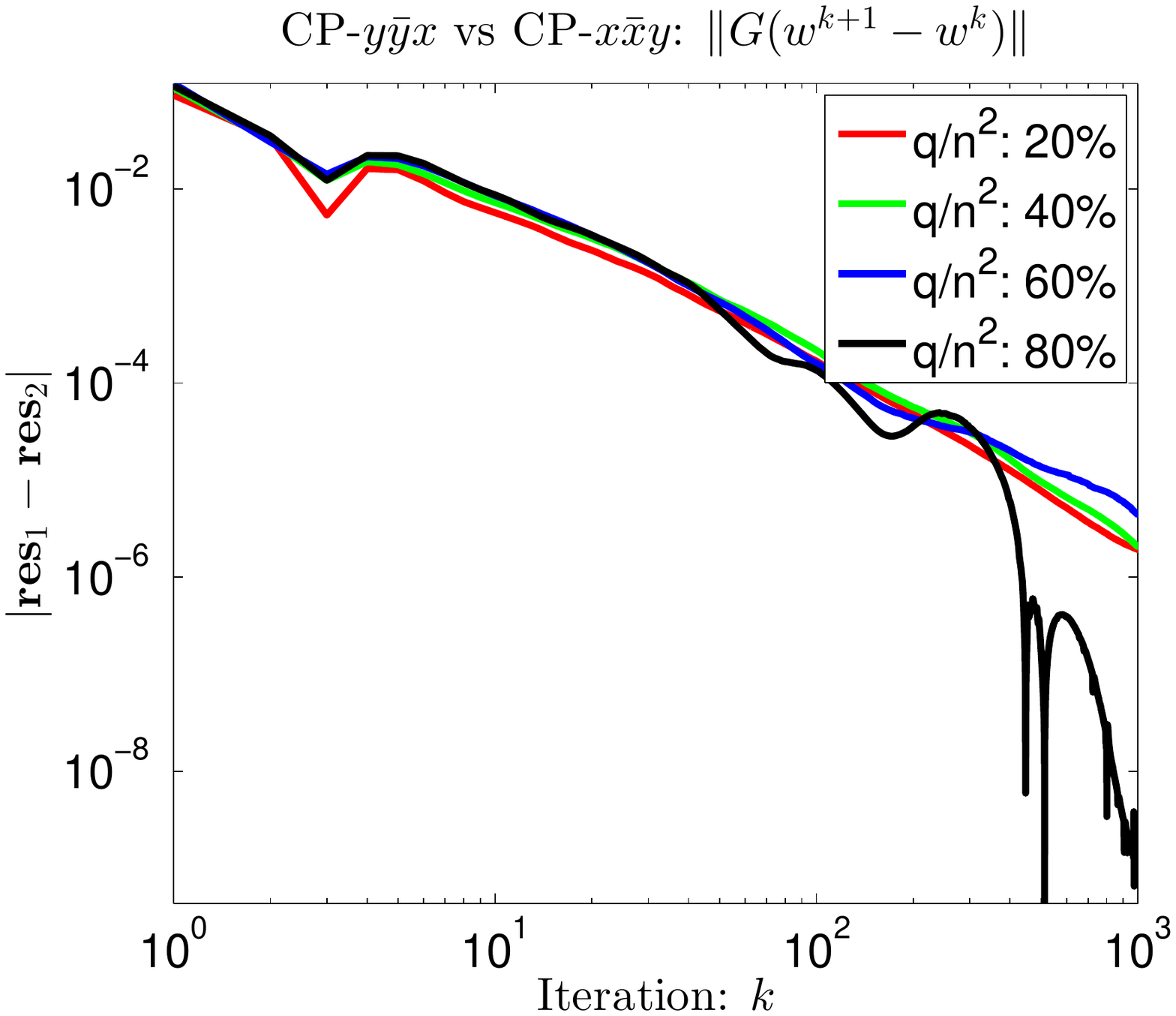}
\includegraphics[trim = 0 240 50  160, scale = 0.28]{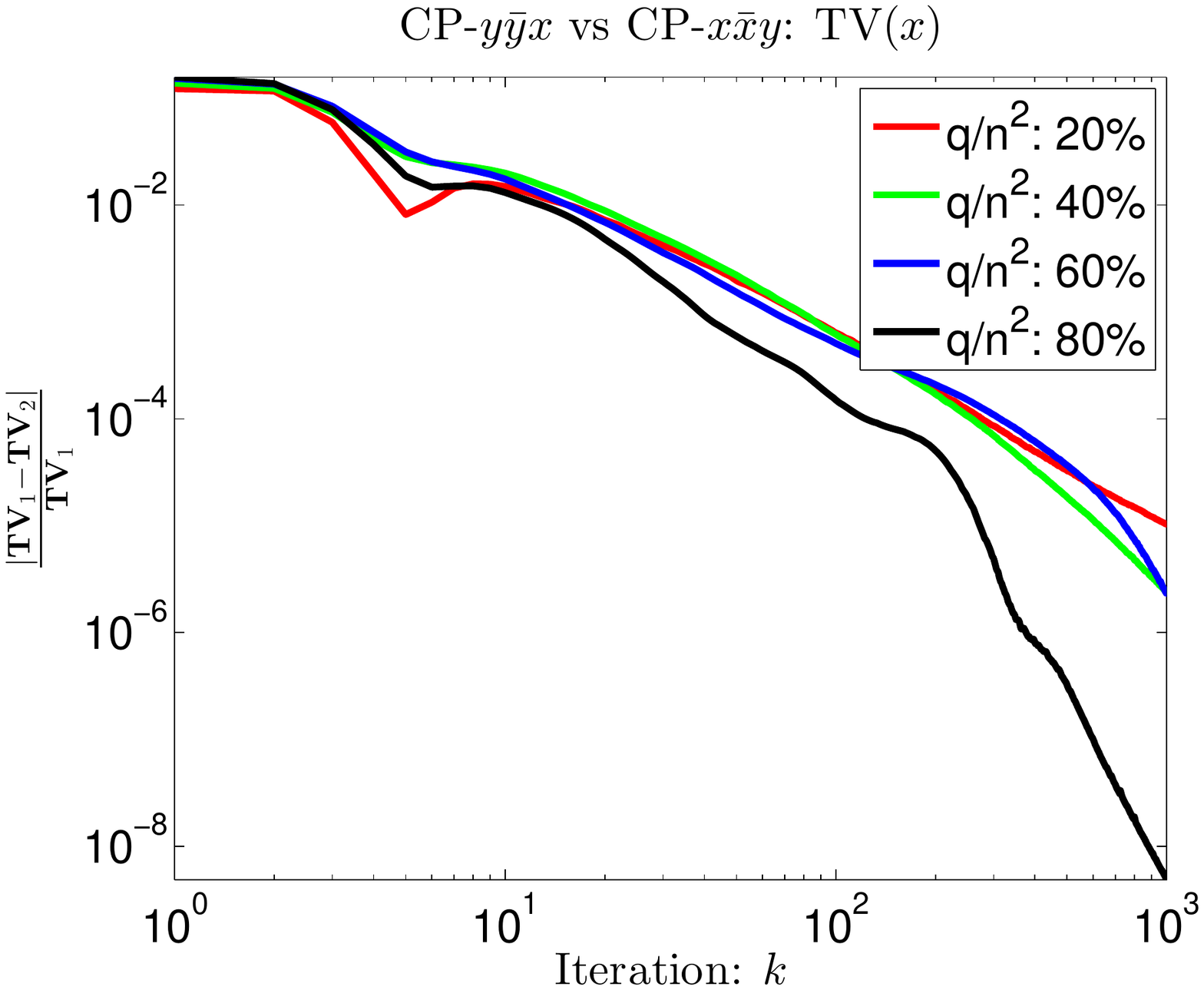}
\includegraphics[trim = 0 240 50  160, scale = 0.28]{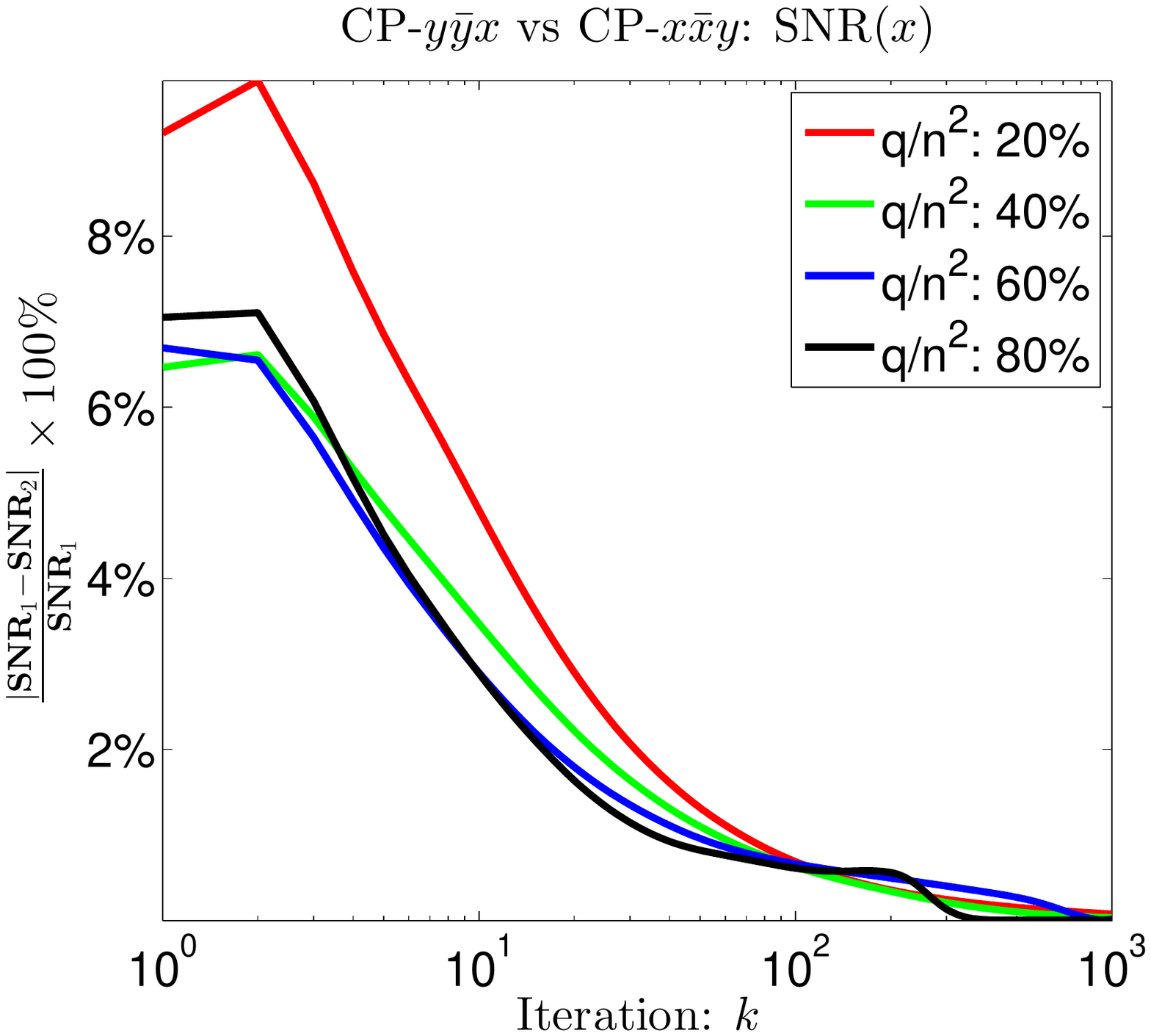}\\ \vspace{1cm}
\includegraphics[trim = 0 240 50  160, scale = 0.28]{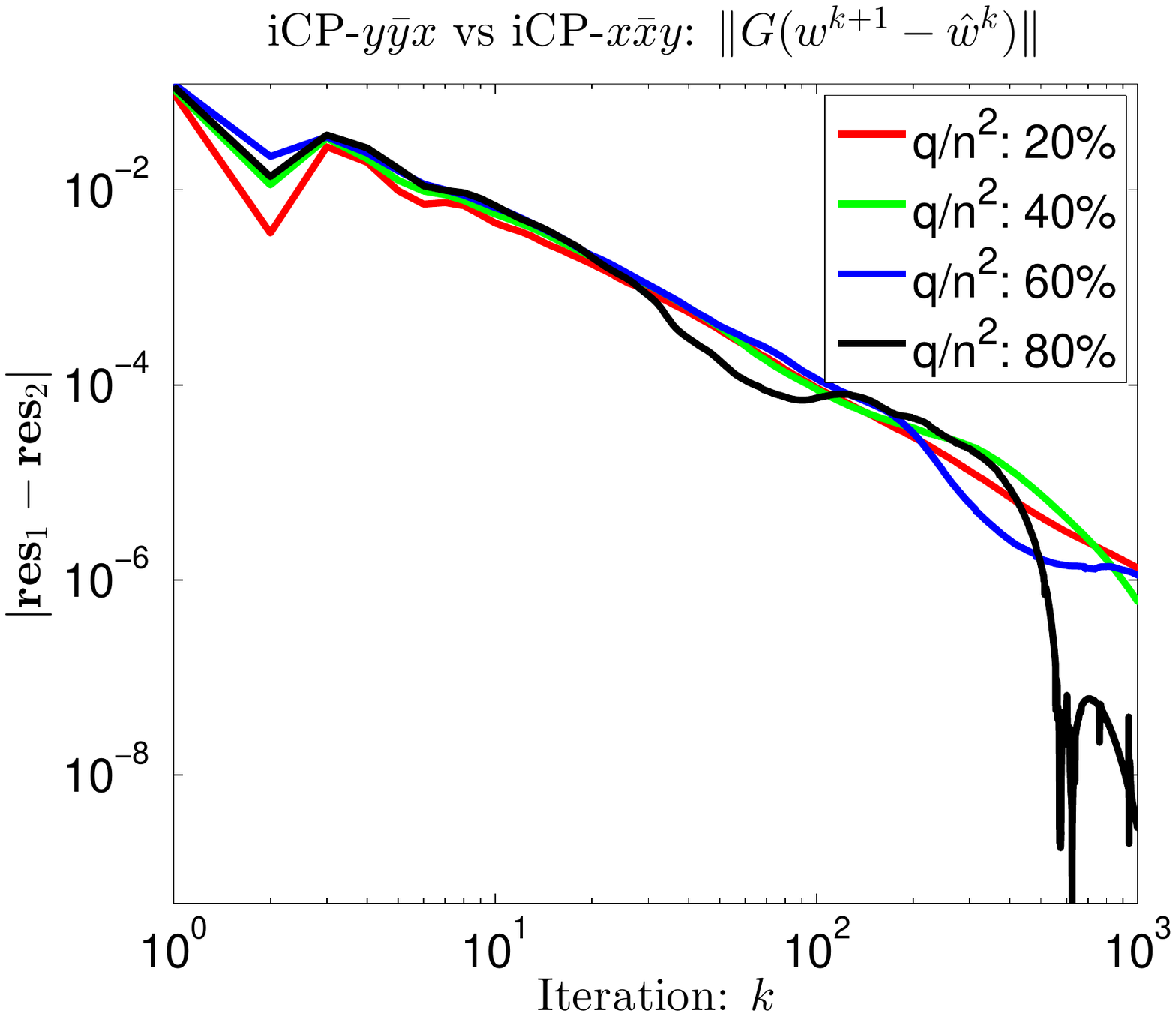}
\includegraphics[trim = 0 240 50  160, scale = 0.28]{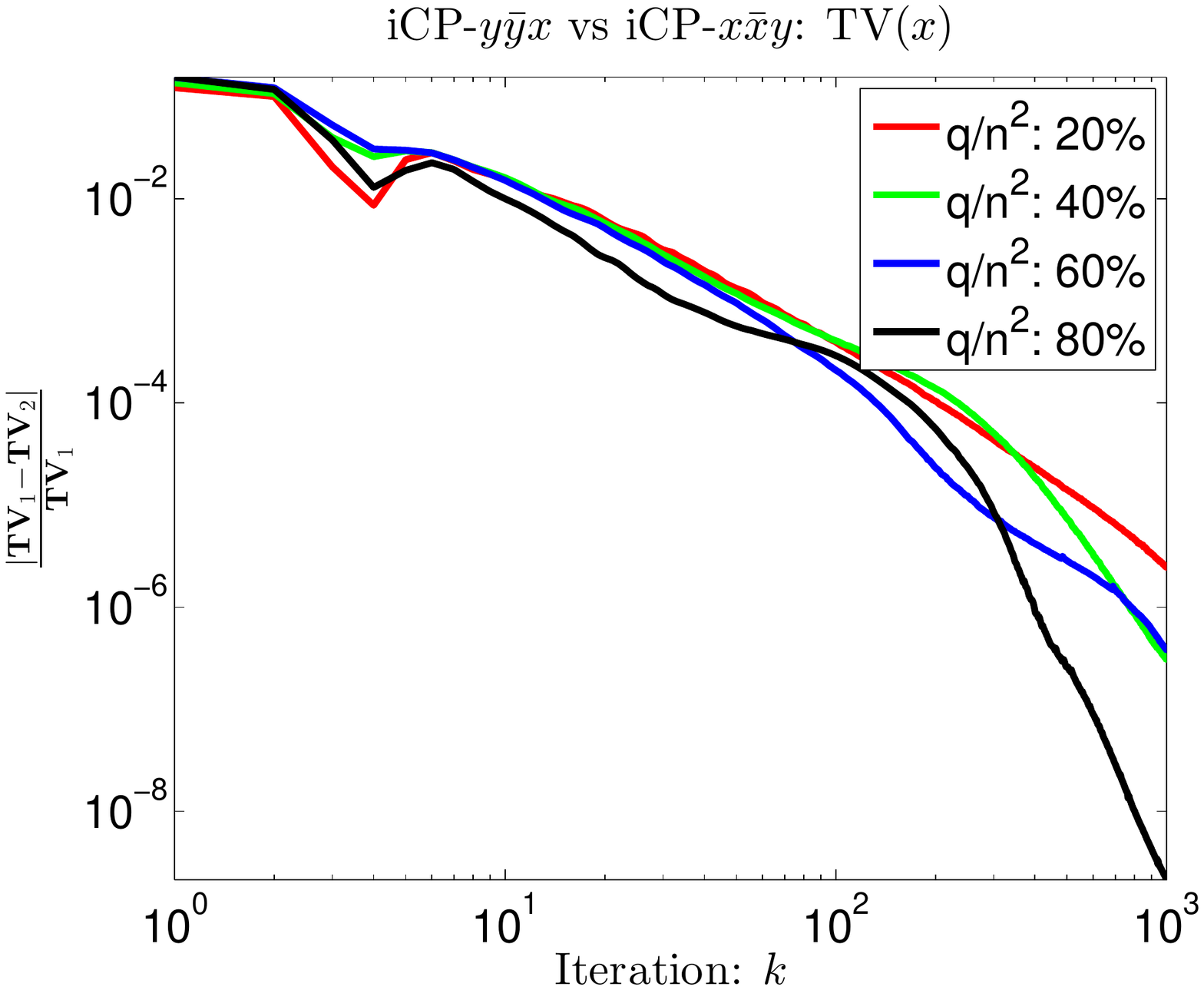}
\includegraphics[trim = 0 240 50  160, scale = 0.28]{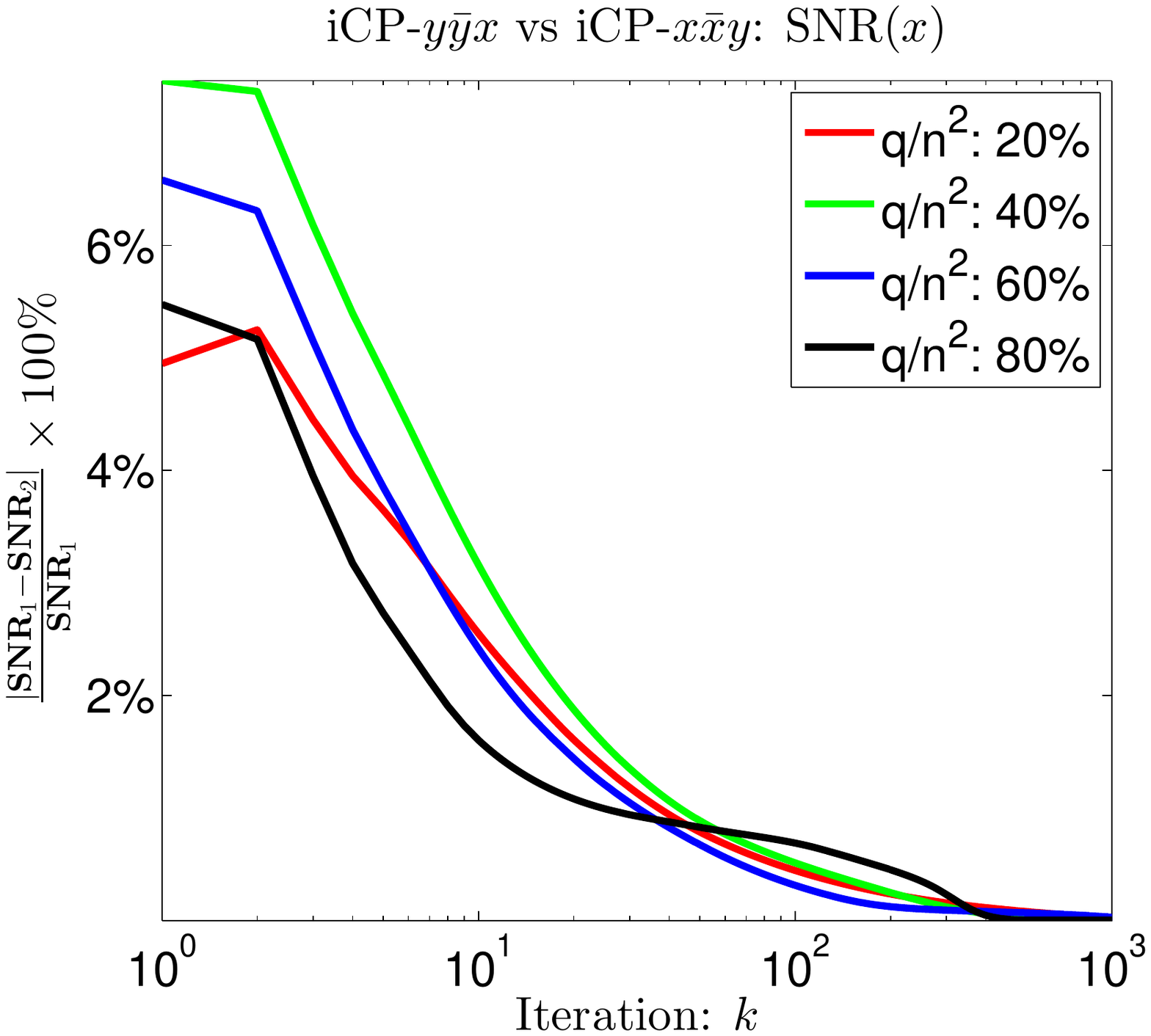} \\ \vspace{.5cm}
}\caption{Comparison results of CPAs and iCPAs. First row: CP-$y\bar yx$ vs CP-$x\bar xy$; Second row: iCP-$y\bar yx$ vs iCP-$x\bar xy$.}
\label{Fig2-CPyyx-CPxxy}
\end{figure}

For illustrative purpose, here we present some computational results on the cameraman image with four levels of measurements, i.e., $q/n^2 \in \{20\%,40\%,60\%,80\%\}$. We used the parameters specified in Section \ref{sc:num-parameters}, run each algorithm for 1000 iterations  and recorded the optimality residues, the objective function values and the SNRs of the generated sequences. These quantities  are denoted by ${\bf res}_i$, ${\bf TV}_i$ and ${\bf SNR}_i$, where $i=1$ for CP-$y\bar yx$ and iCP-$y\bar yx$, and $i=2$ for CP-$x\bar xy$ and iCP-$x\bar xy$.
Note that, for inertial CPAs, the optimality residue is defined as ${\bf res} := \|G(w^{k+1} - \hat w^k)\|$ with $G$ given by \eqref{def:G-yyx} and \eqref{def:G-xxy} for  iCP-$y\bar yx$ and iCP-$x\bar xy$, respectively. This definition is justified by \eqref{iLADM-mVI}.
The comparison results of CP-$y\bar yx$ and CP-$x\bar xy$ (resp. iCP-$y\bar yx$ and iCP-$x\bar xy$) on these quantities are given in the first (resp. the second) row of Figure \ref{Fig2-CPyyx-CPxxy}, where we presented the differences (in either absolute or relative sense) of these quantities.

It is easy to observe that the results given in Figure \ref{Fig2-CPyyx-CPxxy} basically justify our remarks given above in this subsection.
Therefore, in our experiments we only compare CP-$y\bar yx$ and its inertial version iCP-$y\bar yx$.
Experimental results for CP-$x\bar xy$ and iCP-$x\bar xy$ are not presented since they are similar to those of CP-$y\bar yx$ and iCP-$y\bar yx$, respectively.

\subsection{Experimental results}
Recall that the image size is denoted by $n\times n$, and the number of measurements is denoted by $q$. For each image, we tested four levels of measurements, that is $q/n^2 \in \{20\%,40\%,60\%,80\%\}$.
To implement the algorithms, a computation of the form ``$y \leftarrow \text{prox}_{\sigma}^{g^*}(y + \sigma Ax)$"  must be carried out at each iteration.
In implementation, this is completed by using the Moreau's decomposition \eqref{lem:Moreau-eq}, i.e., compute an intermediate variable $u$ first
as ``$u \leftarrow \text{prox}_{\sigma^{-1}}^{g^*}(\sigma^{-1}y +  Ax)$" and then recover $y$ via ``$y\leftarrow y + \sigma (Ax-u)$".
Here the quantity $\|u-Ax\|$ can be viewed as primal residue for the equivalent constrained formulation (P$_2$).
In the experimental results, besides SNR and objective function value, we also present this feasibility residue, measured by infinity norm $\|u - Ax\|_{\infty}$, and the number of iterations required by the algorithms (denoted, respectively, by It1 and It2 for CP-$y\bar yx$ and iCP-$y\bar yx$) to meet the condition \eqref{stop-rule} or \eqref{stop-rule-iCPA}.
We do not present the CPU time results for comparison because the per-iteration cost of the algorithms is roughly identical and the consumed CPU time is basically proportional to the respective number of iterations.
Detailed experimental results for $\varepsilon=10^{-2}, 10^{-3}$ and $10^{-4}$ are given in Tables \ref{table-tol=1e-2}-\ref{table-tol=1e-4}, respectively.
Note that in  Tables \ref{table-tol=1e-2}-\ref{table-tol=1e-4} the results for $\text{TV}(x)$ and $\|u-Ax\|_{\infty}$ are given in scientific notation,
where the first number denotes the significant digit  and the second denotes the power.

 \begin{table}\caption{Experimental results for $\varepsilon = 10^{-2}$ ($\sigma=5, \tau = 0.124/\sigma, \alpha_k\equiv\alpha=0.28$, $x^0=\calB^*b$, $y^0=0$).}
\begin{center}
\begin{footnotesize}
\begin{tabular}{|c|c|c||r|c|r|r||r|c|r|r||c|}\hline
\mc{3}{|c||}{ } & \mc{4}{c||}{CP-$y\bar yx$} & \mc{4}{c||}{iCP-$y\bar yx$} & \\ [1pt] \cline{1-11}
$q/n^2$ & \mc{1}{|c|}{$n$} & \mc{1}{c||}{image}
& \mc{1}{c|}{TV($x$)}  & \mc{1}{c|}{$\|u-Ax\|_{\infty}$} & \mc{1}{c|}{SNR} & \mc{1}{c||}{It1}
& \mc{1}{c|}{TV($x$)}  & \mc{1}{c|}{$\|u-Ax\|_{\infty}$} & \mc{1}{c|}{SNR} & \mc{1}{c||}{It2}
& \mc{1}{c|}{$\frac{\text{It2}}{\text{It1}}$} \\  [2pt] \hline\hline

%
%

20\%   &  256    &   4.1.03   &   1.1319 3   &   7.2942 -3   &   4.87   &   65   &   1.1241 3   &   4.4170 -3   &   4.98   &   50   &   0.77\\\cline{3-12}

   &     &   4.1.05   &   1.3151 3   &   5.8069 -3   &   3.87   &   49   &   1.3069 3   &   5.8904 -3   &   3.96   &   38   &   0.78\\\cline{3-12}

   &     &   lena   &   2.1665 3   &   6.7751 -3   &   2.60   &   49   &   2.1605 3   &   7.5283 -3   &   2.63   &   37   &   0.76\\\cline{3-12}

   &     &   camera   &   2.5009 3   &   8.2455 -3   &   3.76   &   55   &   2.4940 3   &   4.2202 -3   &   3.79   &   41   &   0.75\\\cline{2-12}

   &  512    &   4.2.06   &   6.8642 3   &   5.7185 -3   &   1.49   &   41   &   6.8356 3   &   5.4164 -3   &   1.52   &   32   &   0.78\\\cline{3-12}

   &     &   4.2.07   &   6.1689 3   &   6.1500 -3   &   1.84   &   46   &   6.1445 3   &   6.1123 -3   &   1.89   &   36   &   0.78\\\cline{3-12}

   &     &   elaine   &   5.7648 3   &   7.0215 -3   &   2.27   &   47   &   5.7475 3   &   7.8325 -3   &   2.33   &   37   &   0.79\\\cline{3-12}

   &     &   brain   &   3.7925 3   &   3.3133 -3   &   5.36   &   66   &   3.7599 3   &   3.3699 -3   &   5.43   &   51   &   0.77\\\cline{2-12}

   &  1024    &   5.3.01   &   2.6227 4   &   7.1226 -3   &   2.29   &   51   &   2.6126 4   &   5.3319 -3   &   2.31   &   39   &   0.76\\\cline{3-12}

   &     &   5.3.02   &   3.2031 4   &   6.9707 -3   &   2.07   &   44   &   3.1954 4   &   5.7347 -3   &   2.12   &   35   &   0.80\\\cline{3-12}

   &     &   3.2.25   &   2.9271 4   &   5.7448 -3   &   2.60   &   40   &   2.9168 4   &   6.6754 -3   &   2.68   &   32   &   0.80\\\cline{3-12}

   &     &   7.2.01   &   1.5247 4   &   5.1647 -3   &   1.79   &   43   &   1.5138 4   &   5.3732 -3   &   1.84   &   35   &   0.81\\\hline

40\%   &  256    &   4.1.03   &   1.3488 3   &   4.2212 -3   &   7.86   &   69   &   1.3412 3   &   3.9813 -3   &   7.99   &   52   &   0.75\\\cline{3-12}

   &     &   4.1.05   &   1.7393 3   &   6.6140 -3   &   6.40   &   47   &   1.7350 3   &   9.0815 -3   &   6.52   &   37   &   0.79\\\cline{3-12}

   &     &   lena   &   2.8201 3   &   9.1814 -3   &   4.18   &   52   &   2.8181 3   &   4.7483 -3   &   4.22   &   39   &   0.75\\\cline{3-12}

   &     &   camera   &   3.1159 3   &   1.0701 -2   &   5.66   &   59   &   3.1119 3   &   3.6160 -3   &   5.69   &   44   &   0.75\\\cline{2-12}

   &  512    &   4.2.06   &   9.6638 3   &   6.2966 -3   &   2.82   &   37   &   9.6525 3   &   5.6120 -3   &   2.87   &   29   &   0.78\\\cline{3-12}

   &     &   4.2.07   &   8.6189 3   &   8.5462 -3   &   3.54   &   42   &   8.6143 3   &   7.2810 -3   &   3.61   &   33   &   0.79\\\cline{3-12}

   &     &   elaine   &   8.0715 3   &   9.2392 -3   &   3.15   &   43   &   8.0612 3   &   8.1932 -3   &   3.19   &   33   &   0.77\\\cline{3-12}

   &     &   brain   &   4.6370 3   &   9.4767 -3   &   4.39   &   65   &   4.6091 3   &   4.3974 -3   &   4.43   &   50   &   0.77\\\cline{2-12}

   &  1024    &   5.3.01   &   3.5028 4   &   8.1347 -3   &   3.75   &   48   &   3.5008 4   &   1.0318 -2   &   3.80   &   37   &   0.77\\\cline{3-12}

   &     &   5.3.02   &   4.4279 4   &   6.2587 -3   &   4.22   &   40   &   4.4227 4   &   7.6498 -3   &   4.27   &   31   &   0.78\\\cline{3-12}

   &     &   3.2.25   &   4.0071 4   &   5.7304 -3   &   5.39   &   38   &   3.9977 4   &   6.8326 -3   &   5.46   &   29   &   0.76\\\cline{3-12}

   &     &   7.2.01   &   2.1930 4   &   5.9025 -3   &   3.18   &   37   &   2.1870 4   &   5.7002 -3   &   3.25   &   30   &   0.81\\\hline

60\%   &  256    &   4.1.03   &   1.4182 3   &   5.5252 -3   &   13.44   &   77   &   1.4100 3   &   7.2768 -3   &   13.49   &   58   &   0.75\\\cline{3-12}

   &     &   4.1.05   &   2.0491 3   &   8.1584 -3   &   8.57   &   41   &   2.0462 3   &   9.5036 -3   &   8.72   &   32   &   0.78\\\cline{3-12}

   &     &   lena   &   3.2443 3   &   9.8217 -3   &   8.62   &   50   &   3.2451 3   &   1.0342 -2   &   8.75   &   38   &   0.76\\\cline{3-12}

   &     &   camera   &   3.4820 3   &   1.2103 -2   &   7.12   &   55   &   3.4808 3   &   4.9759 -3   &   7.20   &   42   &   0.76\\\cline{2-12}

   &  512    &   4.2.06   &   1.1612 4   &   8.1289 -3   &   4.21   &   34   &   1.1603 4   &   8.5716 -3   &   4.25   &   26   &   0.76\\\cline{3-12}

   &     &   4.2.07   &   1.0185 4   &   9.6775 -3   &   5.66   &   37   &   1.0188 4   &   1.1501 -2   &   5.75   &   29   &   0.78\\\cline{3-12}

   &     &   elaine   &   9.8019 3   &   1.0034 -2   &   7.56   &   38   &   9.7987 3   &   1.0745 -2   &   7.64   &   29   &   0.76\\\cline{3-12}

   &     &   brain   &   5.0578 3   &   1.0896 -2   &   4.15   &   59   &   5.0421 3   &   6.4967 -3   &   4.22   &   47   &   0.80\\\cline{2-12}

   &  1024    &   5.3.01   &   4.0490 4   &   1.1646 -2   &   6.00   &   43   &   4.0489 4   &   1.0927 -2   &   6.06   &   33   &   0.77\\\cline{3-12}

   &     &   5.3.02   &   5.3012 4   &   7.7387 -3   &   8.60   &   35   &   5.2984 4   &   7.6171 -3   &   8.71   &   27   &   0.77\\\cline{3-12}

   &     &   3.2.25   &   4.7548 4   &   6.9208 -3   &   8.00   &   33   &   4.7513 4   &   8.2037 -3   &   8.15   &   26   &   0.79\\\cline{3-12}

   &     &   7.2.01   &   2.7099 4   &   7.3859 -3   &   5.29   &   32   &   2.7047 4   &   7.9069 -3   &   5.33   &   25   &   0.78\\\hline

80\%   &  256    &   4.1.03   &   1.5117 3   &   1.9780 -2   &   15.84   &   47   &   1.5104 3   &   1.0129 -2   &   16.61   &   39   &   0.83\\\cline{3-12}

   &     &   4.1.05   &   2.2963 3   &   9.2757 -3   &   11.84   &   34   &   2.2947 3   &   9.0494 -3   &   12.06   &   27   &   0.79\\\cline{3-12}

   &     &   lena   &   3.4791 3   &   1.4328 -2   &   12.92   &   45   &   3.4796 3   &   1.4387 -2   &   13.11   &   34   &   0.76\\\cline{3-12}

   &     &   camera   &   3.6207 3   &   1.4613 -2   &   9.57   &   48   &   3.6203 3   &   1.5002 -2   &   9.67   &   38   &   0.79\\\cline{2-12}

   &  512    &   4.2.06   &   1.3283 4   &   9.8529 -3   &   9.40   &   27   &   1.3282 4   &   1.0326 -2   &   9.49   &   21   &   0.78\\\cline{3-12}

   &     &   4.2.07   &   1.1962 4   &   9.9285 -3   &   6.04   &   29   &   1.1967 4   &   1.0784 -2   &   6.11   &   23   &   0.79\\\cline{3-12}

   &     &   elaine   &   1.1196 4   &   1.3529 -2   &   9.11   &   31   &   1.1198 4   &   1.2973 -2   &   9.20   &   24   &   0.77\\\cline{3-12}

   &     &   brain   &   5.2027 3   &   1.3958 -2   &   14.92   &   54   &   5.1933 3   &   1.2994 -2   &   15.45   &   44   &   0.81\\\cline{2-12}

   &  1024    &   5.3.01   &   4.4290 4   &   1.7137 -2   &   9.40   &   36   &   4.4303 4   &   1.4849 -2   &   9.48   &   28   &   0.78\\\cline{3-12}

   &     &   5.3.02   &   5.9732 4   &   9.4555 -3   &   10.75   &   28   &   5.9721 4   &   9.5126 -3   &   10.86   &   22   &   0.79\\\cline{3-12}

   &     &   3.2.25   &   5.3345 4   &   7.3499 -3   &   12.81   &   27   &   5.3338 4   &   9.4321 -3   &   13.07   &   22   &   0.81\\\cline{3-12}

   &     &   7.2.01   &   3.1427 4   &   1.1171 -2   &   7.20   &   26   &   3.1409 4   &   8.6650 -3   &   7.27   &   21   &   0.81\\\hline
\end{tabular}\label{table-tol=1e-2}
\end{footnotesize}
\end{center}
\end{table}

 \begin{table}\caption{Experimental results for $\varepsilon = 10^{-3}$ ($\sigma=5, \tau = 0.124/\sigma, \alpha_k\equiv\alpha=0.28$, $x^0=\calB^*b$, $y^0=0$).}
\begin{center}
\begin{footnotesize}
\begin{tabular}{|c|c|c||r|c|r|r||r|c|r|r||c|}\hline
\mc{3}{|c||}{ } & \mc{4}{c||}{CP-$y\bar yx$} & \mc{4}{c||}{iCP-$y\bar yx$} & \\ [1pt] \cline{1-11}
$q/n^2$ & \mc{1}{|c|}{$n$} & \mc{1}{c||}{image}
& \mc{1}{c|}{TV($x$)}  & \mc{1}{c|}{$\|u-Ax\|_{\infty}$} & \mc{1}{c|}{SNR} & \mc{1}{c||}{It1}
& \mc{1}{c|}{TV($x$)}  & \mc{1}{c|}{$\|u-Ax\|_{\infty}$} & \mc{1}{c|}{SNR} & \mc{1}{c||}{It2}
& \mc{1}{c|}{$\frac{\text{It2}}{\text{It1}}$} \\  [2pt] \hline\hline

20\%   &  256    &   4.1.03   &   1.0592 3   &   9.0554 -4   &   9.33   &   329   &   1.0586 3   &   7.0139 -4   &   9.44   &   243   &   0.74\\\cline{3-12}

   &     &   4.1.05   &   1.2514 3   &   9.8182 -4   &   6.05   &   242   &   1.2508 3   &   7.4895 -4   &   6.10   &   179   &   0.74\\\cline{3-12}

   &     &   lena   &   2.0240 3   &   1.4969 -3   &   4.43   &   318   &   2.0236 3   &   1.0133 -3   &   4.44   &   231   &   0.73\\\cline{3-12}

   &     &   camera   &   2.3622 3   &   1.2124 -3   &   5.92   &   330   &   2.3618 3   &   8.5057 -4   &   5.94   &   241   &   0.73\\\cline{2-12}

   &  512    &   4.2.06   &   6.6077 3   &   1.0215 -3   &   2.40   &   217   &   6.6059 3   &   8.5628 -4   &   2.42   &   160   &   0.74\\\cline{3-12}

   &     &   4.2.07   &   5.8261 3   &   1.7749 -3   &   3.70   &   271   &   5.8246 3   &   1.7558 -3   &   3.74   &   199   &   0.73\\\cline{3-12}

   &     &   elaine   &   5.4348 3   &   1.2689 -3   &   3.97   &   267   &   5.4331 3   &   1.0330 -3   &   4.00   &   196   &   0.73\\\cline{3-12}

   &     &   brain   &   3.5531 3   &   1.4475 -3   &   8.42   &   383   &   3.5509 3   &   1.0458 -3   &   8.45   &   278   &   0.73\\\cline{2-12}

   &  1024    &   5.3.01   &   2.4855 4   &   1.5801 -3   &   3.69   &   290   &   2.4848 4   &   1.5348 -3   &   3.71   &   211   &   0.73\\\cline{3-12}

   &     &   5.3.02   &   3.1250 4   &   1.7648 -3   &   2.86   &   195   &   3.1244 4   &   1.4521 -3   &   2.88   &   145   &   0.74\\\cline{3-12}

   &     &   3.2.25   &   2.8511 4   &   1.5011 -3   &   3.72   &   179   &   2.8501 4   &   1.2207 -3   &   3.75   &   133   &   0.74\\\cline{3-12}

   &     &   7.2.01   &   1.4674 4   &   1.1333 -3   &   2.83   &   212   &   1.4667 4   &   9.6846 -4   &   2.87   &   159   &   0.75\\\hline

40\%   &  256    &   4.1.03   &   1.2637 3   &   8.7982 -4   &   17.13   &   358   &   1.2631 3   &   6.4152 -4   &   17.18   &   265   &   0.74\\\cline{3-12}

   &     &   4.1.05   &   1.6953 3   &   1.8062 -3   &   8.88   &   197   &   1.6950 3   &   1.2055 -3   &   8.97   &   147   &   0.75\\\cline{3-12}

   &     &   lena   &   2.6837 3   &   1.3561 -3   &   6.19   &   295   &   2.6835 3   &   1.0678 -3   &   6.21   &   215   &   0.73\\\cline{3-12}

   &     &   camera   &   3.0109 3   &   1.3347 -3   &   9.22   &   289   &   3.0109 3   &   1.0672 -3   &   9.28   &   211   &   0.73\\\cline{2-12}

   &  512    &   4.2.06   &   9.4002 3   &   1.6715 -3   &   4.16   &   192   &   9.3995 3   &   1.3027 -3   &   4.18   &   141   &   0.73\\\cline{3-12}

   &     &   4.2.07   &   8.2590 3   &   1.7877 -3   &   6.11   &   242   &   8.2584 3   &   1.4797 -3   &   6.14   &   177   &   0.73\\\cline{3-12}

   &     &   elaine   &   7.7316 3   &   2.0480 -3   &   4.97   &   232   &   7.7309 3   &   1.6872 -3   &   5.00   &   170   &   0.73\\\cline{3-12}

   &     &   brain   &   4.4112 3   &   2.0597 -3   &   7.85   &   332   &   4.4100 3   &   1.5813 -3   &   7.89   &   242   &   0.73\\\cline{2-12}

   &  1024    &   5.3.01   &   3.3361 4   &   1.9296 -3   &   6.11   &   283   &   3.3359 4   &   1.8139 -3   &   6.13   &   206   &   0.73\\\cline{3-12}

   &     &   5.3.02   &   4.3547 4   &   1.6664 -3   &   5.44   &   171   &   4.3544 4   &   1.2812 -3   &   5.47   &   127   &   0.74\\\cline{3-12}

   &     &   3.2.25   &   3.9377 4   &   1.4296 -3   &   7.25   &   159   &   3.9372 4   &   1.2021 -3   &   7.29   &   118   &   0.74\\\cline{3-12}

   &     &   7.2.01   &   2.1404 4   &   2.1191 -3   &   4.81   &   186   &   2.1401 4   &   1.6644 -3   &   4.86   &   138   &   0.74\\\hline

60\%   &  256    &   4.1.03   &   1.3827 3   &   1.5129 -3   &   21.78   &   260   &   1.3826 3   &   9.4321 -4   &   22.12   &   193   &   0.74\\\cline{3-12}

   &     &   4.1.05   &   2.0087 3   &   1.9442 -3   &   12.48   &   182   &   2.0086 3   &   1.4960 -3   &   12.58   &   135   &   0.74\\\cline{3-12}

   &     &   lena   &   3.1255 3   &   1.6699 -3   &   16.48   &   283   &   3.1256 3   &   1.7352 -3   &   16.62   &   207   &   0.73\\\cline{3-12}

   &     &   camera   &   3.3504 3   &   2.4634 -3   &   14.46   &   326   &   3.3498 3   &   1.7600 -3   &   14.44   &   234   &   0.72\\\cline{2-12}

   &  512    &   4.2.06   &   1.1364 4   &   2.4834 -3   &   5.48   &   180   &   1.1364 4   &   1.9204 -3   &   5.49   &   132   &   0.73\\\cline{3-12}

   &     &   4.2.07   &   9.8460 3   &   2.8274 -3   &   9.07   &   200   &   9.8463 3   &   2.1139 -3   &   9.13   &   147   &   0.73\\\cline{3-12}

   &     &   elaine   &   9.4950 3   &   3.1823 -3   &   10.79   &   199   &   9.4948 3   &   2.4538 -3   &   10.83   &   146   &   0.73\\\cline{3-12}

   &     &   brain   &   4.8640 3   &   3.0694 -3   &   7.05   &   290   &   4.8640 3   &   3.1651 -3   &   7.11   &   213   &   0.73\\\cline{2-12}

   &  1024    &   5.3.01   &   3.8842 4   &   4.1921 -3   &   9.13   &   273   &   3.8843 4   &   2.3870 -3   &   9.16   &   199   &   0.73\\\cline{3-12}

   &     &   5.3.02   &   5.2337 4   &   2.2509 -3   &   11.43   &   158   &   5.2337 4   &   1.7885 -3   &   11.50   &   117   &   0.74\\\cline{3-12}

   &     &   3.2.25   &   4.7037 4   &   1.9918 -3   &   10.37   &   137   &   4.7035 4   &   2.0416 -3   &   10.43   &   102   &   0.74\\\cline{3-12}

   &     &   7.2.01   &   2.6668 4   &   2.7801 -3   &   7.06   &   159   &   2.6668 4   &   2.1974 -3   &   7.10   &   118   &   0.74\\\hline

80\%   &  256    &   4.1.03   &   1.4787 3   &   1.2684 -3   &   29.97   &   217   &   1.4785 3   &   1.6997 -3   &   29.34   &   162   &   0.75\\\cline{3-12}

   &     &   4.1.05   &   2.2664 3   &   2.9759 -3   &   19.08   &   171   &   2.2664 3   &   2.3118 -3   &   19.17   &   123   &   0.72\\\cline{3-12}

   &     &   lena   &   3.3928 3   &   2.5510 -3   &   22.10   &   229   &   3.3928 3   &   2.0750 -3   &   22.12   &   167   &   0.73\\\cline{3-12}

   &     &   camera   &   3.5637 3   &   6.0327 -3   &   13.63   &   211   &   3.5642 3   &   3.4959 -3   &   13.77   &   156   &   0.74\\\cline{2-12}

   &  512    &   4.2.06   &   1.3092 4   &   3.4162 -3   &   12.37   &   143   &   1.3093 4   &   3.3944 -3   &   12.42   &   105   &   0.73\\\cline{3-12}

   &     &   4.2.07   &   1.1693 4   &   5.4809 -3   &   9.34   &   163   &   1.1693 4   &   4.2707 -3   &   9.38   &   119   &   0.73\\\cline{3-12}

   &     &   elaine   &   1.0980 4   &   4.2817 -3   &   12.57   &   159   &   1.0980 4   &   3.3699 -3   &   12.61   &   116   &   0.73\\\cline{3-12}

   &     &   brain   &   5.0233 3   &   1.5505 -3   &   34.64   &   275   &   5.0219 3   &   1.5675 -3   &   34.09   &   206   &   0.75\\\cline{2-12}

   &  1024    &   5.3.01   &   4.3206 4   &   5.4126 -3   &   13.58   &   205   &   4.3207 4   &   4.3222 -3   &   13.61   &   149   &   0.73\\\cline{3-12}

   &     &   5.3.02   &   5.9350 4   &   4.1668 -3   &   13.04   &   120   &   5.9350 4   &   3.1152 -3   &   13.07   &   88   &   0.73\\\cline{3-12}

   &     &   3.2.25   &   5.3086 4   &   2.9624 -3   &   15.55   &   103   &   5.3086 4   &   2.0553 -3   &   15.64   &   77   &   0.75\\\cline{3-12}

   &     &   7.2.01   &   3.1127 4   &   5.8998 -3   &   8.74   &   128   &   3.1129 4   &   4.5796 -3   &   8.77   &   95   &   0.74\\\hline
\end{tabular}\label{table-tol=1e-3}
\end{footnotesize}
\end{center}
\end{table}

 \begin{table}\caption{Experimental results for $\varepsilon = 10^{-4}$ ($\sigma=5, \tau = 0.124/\sigma, \alpha_k\equiv\alpha=0.28$, $x^0=\calB^*b$, $y^0=0$).}
\begin{center}
\begin{footnotesize}
\begin{tabular}{|c|c|c||r|c|r|r||r|c|r|r||c|}\hline
\mc{3}{|c||}{ } & \mc{4}{c||}{CP-$y\bar yx$} & \mc{4}{c||}{iCP-$y\bar yx$} & \\ [1pt] \cline{1-11}
$q/n^2$ & \mc{1}{|c|}{$n$} & \mc{1}{c||}{image}
& \mc{1}{c|}{TV($x$)}  & \mc{1}{c|}{$\|u-Ax\|_{\infty}$} & \mc{1}{c|}{SNR} & \mc{1}{c||}{It1}
& \mc{1}{c|}{TV($x$)}  & \mc{1}{c|}{$\|u-Ax\|_{\infty}$} & \mc{1}{c|}{SNR} & \mc{1}{c||}{It2}
& \mc{1}{c|}{$\frac{\text{It2}}{\text{It1}}$} \\  [2pt] \hline\hline

20\%   &  256    &   4.1.03   &   1.0529 3   &   1.4547 -4   &   11.76   &   1241   &   1.0528 3   &   1.0827 -4   &   11.79   &   908   &   0.73\\\cline{3-12}

   &     &   4.1.05   &   1.2450 3   &   1.7284 -4   &   6.75   &   953   &   1.2449 3   &   1.1230 -4   &   6.75   &   694   &   0.73\\\cline{3-12}

   &     &   lena   &   2.0075 3   &   1.8010 -4   &   5.30   &   1272   &   2.0074 3   &   1.3062 -4   &   5.31   &   920   &   0.72\\\cline{3-12}

   &     &   camera   &   2.3487 3   &   2.2856 -4   &   7.17   &   1277   &   2.3486 3   &   1.6444 -4   &   7.18   &   924   &   0.72\\\cline{2-12}

   &  512    &   4.2.06   &   6.5656 3   &   2.1007 -4   &   3.57   &   1016   &   6.5654 3   &   1.5352 -4   &   3.58   &   737   &   0.73\\\cline{3-12}

   &     &   4.2.07   &   5.7503 3   &   2.2159 -4   &   7.05   &   1595   &   5.7500 3   &   1.9751 -4   &   7.06   &   1146   &   0.72\\\cline{3-12}

   &     &   elaine   &   5.3811 3   &   1.8871 -4   &   5.78   &   1250   &   5.3809 3   &   1.5154 -4   &   5.79   &   905   &   0.72\\\cline{3-12}

   &     &   brain   &   3.5307 3   &   1.8038 -4   &   10.45   &   1379   &   3.5305 3   &   1.4380 -4   &   10.48   &   1007   &   0.73\\\cline{2-12}

   &  1024    &   5.3.01   &   2.4471 4   &   2.6258 -4   &   6.80   &   1873   &   2.4470 4   &   1.9095 -4   &   6.81   &   1350   &   0.72\\\cline{3-12}

   &     &   5.3.02   &   3.1125 4   &   1.9095 -4   &   3.74   &   926   &   3.1124 4   &   1.4926 -4   &   3.75   &   672   &   0.73\\\cline{3-12}

   &     &   3.2.25   &   2.8419 4   &   2.1814 -4   &   4.33   &   726   &   2.8418 4   &   1.6618 -4   &   4.34   &   531   &   0.73\\\cline{3-12}

   &     &   7.2.01   &   1.4567 4   &   2.6271 -4   &   5.21   &   1209   &   1.4567 4   &   1.8932 -4   &   5.24   &   881   &   0.73\\\hline

40\%   &  256    &   4.1.03   &   1.2613 3   &   1.4225 -4   &   18.57   &   866   &   1.2613 3   &   9.7930 -5   &   18.63   &   659   &   0.76\\\cline{3-12}

   &     &   4.1.05   &   1.6906 3   &   2.1066 -4   &   10.92   &   689   &   1.6905 3   &   1.7779 -4   &   10.96   &   505   &   0.73\\\cline{3-12}

   &     &   lena   &   2.6725 3   &   1.9869 -4   &   7.18   &   956   &   2.6725 3   &   1.6135 -4   &   7.19   &   692   &   0.72\\\cline{3-12}

   &     &   camera   &   2.9844 3   &   1.8506 -4   &   15.42   &   1382   &   2.9843 3   &   1.3780 -4   &   15.44   &   1001   &   0.72\\\cline{2-12}

   &  512    &   4.2.06   &   9.3340 3   &   2.5729 -4   &   6.54   &   1022   &   9.3339 3   &   1.9041 -4   &   6.56   &   741   &   0.73\\\cline{3-12}

   &     &   4.2.07   &   8.1870 3   &   2.0627 -4   &   10.52   &   1179   &   8.1870 3   &   1.5139 -4   &   10.56   &   855   &   0.73\\\cline{3-12}

   &     &   elaine   &   7.6796 3   &   2.9240 -4   &   7.09   &   1015   &   7.6795 3   &   2.2082 -4   &   7.10   &   736   &   0.73\\\cline{3-12}

   &     &   brain   &   4.3411 3   &   1.8608 -4   &   18.22   &   1774   &   4.3409 3   &   1.3406 -4   &   18.31   &   1293   &   0.73\\\cline{2-12}

   &  1024    &   5.3.01   &   3.2940 4   &   3.0227 -4   &   10.09   &   1463   &   3.2940 4   &   2.2911 -4   &   10.11   &   1059   &   0.72\\\cline{3-12}

   &     &   5.3.02   &   4.3426 4   &   3.0691 -4   &   6.76   &   735   &   4.3425 4   &   2.3810 -4   &   6.78   &   535   &   0.73\\\cline{3-12}

   &     &   3.2.25   &   3.9299 4   &   2.8261 -4   &   8.27   &   587   &   3.9298 4   &   2.2508 -4   &   8.29   &   431   &   0.73\\\cline{3-12}

   &     &   7.2.01   &   2.1271 4   &   2.8676 -4   &   7.44   &   1043   &   2.1271 4   &   2.3184 -4   &   7.46   &   761   &   0.73\\\hline

60\%   &  256    &   4.1.03   &   1.3813 3   &   2.1023 -4   &   23.71   &   651   &   1.3812 3   &   1.4074 -4   &   23.75   &   494   &   0.76\\\cline{3-12}

   &     &   4.1.05   &   2.0061 3   &   2.7249 -4   &   14.10   &   508   &   2.0061 3   &   1.9179 -4   &   14.15   &   377   &   0.74\\\cline{3-12}

   &     &   lena   &   3.1180 3   &   3.0073 -4   &   17.80   &   828   &   3.1180 3   &   2.3795 -4   &   17.78   &   601   &   0.73\\\cline{3-12}

   &     &   camera   &   3.3238 3   &   2.9028 -4   &   20.34   &   1126   &   3.3237 3   &   2.3618 -4   &   20.29   &   820   &   0.73\\\cline{2-12}

   &  512    &   4.2.06   &   1.1338 4   &   3.0901 -4   &   6.75   &   582   &   1.1338 4   &   2.4203 -4   &   6.77   &   424   &   0.73\\\cline{3-12}

   &     &   4.2.07   &   9.7655 3   &   3.8127 -4   &   16.08   &   998   &   9.7655 3   &   3.0599 -4   &   16.12   &   725   &   0.73\\\cline{3-12}

   &     &   elaine   &   9.4659 3   &   4.3193 -4   &   13.04   &   679   &   9.4659 3   &   3.0643 -4   &   13.06   &   496   &   0.73\\\cline{3-12}

   &     &   brain   &   4.7679 3   &   2.9519 -4   &   23.57   &   1747   &   4.7678 3   &   2.4006 -4   &   23.81   &   1278   &   0.73\\\cline{2-12}

   &  1024    &   5.3.01   &   3.8615 4   &   8.4962 -4   &   12.40   &   1018   &   3.8615 4   &   6.0623 -4   &   12.43   &   739   &   0.73\\\cline{3-12}

   &     &   5.3.02   &   5.2243 4   &   6.5600 -4   &   12.61   &   608   &   5.2243 4   &   4.5762 -4   &   12.60   &   444   &   0.73\\\cline{3-12}

   &     &   3.2.25   &   4.6980 4   &   4.5021 -4   &   11.83   &   484   &   4.6980 4   &   3.5987 -4   &   11.85   &   356   &   0.74\\\cline{3-12}

   &     &   7.2.01   &   2.6581 4   &   4.1830 -4   &   8.96   &   731   &   2.6581 4   &   3.0788 -4   &   8.98   &   536   &   0.73\\\hline

80\%   &  256    &   4.1.03   &   1.4781 3   &   2.4197 -4   &   30.76   &   419   &   1.4781 3   &   1.7573 -4   &   30.80   &   327   &   0.78\\\cline{3-12}

   &     &   4.1.05   &   2.2649 3   &   3.8634 -4   &   20.65   &   425   &   2.2649 3   &   3.6174 -4   &   20.66   &   314   &   0.74\\\cline{3-12}

   &     &   lena   &   3.3916 3   &   3.4517 -4   &   22.44   &   447   &   3.3916 3   &   4.2157 -4   &   22.44   &   331   &   0.74\\\cline{3-12}

   &     &   camera   &   3.5313 3   &   3.3864 -4   &   27.10   &   970   &   3.5312 3   &   2.7673 -4   &   26.92   &   707   &   0.73\\\cline{2-12}

   &  512    &   4.2.06   &   1.3069 4   &   5.8567 -4   &   14.79   &   519   &   1.3069 4   &   4.1928 -4   &   14.82   &   380   &   0.73\\\cline{3-12}

   &     &   4.2.07   &   1.1585 4   &   1.1237 -3   &   20.78   &   918   &   1.1586 4   &   8.7215 -4   &   20.86   &   666   &   0.73\\\cline{3-12}

   &     &   elaine   &   1.0928 4   &   1.0021 -3   &   19.55   &   826   &   1.0928 4   &   7.3748 -4   &   19.57   &   600   &   0.73\\\cline{3-12}

   &     &   brain   &   5.0193 3   &   1.6601 -4   &   35.21   &   689   &   5.0192 3   &   1.4387 -4   &   35.21   &   549   &   0.80\\\cline{2-12}

   &  1024    &   5.3.01   &   4.2953 4   &   9.6281 -4   &   22.98   &   986   &   4.2953 4   &   7.7170 -4   &   23.03   &   717   &   0.73\\\cline{3-12}

   &     &   5.3.02   &   5.9293 4   &   7.7142 -4   &   15.46   &   452   &   5.9293 4   &   5.5496 -4   &   15.50   &   331   &   0.73\\\cline{3-12}

   &     &   3.2.25   &   5.3062 4   &   8.1420 -4   &   16.79   &   336   &   5.3062 4   &   5.9169 -4   &   16.81   &   249   &   0.74\\\cline{3-12}

   &     &   7.2.01   &   3.1059 4   &   6.7967 -4   &   9.43   &   558   &   3.1059 4   &   5.4514 -4   &   9.43   &   409   &   0.73\\\hline
\end{tabular}\label{table-tol=1e-4}
\end{footnotesize}
\end{center}
\end{table}

It can be seen from Tables \ref{table-tol=1e-2}-\ref{table-tol=1e-4} that, to obtain solutions satisfying the aforementioned conditions,  iCP-$y\bar yx$ is generally faster than CP-$y\bar yx$.  Specifically, within our setting  the numbers of iterations consumed by iCP-$y\bar yx$ range, roughly, from $70\%$--$80\%$ of those consumed by CP-$y\bar yx$.
In most cases, iCP-$y\bar yx$ obtained recovery results with slightly better final objective function values and feasibility residues.
The quality of recovered images is also slightly better in terms of SNR.
By comparing results between different tables, we see that high accuracy solutions in optimization point of view generally imply better image quality measured by SNR.
This could imply that solving the problem to a certain high accuracy is in some sense necessary for better recovery, though the improvement of image quality could be small when the solution
is already very accurate.
It can also be observed from the results that both algorithms converge very fast at the beginning stage and slow down afterwards. In particular, to improve the solution quality by
one more digit of accuracy (measured by optimality residue defined in \eqref{stop-rule}-\eqref{stop-rule-iCPA}), the number of iterations could be multiplied by a few times, which is probably a common feature of first-order optimization algorithms.
The fact is that in most cases one does not need to solve imaging problems to extremely high accuracy,
because the recovered results hardly have any difference detectable by human eyes  when they are already accurate enough.
In words, the inertial technique accelerates the original algorithm to some extent without increasing the total computational cost.

To better visualize the performance improvement of iCP-$y\bar yx$ over CP-$y\bar yx$, we reorganized the results given in Tables \ref{table-tol=1e-2}-\ref{table-tol=1e-4}
and presented them in Figure \ref{Fig-m3}. For each measurement level $q/n^2$ and image size $n$, we accumulated the number of iterations for
different images and took an average. The results for $\varepsilon = 10^{-2}, 10^{-3}$ and $10^{-4}$ are given in Figure \ref{Fig-m3}. By comparing the
three plots in Figure \ref{Fig-m3}, we see that the number of iterations increased from a few dozens to around one thousand when the accuracy tolerance
$\varepsilon$ was decreased from $10^{-2}$ to $10^{-4}$. From the results we can also observe that, on average, both algorithms perform stably in the sense that the consumed number of iterations do not vary much for different image sizes.

\begin{figure}[htbp]
\centering{
\includegraphics[trim = 0 200 50  160, scale = 0.27]{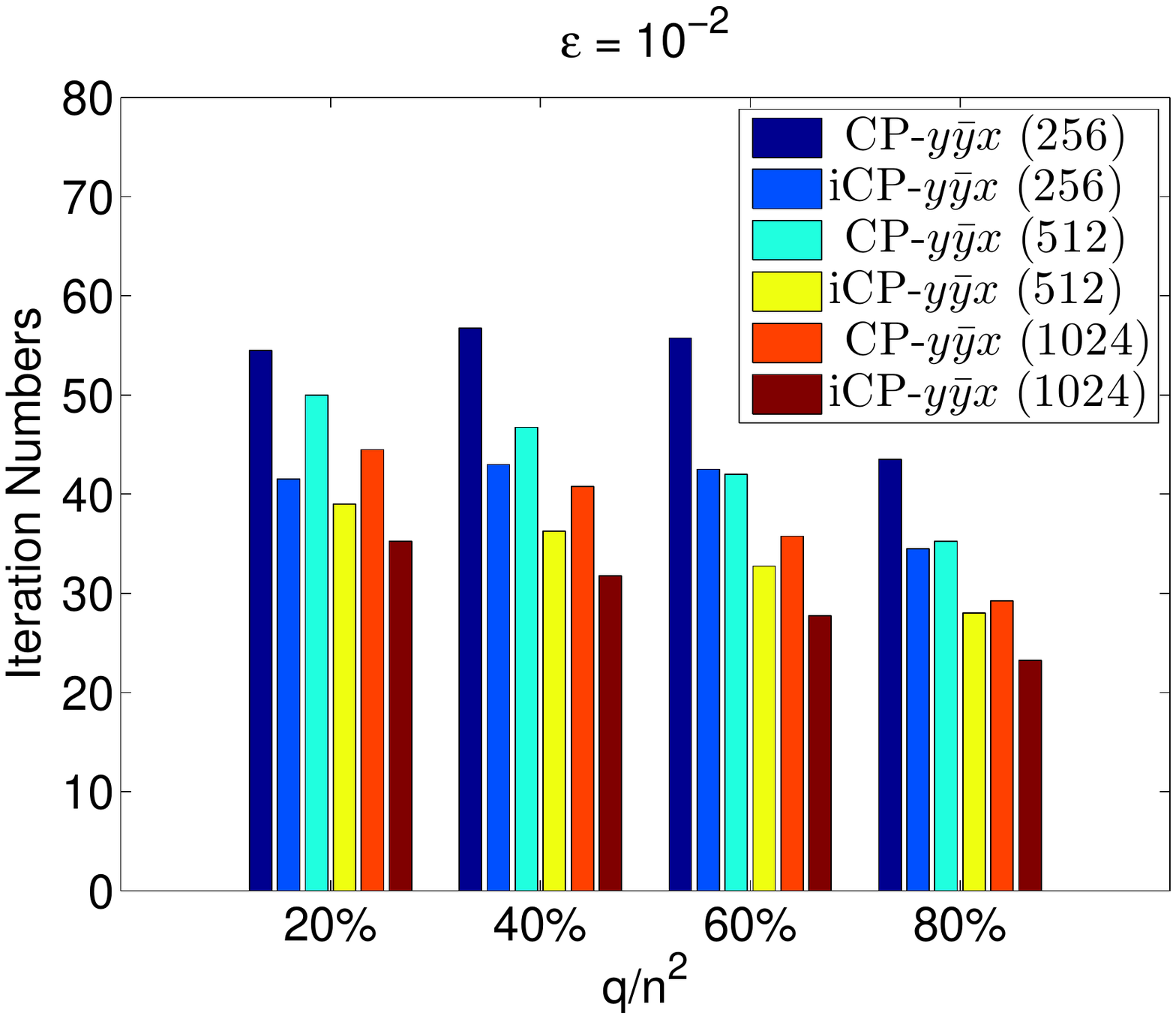}
\includegraphics[trim = 0 200 50  160, scale = 0.27]{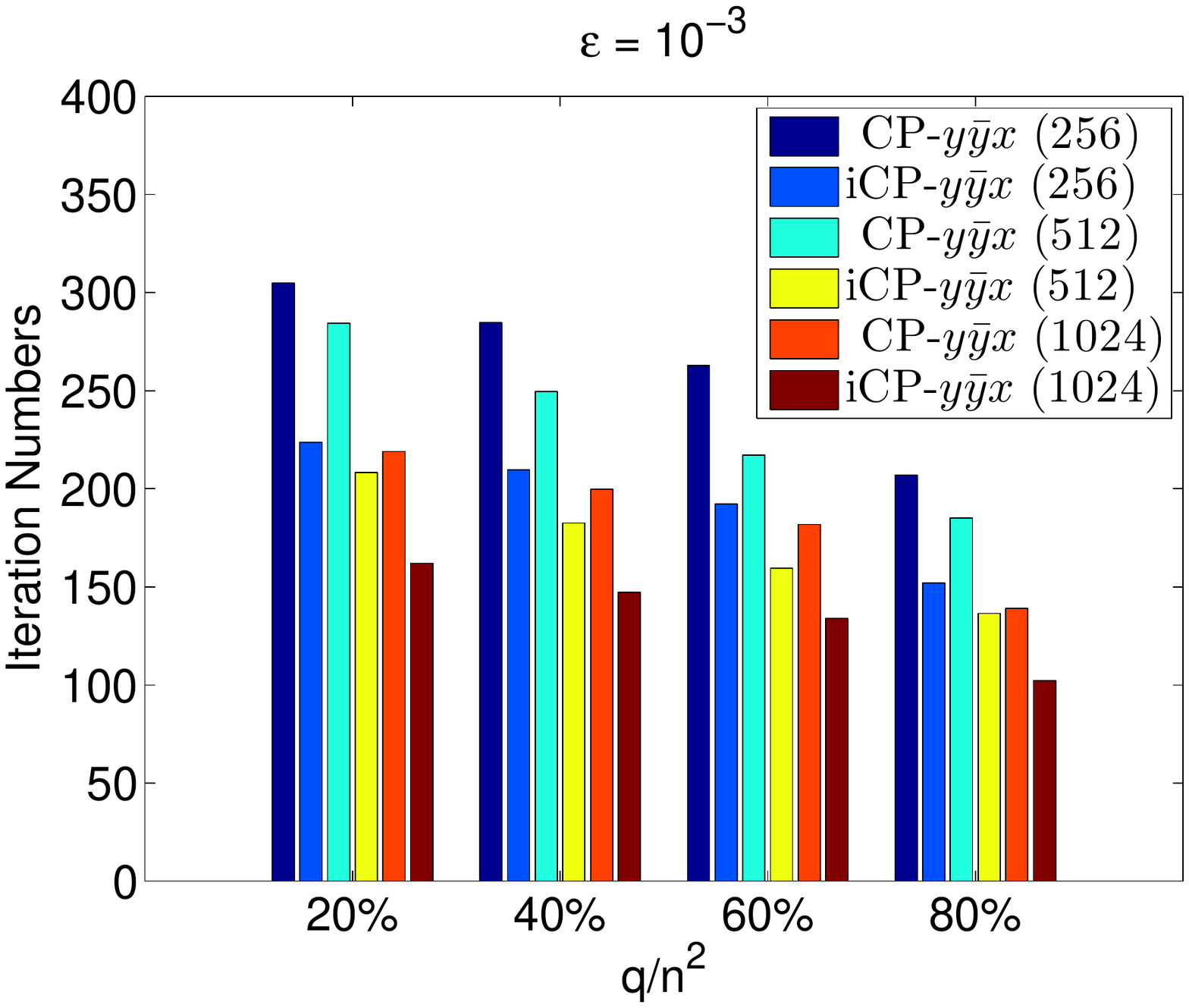}
\includegraphics[trim = 0 200 50  160, scale = 0.27]{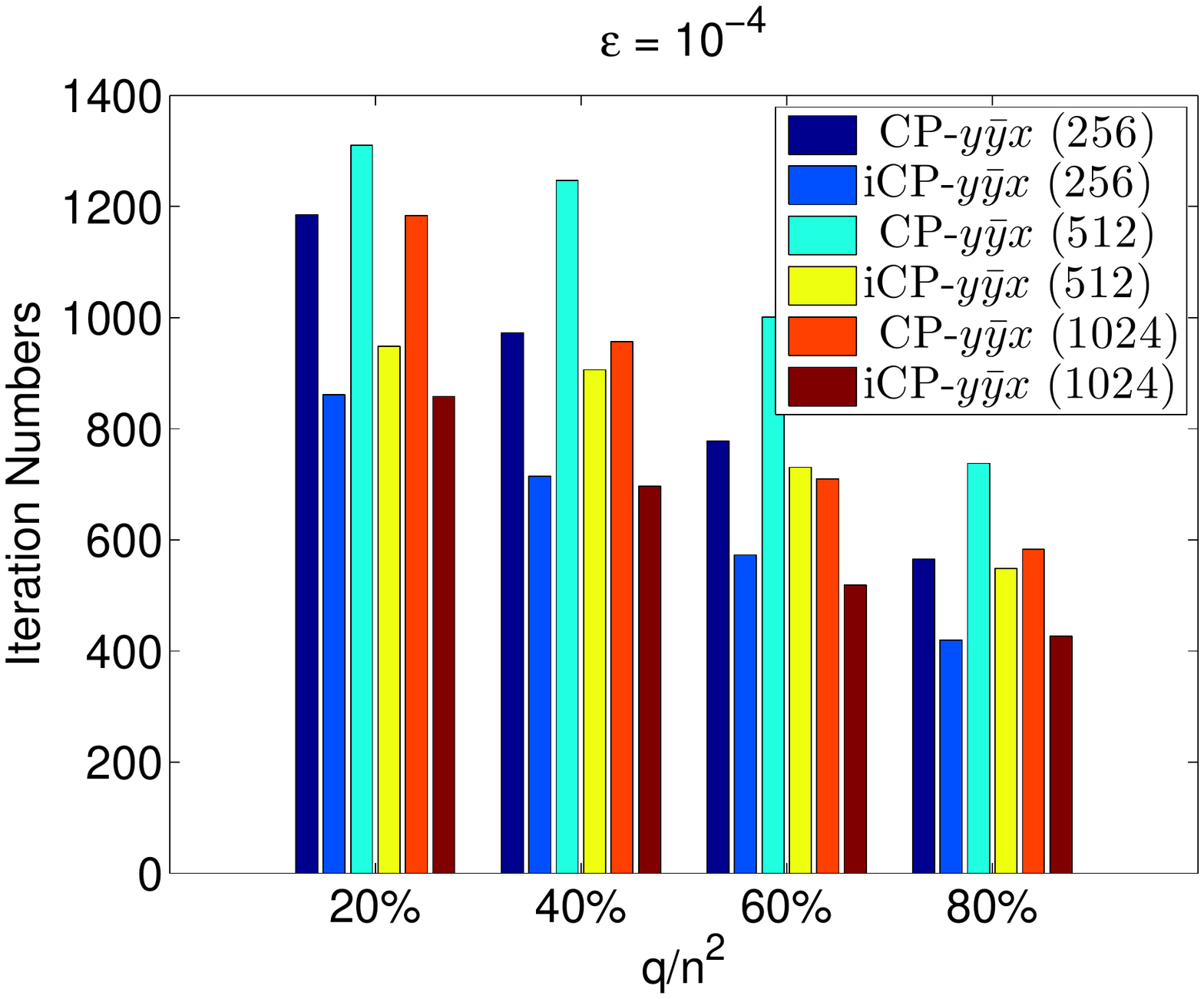}
}\caption{Comparison results of CP-$y\bar yx$ and iCP-$y\bar yx$ on different image  sizes  and stopping tolerance ($n=256,512,1024$, and from left to right $\varepsilon = 10^{-2}, 10^{-3}, 10^{-4}$, respectively). }
\label{Fig-m3}
\end{figure}

We also examined the performance of iCP-$y\bar yx$ with different constant strategies for the inertial extrapolation stepsize $\alpha_k$. In particular, for $n=1024$ we tested $\alpha_k\equiv\alpha \in \{0.05, 0.15, 0.25, 0.35\}$. The results are given in Figure \ref{Fig-alp4}. It can be seen from the results that, for the four tested $\alpha$ values, larger ones generally give better performance. Recall that, according to our analysis, iCP-$y\bar yx$ is guaranteed to converge under the condition $0\leq \alpha_k\leq \alpha_{k+1}\leq \alpha < 1/3$ for all $k$. Indeed, we have observed that iCP-$y\bar yx$ either slows down or performs instable for large values of $\alpha$, say, larger than $0.3$, especially when the number of measurements is relatively small. This is the main reason that we set $\alpha_k$ a constant value that is near $0.3$ but not larger. Similar discussions for compressive principal component pursuit problems can be found in \cite{CMY14a}.

\begin{figure}[htbp]
\centering{
\includegraphics[trim = 0 200 50  160, scale = 0.27]{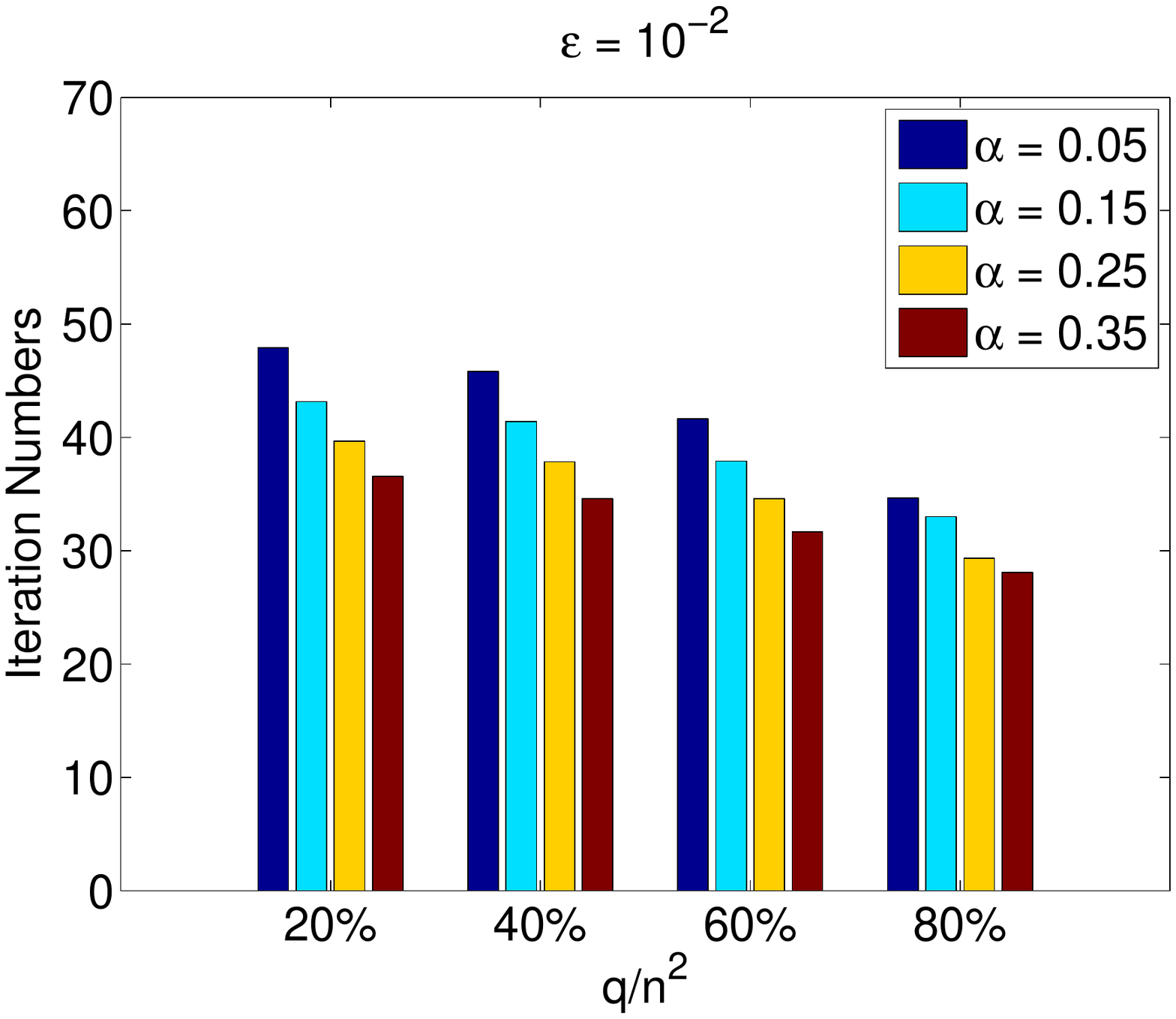}
\includegraphics[trim = 0 200 50  160, scale = 0.27]{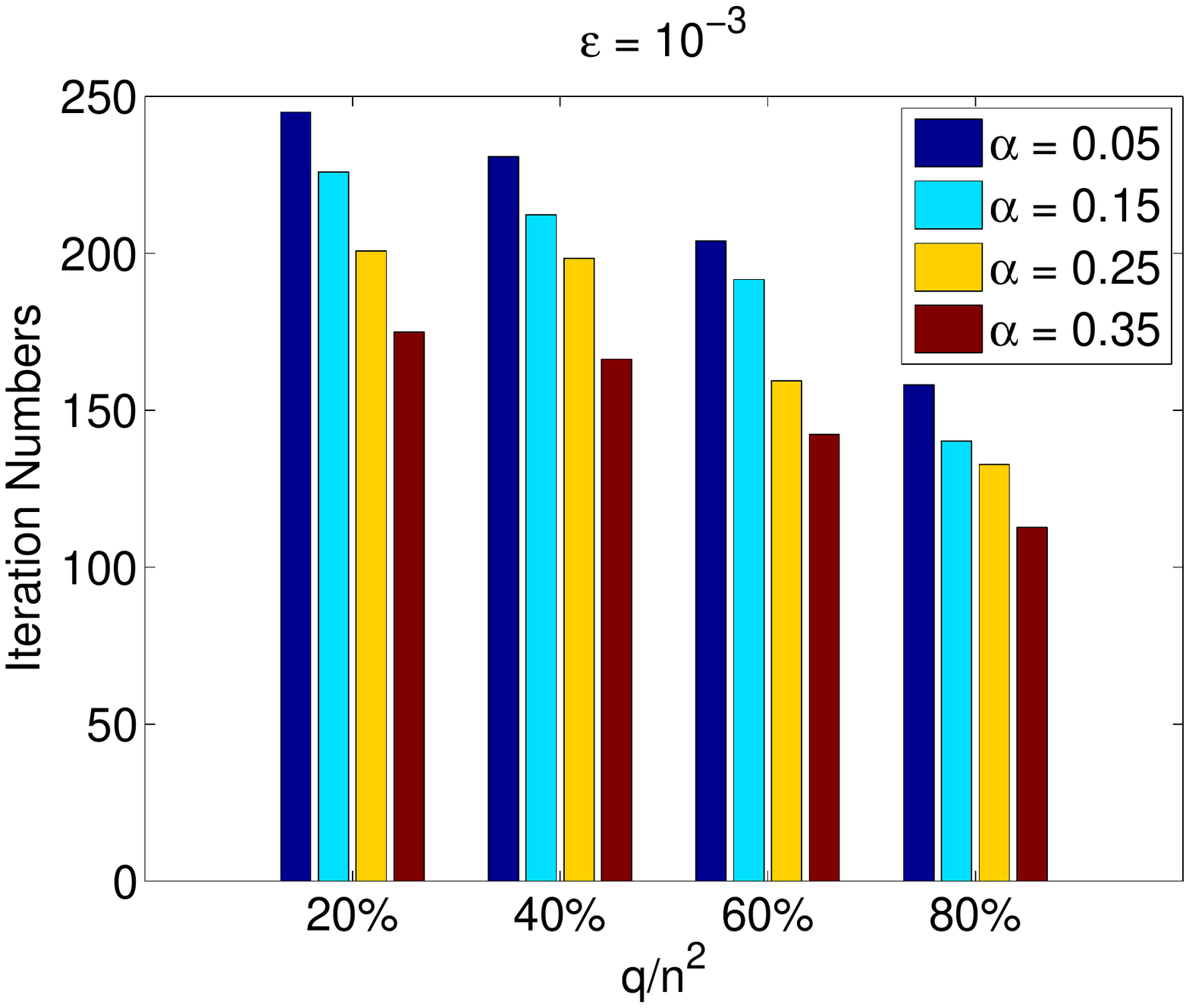}
\includegraphics[trim = 0 200 50  160, scale = 0.27]{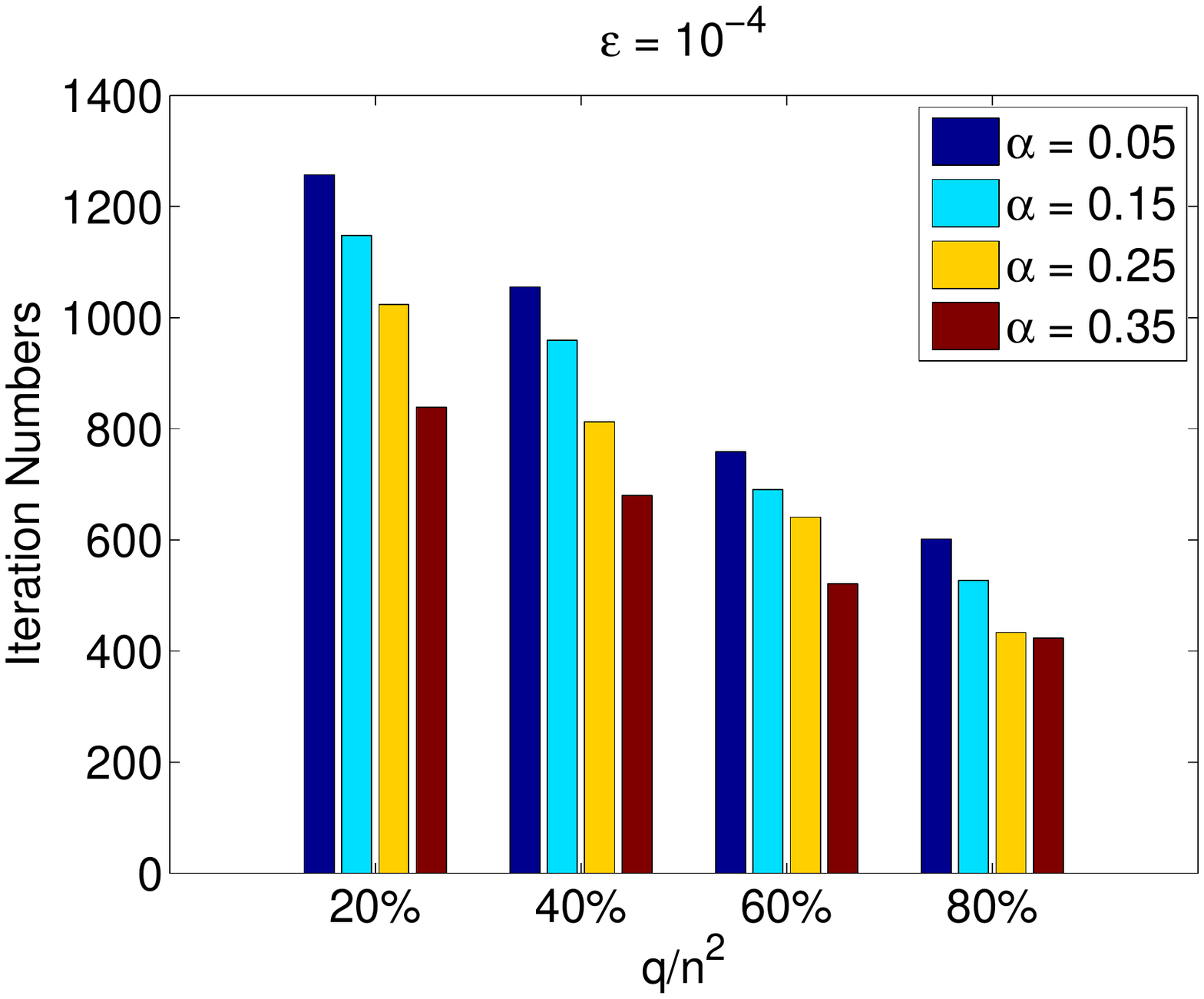}
}\caption{Comparison results of iCP-$y\bar yx$ on different $\alpha_k\equiv\alpha$ and stopping tolerance ($\alpha\in\{0.05,0.15,0.25,0.35\}$,
and from left to right  $\varepsilon = 10^{-2}, 10^{-3}, 10^{-4}$, respectively).}
\label{Fig-alp4}
\end{figure}

\section{Concluding remarks}\label{sc:concluding}
In this paper, based on the observations given in \cite{CP11,ST14}, we showed that CPAs and LADMs generate exactly the same sequence of points if the initial points for LADMs are properly chosen. The dependence on initial points for LADM can be relaxed by focusing on cyclically equivalent forms of the algorithms.
By using the fact that CPAs are applications of a general PPM to the KKT system, we were able to propose
inertial CPAs for solving structured convex optimization problem. Under certain conditions, the global point-convergence, nonasymptotic $O(1/k)$ and asymptotic $o(1/k)$ rate of convergence of the proposed inertial CPAs are guaranteed. These convergence rate results are previously not known for inertial type methods.
Our preliminary implementation of the algorithms and extensive experimental results on TV based image reconstruction problems have shown that inertial CPAs are generally faster than the corresponding original CPAs. Though in a sense the acceleration is not very significant, the inertial extrapolation step does not introduce any additional yet unnegligible computational cost either.

The extrapolation steplength $\alpha_k$ was set to be constant in our experiments, which was determined based on experimental results. How to select $\alpha_k$ adaptively such that the overall performance is stable and more efficient deserves further investigation. Moreover, the requirement that $\{\alpha_k\}_{k=0}^\infty$ is nondecreasing seems not reasonable either. Interesting topics for future research may be to relax the conditions on $\{\alpha_k\}_{k=0}^\infty$, to improve the convergence rate results and to propose modified inertial type algorithms so that the extrapolation stepsize $\alpha_k$ can be significantly enlarged.

\bibliographystyle{plain}

\begin{thebibliography}{10}

\bibitem{APZ84}
Filippo Aluffi-Pentini, Valerio Parisi, and Francesco Zirilli.
\newblock Algorithm {$617$}. {DAFNE}: a differential-equations algorithm for
  nonlinear equations.
\newblock {\em ACM Trans. Math. Software}, 10(3):317--324, 1984.

\bibitem{Alv00}
Felipe Alvarez.
\newblock On the minimizing property of a second order dissipative system in
  {H}ilbert spaces.
\newblock {\em SIAM J. Control Optim.}, 38(4):1102--1119 (electronic), 2000.

\bibitem{Alv04}
Felipe Alvarez.
\newblock Weak convergence of a relaxed and inertial hybrid projection-proximal
  point algorithm for maximal monotone operators in {H}ilbert space.
\newblock {\em SIAM J. Optim.}, 14(3):773--782 (electronic), 2004.

\bibitem{AA01}
Felipe Alvarez and Hedy Attouch.
\newblock An inertial proximal method for maximal monotone operators via
  discretization of a nonlinear oscillator with damping.
\newblock {\em Set-Valued Anal.}, 9(1-2):3--11, 2001.
\newblock Wellposedness in optimization and related topics (Gargnano, 1999).

\bibitem{Ant94}
A.~S. Antipin.
\newblock Minimization of convex functions on convex sets by means of
  differential equations.
\newblock {\em Differentsial Equations}, 30(9):1475--1486, 1652, 1994.

\bibitem{APR14}
H{\'e}dy Attouch, Juan Peypouquet, and Patrick Redont.
\newblock A dynamical approach to an inertial forward-backward algorithm for
  convex minimization.
\newblock {\em SIAM J. Optim.}, 24(1):232--256, 2014.

\bibitem{BT09}
Amir Beck and Marc Teboulle.
\newblock A fast iterative shrinkage-thresholding algorithm for linear inverse
  problems.
\newblock {\em SIAM J. Imaging Sci.}, 2(1):183--202, 2009.

\bibitem{BT89}
Dimitri~P Bertsekas and John~N Tsitsiklis.
\newblock {\em Parallel and distributed computation}.
\newblock Prentice Hall Inc., 1989.

\bibitem{BC14a}
Radu~Ioan Bot and Ern{\"o}~Robert Csetnek.
\newblock An inertial alternating direction method of multipliers.
\newblock {\em arXiv preprint arXiv:1404.4582}, 2014.

\bibitem{BC14d}
Radu~Ioan Bot and Ern{\"o}~Robert Csetnek.
\newblock An inertial tseng's type proximal algorithm for nonsmooth and
  nonconvex optimization problems.
\newblock {\em arXiv preprint arXiv:1406.0724}, 2014.

\bibitem{BCH14}
Radu~Ioan Bot, Ern{\"o}~Robert Csetnek, and Christopher Hendrich.
\newblock Inertial douglas-rachford splitting for monotone inclusion problems.
\newblock {\em arXiv preprint arXiv:1403.3330}, 2014.

\bibitem{Boyd+11}
Stephen Boyd, Neal Parikh, Eric Chu, Borja Peleato, and Jonathan Eckstein.
\newblock Distributed optimization and statistical learning via the alternating
  direction method of multipliers.
\newblock {\em Foundations and Trends{\textregistered} in Machine Learning},
  3(1):1--122, 2011.

\bibitem{Bru75}
Ronald~E. Bruck, Jr.
\newblock Asymptotic convergence of nonlinear contraction semigroups in
  {H}ilbert space.
\newblock {\em J. Funct. Anal.}, 18:15--26, 1975.

\bibitem{CRT06}
Emmanuel~J Cand{\`e}s, Justin Romberg, and Terence Tao.
\newblock Robust uncertainty principles: Exact signal reconstruction from
  highly incomplete frequency information.
\newblock {\em Information Theory, IEEE Transactions on}, 52(2):489--509, 2006.

\bibitem{Cha04}
Antonin Chambolle.
\newblock An algorithm for total variation minimization and applications.
\newblock {\em Journal of Mathematical imaging and vision}, 20(1-2):89--97,
  2004.

\bibitem{CP11}
Antonin Chambolle and Thomas Pock.
\newblock A first-order primal-dual algorithm for convex problems with
  applications to imaging.
\newblock {\em Journal of Mathematical Imaging and Vision}, 40(1):120--145,
  2011.

\bibitem{CMY14a}
Caihua Chen, Shiqian Ma, and Junfeng Yang.
\newblock A general inertial proximal point method for mixed variational
  inequality problem.
\newblock {\em arXiv preprint arXiv:1407.8238}, 2014.

\bibitem{Eck11}
J.~Eckstein.
\newblock Augmented lagrangian and alternating directions methods for convex
  optimization: a tutorial and some illustrative computational results.
\newblock {\em manuscript}, 2011.

\bibitem{Eck89}
Jonathan Eckstein.
\newblock {\em Splitting methods for monotone operators with applications to
  parallel optimization}.
\newblock PhD thesis, Massachusetts Institute of Technology, 1989.

\bibitem{EB92}
Jonathan Eckstein and Dimitri~P. Bertsekas.
\newblock On the {D}ouglas-{R}achford splitting method and the proximal point
  algorithm for maximal monotone operators.
\newblock {\em Math. Programming}, 55(3, Ser. A):293--318, 1992.

\bibitem{EZC10}
Ernie Esser, Xiaoqun Zhang, and Tony~F. Chan.
\newblock A general framework for a class of first order primal-dual algorithms
  for convex optimization in imaging science.
\newblock {\em SIAM J. Imaging Sci.}, 3(4):1015--1046, 2010.

\bibitem{GM76}
Daniel Gabay and Bertrand Mercier.
\newblock A dual algorithm for the solution of nonlinear variational problems
  via finite element approximation.
\newblock {\em Computers and Mathematics with Applications}, 2(1):17--40, 1976.

\bibitem{GM75}
R.~Glowinski and A.~Marrocco.
\newblock Sur l'approximation, par \'el\'ements finis d'ordre un, et la
  r\'esolution, par p\'enalisation-dualit\'e, d'une classe de probl\`emes de
  {D}irichlet non lin\'eaires.
\newblock {\em R.A.I.R.O., R2}, 9(R-2):41--76, 1975.

\bibitem{GO09}
Tom Goldstein and Stanley Osher.
\newblock The split bregman method for l1-regularized problems.
\newblock {\em SIAM Journal on Imaging Sciences}, 2(2):323--343, 2009.

\bibitem{Gul92}
Osman G{\"u}ler.
\newblock New proximal point algorithms for convex minimization.
\newblock {\em SIAM J. Optim.}, 2(4):649--664, 1992.

\bibitem{HYZ08}
Elaine~T Hale, Wotao Yin, and Yin Zhang.
\newblock Fixed-point continuation for $\backslash$ell\_1-minimization:
  Methodology and convergence.
\newblock {\em SIAM Journal on Optimization}, 19(3):1107--1130, 2008.

\bibitem{HY12}
Bingsheng He and Xiaoming Yuan.
\newblock Convergence analysis of primal-dual algorithms for a saddle-point
  problem: From contraction perspective.
\newblock {\em SIAM Journal on Imaging Sciences}, 5(1):119--149, 2012.

\bibitem{HY12b}
Bingsheng He and Xiaoming Yuan.
\newblock On non-ergodic convergence rate of douglas-rachford alternating
  direction method of multipliers.
\newblock Technical report, Tech. rep., Nanjing University, 2012.

\bibitem{HY12a}
Bingsheng He and Xiaoming Yuan.
\newblock On the {$O(1/n)$} convergence rate of the {D}ouglas-{R}achford
  alternating direction method.
\newblock {\em SIAM J. Numer. Anal.}, 50(2):700--709, 2012.

\bibitem{Hes69}
M.~R. Hestenes.
\newblock Multiplier and gradient methods.
\newblock {\em J. Optimization Theory Appl.}, 4:303--320, 1969.

\bibitem{MGC11}
Shiqian Ma, Donald Goldfarb, and Lifeng Chen.
\newblock Fixed point and {B}regman iterative methods for matrix rank
  minimization.
\newblock {\em Math. Program.}, 128(1-2, Ser. A):321--353, 2011.

\bibitem{MM07}
Paul-Emile Maing{\'e} and Abdellatif Moudafi.
\newblock A proximal method for maximal monotone operators via discretization
  of a first order dissipative dynamical system.
\newblock {\em J. Convex Anal.}, 14(4):869--878, 2007.

\bibitem{Mar70}
B.~Martinet.
\newblock R\'egularisation d'in\'equations variationnelles par approximations
  successives.
\newblock {\em Rev. Fran\c caise Informat. Recherche Op\'erationnelle}, 4(Ser.
  R-3):154--158, 1970.

\bibitem{Mor65}
Jean-Jacques Moreau.
\newblock Proximit\'e et dualit\'e dans un espace hilbertien.
\newblock {\em Bull. Soc. Math. France}, 93:273--299, 1965.

\bibitem{ME03}
A.~Moudafi and E.~Elissabeth.
\newblock Approximate inertial proximal methods using the enlargement of
  maximal monotone operators.
\newblock {\em International Journal of Pure and Applied Mathemtics},
  5(3):283--299, 2003.

\bibitem{NW13b}
Deanna Needell and Rachel Ward.
\newblock Near-optimal compressed sensing guarantees for anisotropic and
  isotropic total variation minimization.
\newblock {\em IEEE TRANSACTIONS ON IMAGE PROCESSING}, 22(10):3941--3949, 2013.

\bibitem{NW13a}
Deanna Needell and Rachel Ward.
\newblock Stable image reconstruction using total variation minimization.
\newblock {\em SIAM Journal on Imaging Sciences}, 6(2):1035--1058, 2013.

\bibitem{Nes83}
Yu.~E. Nesterov.
\newblock A method for solving the convex programming problem with convergence
  rate {$O(1/k^{2})$}.
\newblock {\em Dokl. Akad. Nauk SSSR}, 269(3):543--547, 1983.

\bibitem{Nes07}
Yurii Nesterov.
\newblock Gradient methods for minimizing composite objective function, 2007.

\bibitem{NCT99}
Michael~K Ng, Raymond~H Chan, and Wun-Cheung Tang.
\newblock A fast algorithm for deblurring models with neumann boundary
  conditions.
\newblock {\em SIAM Journal on Scientific Computing}, 21(3):851--866, 1999.

\bibitem{OBP14}
P.~Ochs, T.~Brox, and T.~Pock.
\newblock ipiasco: Inertial proximal algorithm for strongly convex
  optimization.
\newblock {\em manuscript}, 2014.

\bibitem{OCBP14}
P.~Ochs, Y.~Chen, T.~Brox, and T.~Pock.
\newblock ipiano: Inertial proximal algorithm for non-convex optimization.
\newblock {\em manuscript}, 2014.

\bibitem{Pol64}
B.~T. Poljak.
\newblock Some methods of speeding up the convergence of iterative methods.
\newblock {\em \v Z. Vy\v cisl. Mat. i Mat. Fiz.}, 4:791--803, 1964.

\bibitem{Pow69}
M.~J.~D. Powell.
\newblock A method for nonlinear constraints in minimization problems.
\newblock In {\em Optimization ({S}ympos., {U}niv. {K}eele, {K}eele, 1968)},
  pages 283--298. Academic Press, London, 1969.

\bibitem{Rock70}
R.~T. Rockafellar.
\newblock {\em Convex analysis}.
\newblock Princeton Mathematical Series, No. 28. Princeton University Press.

\bibitem{Roc76b}
R.~T. Rockafellar.
\newblock Augmented {L}agrangians and applications of the proximal point
  algorithm in convex programming.
\newblock {\em Math. Oper. Res.}, 1(2):97--116, 1976.

\bibitem{Roc76a}
R.~T. Rockafellar.
\newblock Monotone operators and the proximal point algorithm.
\newblock {\em SIAM J. Control Optimization}, 14(5):877--898, 1976.

\bibitem{ROF92}
Leonid~I Rudin, Stanley Osher, and Emad Fatemi.
\newblock Nonlinear total variation based noise removal algorithms.
\newblock {\em Physica D: Nonlinear Phenomena}, 60(1):259--268, 1992.

\bibitem{ST14}
Ron Shefi and Marc Teboulle.
\newblock Rate of convergence analysis of decomposition methods based on the
  proximal method of multipliers for convex minimization.
\newblock {\em SIAM J. Optim.}, 24(1):269--297, 2014.

\bibitem{WY12}
Xiangfeng Wang and Xiaoming Yuan.
\newblock The linearized alternating direction method of multipliers for
  dantzig selector.
\newblock {\em SIAM Journal on Scientific Computing}, 34(5):2792--2811, 2012.

\bibitem{WYYZ08}
Yilun Wang, Junfeng Yang, Wotao Yin, and Yin Zhang.
\newblock A new alternating minimization algorithm for total variation image
  reconstruction.
\newblock {\em SIAM Journal on Imaging Sciences}, 1(3):248--272, 2008.

\bibitem{YY13}
Junfeng Yang and Xiaoming Yuan.
\newblock Linearized augmented lagrangian and alternating direction methods for
  nuclear norm minimization.
\newblock {\em Mathematics of Computation}, 82(2):301--329, 2013.

\bibitem{YZ11}
Junfeng Yang and Yin Zhang.
\newblock Alternating direction algorithms for {$\ell_1$}-problems in
  compressive sensing.
\newblock {\em SIAM J. Sci. Comput.}, 33(1):250--278, 2011.

\bibitem{YZY10}
Junfeng Yang, Yin Zhang, and Wotao Yin.
\newblock A fast alternating direction method for tvl1-l2 signal reconstruction
  from partial fourier data.
\newblock {\em Selected Topics in Signal Processing, IEEE Journal of},
  4(2):288--297, 2010.

\end{thebibliography}
\def\cprime{$'$}

\end{document}